\documentclass[twoside, leqno]{article}
\usepackage{pifont}

\usepackage{amssymb}
\usepackage{amsmath}
\usepackage{amsthm}

\usepackage{mathrsfs}

\allowdisplaybreaks

\def\dsum{\displaystyle\sum}
\def\dint{\displaystyle\int}

\def\dsup{\displaystyle\sup}
\def\dinf{\displaystyle\inf}

\def\r{\right}
\def\lf{\left}
\def\pat{\partial}
\def\ls{\lesssim}

\def\fz{\infty}
\def\fz{\infty}
\def\az{\alpha}
\def\supp{{\mathop\mathrm{\,supp\,}}}

\def\loc{{\mathop\mathrm{\,loc\,}}}
\def\bmo{{\mathop\mathrm{BMO}}}
\def\dist{{\mathop\mathrm{dist}}}
\def\mol{{\mathop\mathrm{mol}}}

\def\lz{\lambda}
\def\dz{\delta}
\def\bdz{\Delta}
\def\ez{\epsilon}

\def\bz{\beta}
\def\fai{\varphi}
\def\gz{{\gamma}}

\def\tz{\theta}
\def\sz{\sigma}
\def\wz{\widetilde}

\def\ls{\lesssim}

\def\laz{\langle}
\def\raz{\rangle}

\def\boz{\Omega}

\def\gfz{\genfrac{}{}{0pt}{}}
\def\pat{\partial}

\def\rr{{\mathbb R}}
\def\rn{{{\rr}^n}}
\def\+rn{{{\rr}^{n+1}_+}}
\def\zz{{\mathbb Z}}
\def\nn{{\mathbb N}}

\def\zz{{\mathbb Z}}
\def\nn{{\mathbb N}}
\def\cc{{\mathbb C}}

\def\cl{{\mathcal L}}

\def\cm{{\mathcal M}}

\def\hs{\hspace{0.3cm}}

\newtheorem{theorem}{Theorem}[section]
\newtheorem{lemma}[theorem]{Lemma}
\newtheorem{corollary}[theorem]{Corollary}
\newtheorem{proposition}[theorem]{Proposition}
\theoremstyle{definition}
\newtheorem{remark}[theorem]{Remark}
\newtheorem{definition}[theorem]{Definition}
\numberwithin{equation}{section}

\begin{document}

\arraycolsep=1pt

\title{\Large\bf Weak Hardy Spaces $WH_L^p({\mathbb R}^n)$ Associated to
Operators Satisfying $k$-Davies-Gaffney Estimates\footnotetext
{\hspace{-0.35cm} 2010 {\it Mathematics Subject Classification}.
Primary 42B35; Secondary 42B30, 42B20, 42B37, 47B06, 47A60, 35J30, 35J10.\endgraf
{\it Key words and phrases}. weak Hardy space, dual, atom, molecule, BMO, real interpolation,
Riesz transform, $k$-Davies-Gaffney estimate, functional calculus,
higher order elliptic operator, Schr\"odinger type operator.
\endgraf
Der-Chen Chang is partially supported by an
NSF grant DMS-1203845 and a Hong Kong RGC competitive earmarked
research grant $\#$601410, Huoxiong Wu is supported by National Natural Science Foundation
of China (Grant Nos. 11371295 and 11471041), and Dachun Yang is supported by the National
Natural Science Foundation  of China (Grant Nos. 11171027 and 11361020),
the Specialized Research Fund for the Doctoral Program of Higher Education
of China (Grant No. 20120003110003) and the Fundamental Research
Funds for Central Universities of China (Grant Nos. 2012LYB26, 2012CXQT09 and 2014KJJCA10).}}
\author{Jun Cao, Der-Chen Chang, Huoxiong Wu
and Dachun Yang\,\footnote{Corresponding author}}
\date{}
\maketitle

\vspace{-0.6cm}

\begin{center}
\begin{minipage}{13cm}
{\small {\bf Abstract}\quad Let $L$ be a one-to-one operator of type
$\omega$ having a bounded $H_\infty$ functional calculus and
satisfying the $k$-Davies-Gaffney estimates with $k\in{\mathbb N}$.
In this article, the authors introduce the
weak Hardy space $WH_L^p(\mathbb{R}^n)$  associated to $L$ for $p\in (0,\,1]$ via the
non-tangential square function $S_L$ and establish a weak
molecular characterization of $WH_L^p(\mathbb{R}^n)$.
A typical example of such operators is
the $2k$-order divergence form homogeneous elliptic operator
$L:=(-1)^k\sum_{|\alpha|=k=|\beta|}\partial^\beta(a_{\alpha,\beta}\partial^\alpha)$,
where $\{a_{\alpha,\beta}\}_{|\alpha|=k=|\beta|}$ are
complex bounded measurable functions. As applications, for
$p\in(\frac{n}{n+k},\,1]$, the authors prove
that the associated Riesz transform $\nabla^k L^{-1/2}$ is
bounded from $WH^p_{L}(\mathbb{R}^n)$ to the classical weak Hardy
space $WH^p(\mathbb{R}^n)$ and, for all $0<p<r\le1$ and
$\alpha=n(\frac{1}{p}-\frac{1}{r})$, the fractional power
$L^{-\frac{\alpha}{2k}}$ is bounded from
$WH_{L}^p(\mathbb{R}^n)$ to $WH_{L}^r(\mathbb{R}^n)$.
Also, the authors establish an interpolation theorem of $WH_{L}^p(\mathbb{R}^n)$
by showing that $L^2(\mathbb{R}^n)\cap WH_{L}^p(\mathbb{R}^n)$ for all $p\in(0,\,1]$
are the intermediate spaces
in the real method of interpolation between the Hardy space
$L^2(\mathbb{R}^n)\cap H_L^p(\mathbb{R}^n)$  for different $p\in(0,\,1]$.
In particular, if $L$ is a
nonnegative self-adjoint operator in $L^2({\mathbb R}^n)$ satisfying
the Davies-Gaffney estimates, the authors further establish the weak
atomic characterization of $WH_L^p(\mathbb{R}^n)$. Furthermore, the authors find the dual space of
$WH_L^p(\mathbb{R}^n)$ for $p\in(0,\,1]$, which can be defined via mean oscillations
based on some subtle coverings of bounded open sets and,
even when $L:=-\Delta$, are also previously unknown.

}
\end{minipage}
\end{center}

\vspace{0.2cm}

\section{Introduction}\label{s1}

\hskip\parindent It is well known that Stein and Weiss \cite{sw60}
originally inaugurated the study of real Hardy spaces $H^p(\rn)$
with $p\in(0,\,1]$ on the Euclidean space $\rn$.
Later, a real-variable theory of
$H^p(\rn)$ for $p\in(0,\,1]$ was systematically developed by
Fefferman and Stein in \cite{fs72}. Since then, the real-variable theory of Hardy spaces
$H^p(\rn)$ has found many important applications in various fields of analysis
and partial differential equations; see, for example, \cite{co74,clms93,
cw77,gr08, la78,mu94,se94, st93,tw80}. It is now known that
$H^p(\rn)$ is a good substitute of the Lebesgue space $L^p(\rn)$
with $p\in(0,\,1]$ when studying the boundedness of operators;
for example, when $p\in(0,\,1]$, the \emph{Riesz
transform} $\nabla(-\bdz)^{-1/2}$ is not bounded on $L^p(\rr^n)$,
but bounded on $H^p(\rr^n)$, where $\Delta$ is the
\emph{Laplace operator} $\sum_{i=1}^n\frac {\partial^2}{\partial
x_i^2}$ and $\nabla$ is the \emph{gradient operator}
$(\frac\partial{\partial x_1}, \ldots, \frac
\partial{\partial x_n})$ on $\rn$. Moreover, when considering some weak
type inequalities for some of the most important operators from
harmonic analysis and partial differential equations,
we are led to the more general \emph{weak Hardy
space} $WH^p(\rn)$ (see, for example, \cite{fso86, li91, gr92, a94, gr95,
qy00,at07,mmy10}). It is well known that the weak
Hardy space $WH^p(\rn)$ is a suitable substitute
of both the weak Lebesgue space $WL^p(\rn)$ and the Hardy space
$H^p(\rn)$ when studying the boundedness of operators
in the critical case. For example, let $\dz\in(0,\,1]$, $T$ be a
\emph{$\dz$-Calder\'on-Zygmund operator} and $T^*(1)=0$, where $T^*$
denotes the \emph{adjoint operator} of $T$. It is known that $T$ is
bounded on $H^p(\rn)$ for all $p\in (\frac{n}{n+\dz},\,1]$ and not bounded on
$H^{\frac{n}{n+\dz}}(\rn)$, but, instead of this, $T$ is
bounded from $H^{\frac{n}{n+\dz}}(\rn)$ to
$WH^{\frac{n}{n+\dz}}(\rn)$ (see \cite{li91,a94}). Recall that the Riesz
transform $\nabla (-\bdz)^{-1/2}$ is a $1$-Calder\'on-Zygmund
operator with convolution kernel, which is smooth on $\rn\times\rn$
except on the diagonal points $$\{(x,\,y)\in \rn\times\rn:\ x=y\}.$$
For more related history and properties about $WH^p(\rn)$, we refer to
\cite{frs74,fso86,li91,a94,lu95,qy00, at07} and the references cited therein.
We should point out that Fefferman, Rivi\`ere and Sagher \cite{frs74}
proved that the weak Hardy space $WH^p(\rn)$ naturally occurs as the intermediate spaces
in the real method of interpolation between the Hardy space
$H^p(\rn)$. It is easy to see that the classical Hardy
spaces $H^p(\rn)$ and the weak Hardy spaces $WH^p(\rn)$ are
essentially related to the Laplace operator $\bdz$.

In recent years, the study of Hardy spaces and their generalizations associated to
differential operators attracts a lot of attentions; see, for example,
\cite{ao11, adm05,  amr08,  cly10, cy,  cds99,  cks92, cks93, dxy07, dy05-1, dy05-2, dz99, dz00,
hlmmy,hm09,hm09c,hmm, jy11, jy10, yy11} and their references. In
particular, Auscher et al. \cite{adm05} first
introduced the Hardy space $H_L^1(\rn)$ associated to $L$, where the
\emph{heat kernel generated by $L$ satisfies a pointwise Poisson
type upper bound}. Later, Duong and Yan \cite{dy05-1,dy05-2}
introduced the dual space $\bmo_L(\rn)$ and showed that the dual
space of $H_L^1(\rn)$ is $\bmo_{L^*}(\rn)$, where $L^*$
denotes the {\it adjoint operator} of $L$ in $L^2(\rn)$. Yan \cite{y08}
further introduced the Hardy space $H_L^p(\rn)$ for some
$p\in(0,\,1]$ but near to $1$ and generalized these results to
$H_L^p(\rn)$ and their dual spaces. A real-variable theory of Orlicz-Hardy spaces
and their dual spaces associated to $L$ was also developed in
\cite{jyz09,jya}.

Recently, the (Orlicz-)Hardy space associated to
a \emph{one-to-one operator of type $\omega$ satisfying the
$k$-Davies-Gaffney estimates and having a bounded $H_{\fz}$
functional calculus} was introduced in \cite{cy,dl13,ccyy,ddy12}. A typical
example of such operators is the following \emph{$2k$-order divergence
form homogeneous elliptic operator}
\begin{equation}\label{1.1}
L:=(-1)^k\sum_{|\az|=k=|\bz|}\pat^\bz(a_{\az,\bz}\pat^\az),
\end{equation}
interpreted in the usual weak sense via a
sesquilinear form, with complex bounded measurable coefficients
$\{a_{\az,\,\bz}\}_{|\az|=k=|\bz|}$ satisfying
the \emph{elliptic condition}, namely, there exist constants
$0<\lz\le\Lambda<\fz$ such that, for all $\az$, $\bz$ with
$|\az|=k=|\bz|$,
$\|a_{\az,\bz}\|_{L^\fz(\rn)}\le\Lambda$ and, for all $f\in W^{k,2}(\rn)$,
$\Re (L_1f,\,f)\ge
\lz\,\|\nabla^kf\|_{L^2(\rn)}^2.$
Here and hereafter, $\Re z$ denotes the {\it
real part} of $z$ for all $z\in\cc$.

Notice that, when $k=1$, $H_{L}^p(\rn)$ is the Hardy space associated to the
\emph{second-order divergence form elliptic operator on $\rn$ with complex
bounded measurable coefficients}, which was introduced by Hofmann
and Mayboroda \cite{hm09,hm09c}, Hofmann et al. \cite{hmm}, and Jiang
and Yang \cite{jy10}. It is known that the associated Riesz
transform $\nabla^k {L^{-1/2}}$ is bounded
from $H^p_{L}(\rn)$ to the classical Hardy space $H^p(\rn)$ for
all $p\in(\frac{n}{n+k},\,1]$ (see \cite{cy}). Unlike the classical
case, in this case, $\nabla^k {L^{-1/2}}$ may even not have a
smooth convolution kernel. Thus, the boundedness of $\nabla^k
{L^{-1/2}}$ can not be extended to the full range of
$p\in(0,\,\fz)$ as before. However, when considering the endpoint
boundedness of the associated Riesz transforms, it is found that the
weak Hardy space is useful. For example, it was proved in
\cite{lyy11} that $\nabla^kL^{-1/2}$ is bounded from
$H^{n/(n+k)}_{L}(\rn)$ to the weak Hardy space
$WH^{n/(n+k)}(\rn)$, which may not be bounded on $H_L^{n/(n+k)}(\rn)$.

Motivated by the above results, in this article, we wish to
develop a real-variable theory of weak Hardy spaces associated
to a class of differential operators and study their applications.
More precisely, we always assume that \emph{$L$ is a one-to-one operator of type $\omega$
having a bounded $H_\infty$ functional calculus and satisfying the
$k$-Davies-Gaffney estimates}. For $p\in (0,\,1]$, we introduce
the weak Hardy space $WH_L^p(\mathbb{R}^n)$ associated to $L$ via
the non-tangential square function $S_L$ and establish its weak
molecular characterization. In particular,
\emph{if $L$ is a nonnegative self-adjoint operator in $L^2({\mathbb R}^n)$ satisfying
the Davies-Gaffney estimates}, we further establish the weak atomic
decomposition of $WH_L^p(\mathbb{R}^n)$. By their atomic characterizations,
we easily see that $WH_{-\Delta}^p(\mathbb{R}^n)$ and
the \emph{closure} of $WH^p(\rn)\cap L^2(\rn)$ on the quasi-norm
$\|\cdot\|_{WH^p(\rn)}$ coincide with equivalent quasi-norms.
Let $L$ be the $2k$-order divergence
form homogeneous elliptic operator as in \eqref{1.1}. As applications,
we prove that, for all $p\in({n}/{(n+k)},\,1]$, the associated Riesz
transform $\nabla^k L^{-1/2}$ is bounded from
$WH^p_{L}(\mathbb{R}^n)$ to the classical weak Hardy space
$WH^p(\mathbb{R}^n)$; furthermore,  for all $0<p<r\le1$  and
$\alpha=n({1}/{p}-{1}/{r})$, the fractional power
$L^{-{\alpha}/{(2k)}}$ is bounded from $WH_{L}^p(\rn)$ to
$WH_{L}^r(\rn)$. We also establish a real interpolation theorem
on $WH_{L}^p(\rn)$ by showing that $L^2(\rn)\cap WH_{L}^p(\rn)$
for all $p\in(0,\,1]$ are the the
intermediate spaces in the real method of interpolation between the
Hardy spaces $L^2(\rn)\cap H_L^p(\rn)$ for different $p\in(0,\,1]$.
Moreover, if $L$ is nonnegative self-adjoint and satisfies the Davies-Gaffney estimates,
then, for $p\in(0,\,1]$, we give out the
dual space of $WH_L^p(\mathbb{R}^n)$,
which is defined via mean oscillations of distributions
based on some subtle coverings of bounded open sets,
and prove that the elements in $W\Lambda_{L}^\az(\rn)$
can be viewed as a weak type Carleson measure of order $\az$.
We point out that, even when
$L:=-\Delta,$
the dual spaces of $WH_L^p(\mathbb{R}^n)$ are also previously unknown,
since the seminal article \cite{fso86} of Fefferman and Soria
on $WH^1(\rn)$ was published in 1986. Our aforementioned result on dual spaces of
$WH_L^p(\mathbb{R}^n)$ may give some light on this problem. In short,
the results of this article
round out the picture on weak Hardy spaces associated to
operators satisfying $k$-Davies-Gaffney estimates.  As in the aforementioned
articles on the theory of Hardy spaces associated with operators, the achievement
of all results in this article stems from subtle atomic decompositions
of weak tent spaces introduced in this article.
To the best of our knowledge, all results obtained in this article are new
even when $L$ is the Laplace operator.

This article is organized as follows.

In Section \ref{s2}, we first present some assumptions on the
operator $L$ used throughout the whole article (see Assumptions
$(\mathcal{L})_1$, $(\mathcal{L})_2$, $(\mathcal{L})_3$ and $(\mathcal{L})_4$
below) and
recall some basic facts concerning the $k$-Davies-Gaffney estimates
(see Lemmas \ref{l2.3}, \ref{l2.4} and \ref{l2.5} below) in Subsection
\ref{s2.1}. In Subsection \ref{s2.2}, we introduce the weak tent
space and establish its weak atomic decomposition (see Theorem
\ref{t2.6} below). Later, via the Whitney decomposition lemma,
we obtain another  weak atomic decomposition of $WT^p(\rr^{n+1}_+)$ (see
Theorem \ref{t2.12} below), which plays an important role
in the dual theory of our weak Hardy spaces.

Finally, in Section \ref{s2.3}, after recalling some
necessary results on the Hardy space $H_L^p(\rn)$ associated to $L$,
we introduce the weak Hardy space $WH_L^p(\rn)$ associated to $L$
(see Definition \ref{d2.13} below) and establish its weak molecular
characterization (see Theorem \ref{t2.21} below). As applications, if $L$ is
the $2k$-order divergence form homogeneous elliptic operator as in \eqref{1.1},
we prove the boundedness of
the associated Riesz transform $\nabla^k L^{-1/2}$ and fractional power
$L^{-\az/(2k)}$ on $WH_{L}^p(\rn)$ (see Theorems \ref{t2.23} and \ref{t2.24} below).
Moreover, when $L$ is a nonnegative self-adjoint operator in $L^2(\rn)$ satisfying the
Davies-Gaffney estimates, we obtain its weak atomic characterization
(see Theorem \ref{t2.16} below). Recall that, in \cite{fso86,li91}, for all $p\in(0,\,1]$,
a weak  atomic decomposition of the classical weak Hardy space $WH^p(\rn)$
is obtained. However, the ``atoms" appeared in the weak atomic characterization
of $WH^p(\rn)$ in \cite{fso86,li91} are essentially more in the spirit of
the classical ``$L^\fz(\rn)$-atoms", while the ``atoms" appeared in our weak atomic characterization
of $WH_L^p(\rn)$ are just $H_L^p(\rn)$-atoms associated to $L$
from \cite{hlmmy} when $p=1$ and from \cite{jy11} when $p\in (0,1]$
(see Theorem \ref{t2.16} below).

In order to establish the weak atomic
decomposition of weak tent spaces in Theorem \ref{t2.6} below,
we want to point out that we borrow
some ideas from the proof of \cite[Proposition 2]{cms85}, with some necessary adjustments by
changing the formation of the norm from the original strong version
to the present weak version. We also remark that this weak atomic
characterization still holds true under some small modifications of the
level set of the $\mathcal{A}$-functional (see \eqref{2.1} and Remark \ref{r2.10} below). An innovation of Theorem \ref{t2.6} is to establish an explicit relation between the supports of $T^p(\rr^{n+1}_+)$-atoms and the corresponding coefficients,
which plays a key role in establishing the weak atomic/molecular characterizations
of weak Hardy spaces associated to $L$ (see Theorems \ref{t2.6}(ii), \ref{t2.16}
and \ref{t2.21} below). Indeed, the  proof of Theorem \ref{t2.16} strongly depends on Theorem \ref{t2.6} and a
\emph{superposition principle on weak type estimate} from Stein et al. \cite{stw81} (see also Lemma
\ref{l2.18} below).  We point out that, without Theorem \ref{t2.6}(ii), Theorem
\ref{t2.16} seems impossible (see \eqref{2.10} and \eqref{2.11} below).
The proof of Theorem \ref{t2.21} is similar to that of Theorem \ref{t2.16},
but needs more careful off-diagonal estimates because of the lack of the
support condition of molecules.

In Section \ref{s4}, we establish an interpolation theorem of $WH_L^p(\rn)$
by showing that $L^2(\rn)\cap WH_{L}^p(\rn)$ for all $p\in(0,\,1]$ are the
intermediate spaces in the real method of interpolation between the
Hardy spaces $L^2(\rn)\cap H_L^p(\rn)$ for different $p\in(0,\,1]$
(see Theorem \ref{t4.5} below). Unlike in the classical case in \cite{frs74},
we prove Theorem \ref{t4.5} by
using a real interpolation result on the tent space $T^p(\rr^{n+1}_+)$ from
\cite{ccy14} and a result on the  interpolation of intersections from
Krugljak et al. \cite{kmp99}.

Section \ref{s3} is devoted to the dual theory of $WH_L^p(\rn)$.  Let $L$
be nonnegative self-adjoint and satisfy the Davies-Gaffney estimates.
We first introduce the notion of the \emph{weak Lipschitz space}
$W\Lambda_{L,{\mathcal{N}}}^\az(\rn)$
via the mean oscillation over bounded open sets,
then we prove that the elements in
$W\Lambda_{L, {\mathcal{N}}}^\az(\rn)$
can be viewed as some weak Carleson measures of order $\az$ (see Proposition
\ref{p3.7} below) and prove that the dual space of $WH_L^p(\rn)$
is $W\Lambda_{L, {\mathcal{N}}}^{n({1}/{p}-1)}(\rn)$ (see Theorem \ref{t3.6} below).

Recall that the dual space of the classical weak
Hardy space $WH^1(\rn)$ was first considered by Fefferman and Soria
\cite{fso86}. More precisely,  for any bounded open set $\boz\subset \rn$ and function $\fai$
on $\rn$, the \emph{mean oscillation} $\mathcal{O}(\fai,\,\boz)$ of $\fai$
\emph{over} $\boz$ was defined by Fefferman and Soria in \cite{fso86} as
$$\mathcal{O}(\fai,\,\boz):=\dsup\frac{1}{|\boz|}\dsum_{k}\dint_{Q_k}\lf|\fai(x)
-\fai_{Q_k}\r|\,dx,$$
where $\fai_{Q}:=\frac{1}{|Q|}\int_{Q}\fai(x)\,dx$ for any cube $Q$
and the supremum is taken over all collections $\{Q_k\}_k$ of subcubes of $\boz$
with \emph{bounded $C(n)$-overlap} (which means that
there exists a positive constant $C(n)$ such that $\sum_{k} \chi_{Q_k}\le C(n)$).
Let  $\omega(\dz):=\sup_{|\boz|=\dz}\mathcal{O}(\fai,\,\boz)$,
$$L_0^1(\rn):=\lf\{f\in L^1(\rn):\ \int_{\rn}f(x)\,dx=0\r\}$$ and
$\overline{L_0^1(\rn)}$ be the \emph{closure} of
$L_0^1(\rn)$ in the norm of the weak Hardy space $WH^1(\rn)$.
In \cite{fso86}, Fefferman and Soria proved that the dual of $\overline{L_0^1(\rn)}$
is the \emph{set} of all the functions
$\fai$ satisfying
$$\|\fai\|_*:=\dint_0^\fz\frac{\omega(\dz)}{\dz}\,d\dz<\fz.$$
In the present article, we show, in Theorem \ref{t3.6} below,
that the dual space of $WH_L^p(\rn)$ for all $p\in(0,\,1]$
is $W\Lambda^{n(1/p-1)}_{L, {\mathcal{N}}}(\rn)$,
which is defined by means of a similar integral of the mean oscillation
based on some smart coverings of bounded open sets (see Definitions \ref{d3.1}
and \ref{d3.2} below). Here the integral mean $\fai_{Q}$ is replaced
by some approximation of
identity, and  the collections of subcubes of an open set with
bounded $C(n)$-overlap by another new class of sets (see Definition \ref{d3.1}).
In particular, when $L=-\bdz$, $W\Lambda_{-\bdz, \overrightarrow{\mathcal{N}}}^{n({1}/{p}-1)}(\rn)$
is the dual space of the space $\overline{WH^p(\rn)\cap L^2(\rn)}^{\|\cdot\|_{WH^p(\rn)}}$
for $p\in(0,\,1]$, which seems also new. Here, the \emph{space}
$\overline{WH^p(\rn)\cap L^2(\rn)}^{\|\cdot\|_{WH^p(\rn)}}$
denotes the \emph{closure} of $WH^p(\rn)\cap L^2(\rn)$ on the quasi-norm
$\|\cdot\|_{WH^p(\rn)}$.

The proof of Theorem \ref{t3.6} strongly depends on a Carder\'on reproducing
formula obtained in \cite{hm09}, a subtle weak atomic decomposition
of the weak tent space (see Theorem \ref{t2.12}), and a
resolvent characterization of $WH_L^p(\rn)$ (see Proposition
\ref{p3.5} below) and Proposition \ref{p3.7} below.

Recall that a key ingredient to prove
the duality between Hardy spaces and Lipschitz spaces is to
represent the Lipschitz norm by means of a dual norm expression of
some Hilbert spaces. It is known that, in the case of the classical
``strong" Lipschitz space $\Lambda_L^\az(\rn)$, this Hilbert space
can be chosen to be $L^2(B)$, where $B$ is some ball (see the proof of
\cite[Theorem 3.51]{hmm}).  Observe also that the mean oscillation
appearing in the norm of  the ``strong"
Lipschitz space $\Lambda_L^\az(\rn)$ has the form
$$\lf\{\frac{1}{|B|}\dint_{B}\lf|\lf(I-e^{-r_B^2L}\r)^Mf(x)
\r|^2\,dx\r\}^{1/2},$$
which involves only one ball and hence the
off-diagonal estimates can be applied directly.
However, in the weak case, the  mean oscillation involves a general bounded
open set (see \eqref{3.5} below).
Therefore, we can not apply
off-diagonal estimates directly.
To overcome this difficulty,
we first introduce, in Definition \ref{d3.1}, subtle coverings of
bounded open sets, which stem from the proof of the weak atomic
decomposition for weak tent spaces in Theorem \ref{t2.12},
obtained via the Whitney-type decompositions on level sets for
$\mathcal{A}$-functionals in \eqref{2.1}. More precisely,
we first find a sequence of balls which cover the considered open set
via the Whitney-type decomposition. Then we construct
the annuli sets based on a sequence of balls and consider the off-diagonal estimates
on these annuli. Since the radius of the balls in the sequence are different,
the off-diagonal estimates on these annuli are more difficult than those on a single ball.

As usual, we make some conventions on the notation. Throughout the
whole article, we always let $\nn:=\{1,2,\ldots\}$, $\zz_+:=
\nn\cup\{0\}$ and $\rr^{n+1}_+:=\{(x,\,t):\
x\in\rn,\,t\in(0,\,\fz)\}$. We use $C$ to denote a {\it positive
constant}, independent of the main parameters involved but whose
value may differ from line to line. {\it Constants with subscripts},
such as $C_0$, $M_0$ and $\az_0$, do not change in different
occurrences. If $f\le Cg$, we write $f\ls g$ and, if $f\ls g\ls f$,
we then write $f\sim g$. For all $x\in\rr^n$, $r\in(0,\fz)$ and
$\az\in(0,\,\fz)$, let $B(x,r):=\{y\in\rr^n:|x-y|<r\}$, $\az
B(x,r):= B(x,\az r)$, $S_0(B):=B$,
$S_i(B):=2^iB\setminus(2^{i-1}B)$ and
$\wz{S}_i(B):=2^{i+1}B\setminus(2^{i-2}B)$  for all $i\in{\zz_+}$,
where, when $i<0$, $2^iB:=\emptyset$. Also, for any set $E\subset \rn$,
we use $E^\complement$ to denote the \emph{set} $\rn\setminus E$ and
$\chi_E$ its {\it characteristic function}. For any index
$q\in[1,\fz]$, we denote by $q'$ its \emph{conjugate index}, namely,
$1/q+1/q'=1$. Let $\mathcal{S}(\rn)$ be the \emph{space of Schwartz
functions} on $\rn$ and $\mathcal{S}'(\rn)$ its \emph{dual space}.

\section{The weak Hardy space $WH_L^p(\rn)$}\label{s2}

\hskip\parindent The main purpose of this section is to introduce
the weak Hardy space $WH_L^p(\rn)$ and establish its weak
atomic and molecular characterizations. As applications of this weak
molecular characterization, we obtain the boundedness of the
associated Riesz transform and fractional power on $WH_L^p(\rn)$.
In order to achieve this goal, we
need to describe our hypotheses on the operator $L$
throughout the whole article.

\subsection{Assumptions on $L$}\label{s2.1}
\hskip\parindent In this subsection, we first survey some known
results on the bounded $H_\fz$ functional calculus. Then, after
stating our assumptions on the operator $L$ throughout the whole
article, we recall some useful technical lemmas on the
$k$-Davies-Gaffney estimates.

For $\tz\in[0,\,\pi)$, the {\it open} and {\it closed sectors, $S_\tz^0$}
and {\it $S_\tz$, of angle $\tz$} in the complex plane $\cc$ are defined,
respectively, by setting $S_\tz^0:=\lf\{z\in\cc\setminus\{0\}:\
|\arg z|<\tz\r\}$ and $S_\tz:=\lf\{z\in\cc:\ |\arg z|\le\tz\r\}$.
Let $\omega\in[0,\,\pi)$. A closed operator $T$ on $L^2(\rn)$ is
said to be of {\it type} $\omega$, if

(i) the \emph{spectrum} of $T$, $\sz(T)$, is contained in
$S_\omega$;

(ii) for each $\tz\in (\omega,\,\pi)$, there exists a nonnegative
constant $C$ such that, for all $z\in\cc\setminus S_\tz$,
$$\lf\|(T-zI)^{-1}\r\|_{\cl(L^2(\rn))}\le
C|z|^{-1},$$
here and hereafter, for any normed linear space
${\mathcal H}$, $\|S\|_{\cl(\mathcal H)}$ denotes the {\it operator
norm} of the linear operator $S:\ {\mathcal H}\to {\mathcal H}$.

For $\mu\in\![0,\,\pi)$ and $\sz,\tau\in(0,\,\fz)$, let
$H(S_\mu^0)\!:=\!\lf\{f:\ f\ \text{is a holomorphic function on}\
S_\mu^0\r\},$
\begin{eqnarray*}
H_{\fz}(S_\mu^0):= \lf\{f\in H(S_\mu^0):\
\|f\|_{L^{\fz}(S^0_{\mu})}<\fz\r\}
\end{eqnarray*}
and
\begin{eqnarray*}
\Psi_{\sz,\tau}(S_\mu^0):=&&\lf\{f\in H(S_\mu^0):\ \text{there
exists a positive
constant}\ C \ \text{such that},\r.\\
&&\hspace{2.5cm}\text{for all}\ \xi\in S_\mu^0,\lf. |f(\xi)|\le C
\inf\{|\xi|^\sz,\,|\xi|^{-\tau}\}\r\}.
\end{eqnarray*}

It is known that every one-to-one operator $T$ of type $\omega$ in
$L^2(\rn)$ has a unique holomorphic functional calculus; see, for
example, \cite{mc86}. More precisely, let $T$ be a one-to-one
operator of type $\omega$, with $\omega\in[0,\,\pi)$,
$\mu\in(\omega,\,\pi)$, $\sz,\tau\in(0,\,\fz)$ and
$f\in\Psi_{\sz,\tau}(S_\mu^0)$. The \emph{function of the operator
$T$}, $f(T)$, can be defined by the $H_\fz$ functional calculus in
the following way,
\begin{eqnarray*}
f(T)=\frac{1}{2\pi i}\dint_\Gamma(\xi I-T)^{-1}f(\xi)\,d\xi,
\end{eqnarray*}
where $\Gamma:=\{re^{i\nu}:\ \fz>r>0\}\cup\{re^{-i\nu}:\ 0<r<\fz\}$,
$\nu\in(\omega,\,\mu)$, is a curve consisting of two rays
parameterized anti-clockwise. It is known that $f(T)$  is
independent of the choice of $\nu\in(\omega,\,\mu)$ and the integral
is absolutely convergent in $\|\cdot\|_{\cl(L^2(\rn))}$ (see
\cite{mc86,ha06}).

In what follows, we {\it always assume $\omega\in[0,\,\pi/2)$}.
Then, it follows, from \cite[Proposition 7.1.1]{ha06}, that, for every
operator $T$ of type $\omega$ in $L^2(\rn)$, $-T$ generates a
holomorphic $C_0$-semigroup $\{e^{-zT}\}_{z\in S^0_{\pi/2-\omega}}$
on the open sector $S^0_{\pi/2-\omega}$ such that,
for all $z\in S^0_{\pi/2-\omega}$, $\|e^{-zT}\|_{\cl(L^2(\rn))}\le1$
and, moreover, every \emph{nonnegative self-adjoint operator is of
type $0$}.

Let $\Psi(S_\mu^0):=\cup_{\sz,\tau>0}\Psi_{\sz,\tau}(S_\mu^0)$. It
is well known that the above holomorphic functional calculus defined
on $\Psi(S_\mu^0)$ can be extended to $H_\fz(S_\mu^0)$ via a limit
process (see \cite{mc86}).  Recall that, for $\mu\in(0,\,\pi)$, the
operator $T$ is said to have a \emph{bounded $H_\fz(S_\mu^0)$
functional calculus} in the Hilbert space $\mathcal{H}$ if there
exists a positive constant $C$ such that, for all $\psi\in
H_\fz(S_\mu^0)$, $\|\psi(T)\|_{\cl(\mathcal{H})}\le
C\|\psi\|_{L^\fz(S_\mu^0)}$, and $T$ is said to have a \emph{bounded
$H_\fz$ functional calculus} in the Hilbert space $\mathcal{H}$ if
there exists $\mu\in (0,\,\pi)$ such that $T$ has a bounded
$H_\fz(S_\mu^0)$ functional calculus.

Throughout the whole article, we \emph{always assume} that $L$ satisfies the
following three assumptions:

\smallskip

\noindent {\bf Assumption {{\bf $(\mathcal{L})_1$}}.} \emph{The
operator $L$ is a one-to-one operator of type $\omega$ in $L^2(\rn)$
with $\omega\in[0,\,\pi/2)$.}

\smallskip

\noindent {\bf Assumption {\bf $(\mathcal{L})_2$}.} \emph{The
operator $L$ has a bounded $H_\fz$ functional calculus in
$L^2(\rn)$.}

\smallskip

\noindent {\bf Assumption {\bf $(\mathcal{L})_3$}.} \emph{Let $k\in
\nn$. The operator $L$ generates a holomorphic semigroup
$\{e^{-tL}\}_{t>0}$ which satisfies the {\it $k$-Davies-Gaffney
estimates}, namely, there exist positive constants $C$ and $C_1$
such that, for all closed sets $E$ and $F$ in $\rn$, $t\in(0,\,\fz)$
and $f\in L^2(\rn)$ supported in $E$,
\begin{eqnarray*}
\lf\|e^{-tL}f\r\|_{L^2(F)}\le C
\exp\lf\{-C_1\frac{\lf[\dist(E,\,F)\r]^{2k/(2k-1)}}
{t^{1/(2k-1)}}\r\}\|f\|_{L^2(E)},
\end{eqnarray*}
here and hereafter, for any $p\in (0,\,\fz)$,
$\|f\|_{L^p(E)}:=\{\int_{E}\lf|f(x)\r|^p\,dx\}^{1/p}$
and
$$\dist(E,\,F):= \inf_{x\in E,\,y\in
F}|x-y|$$
denotes the \emph{distance between $E$ and $F$}.}

\smallskip

In many cases, we also need the following assumption,
which is stronger than Assumption {\bf $(\mathcal{L})_3$}.

\smallskip

\noindent {\bf Assumption {\bf $(\mathcal{L})_4$}.} \emph{Let
$k\in {\zz_+}$ and $(p_-(L),\,p_+(L))$ be the \emph{range of
exponents} $p\in[1,\,\fz]$ for which the holomorphic semigroup
$\{e^{-tL}\}_{t>0}$ is bounded on $L^p(\rn)$. Then, for all
$p_-(L)<p\le q<p_+(L)$, $\{e^{-tL}\}_{t>0}$ satisfies the {\it
$L^p-L^q$ $k$-off-diagonal estimate}, namely, there exist positive
constants $C_2$ and $C_3$ such that, for all closed sets $E$,
$F\subset\rn$ and $f\in L^p(\rn) \cap L^2(\rn)$ supported in $E$,
\begin{eqnarray*}
\lf\|e^{-tL}f\r\|_{L^q(F)}\le
C_2t^{\frac{n}{2k}(\frac{1}{q}-\frac{1}{p})}
\exp\lf\{-\frac{\lf[\dist(E,\,F)\r]^{2k/(2k-1)}}
{C_3t^{1/(2k-1)}}\r\}\|f\|_{L^p(E)}.
\end{eqnarray*}}

\begin{remark} \label{r2.1}
The notion of the off-diagonal estimates
(or the so called Davies-Gaffney estimates)
of the semigroup $\{e^{-tL}\}_{t>0}$ are first introduced by
Gaffney \cite{ga59} and Davies \cite{da95}, which serves as good substitutes
of the Gaussian upper bound of the associated heat kernel; see also
\cite{bk05,am07} and related references.
We point out that, when $k=1$, the $k$-Davies-Gaffney estimates are
the usual Davies-Gaffney estimates (or the {\it $L^2$
off-diagonal estimates} or just the {\it Gaffney estimates}) (see, for
example, \cite{hlmmy,hm09,hm09c,jy11,hmm}).
\end{remark}

\begin{proposition}[\cite{cy}]\label{p2.2}
Let $L$ be the
$2k$-order divergence form homogeneous elliptic operator  as in \eqref{1.1}. Then
$L$ satisfy Assumptions {$(\mathcal{L})_1$}, {$(\mathcal{L})_2$},
{$(\mathcal{L})_3$} and {$(\mathcal{L})_4$}.
\end{proposition}

In order to make this article self-contained,  we list the following three technical lemmas which are needed in the proofs of our main results.

\begin{lemma}[\cite{cy}]\label{l2.3}
Assume that the operator $L$ defined on $L^2(\rn)$ satisfies
Assumptions $(\mathcal{L})_1$, $(\mathcal{L})_2$ and
$(\mathcal{L})_3$. Then, for all $m\in\nn$, the family of operators,
$\{(tL)^me^{-tL}\}_{t>0}$, also satisfy the $k$-Davies-Gaffney estimates.
\end{lemma}

\begin{lemma}[\cite{cy}]\label{l2.4}
Let $\{A_t\}_{t>0}$ and $\{B_t\}_{t>0}$ be two families of linear
operators satisfying the $k$-Davies-Gaffney estimates. Then the
families of linear operators $\{A_tB_t\}_{t>0}$ also satisfy the
$k$-Davies-Gaffney estimates.
\end{lemma}

\begin{lemma}[\cite{cy}]\label{l2.5}
Let $M\in\nn$ and $L$ be as in
\eqref{1.1}. Then there exists a positive constant
$C$ such that, for all closed sets $E$, $F$ in $\rn$
with $\dist(E,\,F)>0$, $f\in L^2(\rn)$ supported in $E$ and
$t\in(0,\,\fz)$,
\begin{eqnarray*}
\lf\|\nabla^{k}{L}^{-1/2}\lf(I-e^{-t{L}}\r)^{M}f\r\|_{L^2(F)}\le
C \lf(\frac{t}{\lf[\dist(E,\,F)\r]^{2k}}\r)^M\|f\|_{L^2(E)}
\end{eqnarray*}
and
\begin{eqnarray*}
\lf\|\nabla^{k}{L}^{-1/2}\lf(t{L}e^{-t{L}}\r)^{M}f\r\|_{L^2(F)}\le
C \lf(\frac{t}{\lf[\dist(E,\,F)\r]^{2k}}\r)^M\|f\|_{L^2(E)}.
\end{eqnarray*}
\end{lemma}

\subsection{The weak tent spaces $WT^p(\mathbb{R}^{n+1}_+)$}\label{s2.2}

\hskip\parindent In this subsection, we introduce the weak tent
space and establish its weak atomic characterization. This construction constitutes
a crucial component to obtain the weak atomic or molecular
characterizations of the weak Hardy space.

We first recall the definition of the \emph{tent space}. Let $F$ be
a function on $\mathbb{R}^{n+1}_+:=\rn\times(0,\fz)$. For all
$x\in\rn$, the $\mathcal{A}$-{\it functional} $\mathcal{A}(F)(x)$
of $F$ is defined by setting
\begin{eqnarray}\label{2.1}
\mathcal{A}(F)(x):=\lf\{\iint_{\Gamma(x)}\lf|F(y,\,t)\r|^2
\frac{dy\,dt}{t^{n+1}}\r\}^{\frac{1}{2}},
\end{eqnarray}
 where
\begin{eqnarray}\label{2.x1}
\Gamma(x):=\{(y,\,t)\in\mathbb{R}^{n+1}_+:\ |y-x|<t\}
\end{eqnarray}
 is a
\emph{cone with vertex $x$}. For all $p\in(0,\,\fz)$, the {\it tent
space} $T^p(\mathbb{R}^{n+1}_+)$ is defined by
\begin{eqnarray}\label{2.xx2}
T^p(\mathbb{R}^{n+1}_+):=\lf\{F:\ \mathbb{R}^{n+1}_+\rightarrow
\cc: \ \|F\|_{T^p(\mathbb{R}^{n+1}_+)}:=\|\mathcal{A}(F)
\|_{L^p(\rn)}<\fz \r\}.
\end{eqnarray}
 For all open sets $\boz$, let
$\widehat{\boz}:=\rr^{n+1}_+\setminus\cup_{x\in\rn\setminus \boz}
\Gamma(x)$ be the
\emph{tent over $\boz$}.  For all $x_B\in\rn$ and $r_B\in(0,\,\fz)$, let
$B:=B(x_B,\,r_B)$ be the \emph{ball} in $\rn$. It is easy to see that
$\widehat B=\{(y,\,t):\ |y-x_B|\le r_B-t\}$. For any $p\in(0,\,\fz)$
and ball $B$, a function $A$ defined on $\mathbb{R}^{n+1}_+$ is
called a $T^p(\mathbb{R}^{n+1}_+)$-\emph{atom associated to} $B$, if
$\supp A\subset \widehat B$ and
$$\lf\{\iint_{\widehat{B}}\lf|A(x,t)\r|^2\frac{dx\,dt}
{t}\r\}^{\frac{1}{2}}\le|B|^{\frac{1}{2}-\frac{1}{p}}.$$

For $p\in(0,\,\fz)$, let $WL^p(\rn)$ be the \emph{weak Lebesgue
space} with the \emph{quasi-norm}
\begin{eqnarray*}
\|f\|_{WL^p(\rn)}:=\lf[\sup_{\az\in(0,\,\fz)}\az^p \lf|\{x\in\rn:\
|f(x)|>\az\}\r|\r]^{1/p}.
\end{eqnarray*}
The \emph{weak tent space} $WT^p(\mathbb{R}^{n+1}_+)$ is defined to
be the collection of all functions $F$ on $\mathbb{R}^{n+1}_+$ such that
its $\mathcal{A}$-functional $\mathcal{A}(F)\in WL^p(\rn)$. For any
$F\in WT^p(\mathbb{R}^{n+1}_+)$, define its \emph{quasi-norm} by
$\|F\|_{WT^p(\mathbb{R}^{n+1}_+)}:=\|\mathcal{A}(F)\|_{WL^p(\rn)}$.

For the weak tent space, we have the following weak atomic
decomposition.

\begin{theorem}\label{t2.6}
Let $p\in(0,\,1]$ and $F\in WT^p(\mathbb{R}^{n+1}_+)$. Then there
exists a sequence of $T^p(\mathbb{R}^{n+1}_+)$-atoms, $\{A_{i,j}\}_{i\in\zz,j\in\zz_+}$,
associated, respectively, to the
balls $\{B_{i,j}\}_{i\in\zz,\,j\in\zz_+}$ such that

{\rm(i)} $F=\sum_{i\in\zz,\,j\in\zz_+}\lz_{i,j}A_{i,j}$ pointwisely almost
everywhere in $\rr^{n+1}_+$, where
$\lz_{i,j}:=\wz C2^i |B_{i,j}|^{1/p}$ and $\wz C$ is a positive constant
independent of $F$;

{\rm(ii)}  there exists a positive constant $C$, independent of $F$,  such that
$$\dsup_{i\in\zz}\lf(\sum_{j\in\zz_+}|\lz_{i,j}|^p\r)^{\frac{1}{p}}\le C
\|F\|_{WT^p(\mathbb{R}^{n+1}_+)};$$

{\rm(iii)} for all $i\in\zz$ and $j\in\zz_+$, let $\wz B_{i,\,j}:=\frac{1}{10\sqrt{n}}B_{i,\,j}$.
Then, for all $i\in\zz$, $\{\wz B_{i,j}\}_{j\in\zz_+}$ are mutually disjoint.
\end{theorem}

In order to prove this theorem, we need the following
\emph{Whitney decomposition theorem} (see, for example, \cite[p.\,463]{gr08}).

\begin{lemma}\label{l2.x7}
 Let $\boz$ be an open nonempty proper subset of $\rn$.
 Then there exists a family of closed cubes $\{Q _j\}_{j\in\zz_+}$ such that

{\rm (i)} $\cup_{j\in\zz_+} Q_j = \boz$ and $\{Q _j\}_{j\in\zz_+}$ have disjoint interiors;

{\rm (ii)} for all $j\in\zz_+$, $\sqrt{n}\,l_{Q_j } \le \dist(Q_j ,\, \boz^{\complement}) \le 4 \sqrt{n}\,l_{Q_j }$,
where $l_{Q_j }$ denotes the length of the cube $Q_j$;

{\rm (iii)} for any $j,\, k\in\zz_+$, if the boundaries of two cubes $Q_j$ and $Q_k$ touch, then
$\frac14\le\frac{l_{Q_ j}}{ l_{Q_k}}\le 4;$

{\rm (iv)} for a given $j\in\zz_+$, there exist at most $12 n$ different cubes $Q_k$ that touch $Q_j$.
\end{lemma}

For any fixed $\gz\in(0,\,1)$ and bounded open set $\boz$ in
$\rn$ with the complementary set $F$, let $\boz_\gz^*:=\{x\in\rn:\
\mathcal{M}(\chi_{\boz})(x)>1-\gz\}$ and $F_\gz^*:=\rn\setminus
\boz_\gz^*$, where $\mathcal{M}$ denotes the usual \emph{Hardy-Littlewood
maximal function}, namely, for all $f\in L^1_\loc(\rn)$ and $x\in\rn$,
\begin{eqnarray}\label{2.x2}
\mathcal{M}f(x):=\dsup_{B\ni x}\frac{1}{|B|}\dint_{B}|f(y)|\,dy,
\end{eqnarray}
where the supremum is taken over all balls containing $x$.
We also need the following auxiliary lemma.

\begin{lemma}[\cite{cms85}]\label{l2.8}
Let $\az\in(0,\,\fz)$. Then there exist constant $\gz\in(0,\,1)$,
sufficiently close to $1$, and positive constant $C$ such that, for
any closed set $F$, whose complement (denoted by $\boz$) has finite
measure, and  for any non-negative function $\Phi$ on $\rr^{n+1}_+$,
\begin{eqnarray*}
\dint_{\mathcal{R}(F_\gz^*)}\Phi(y,\,t)t^n\,dy\,dt\le C
\dint_{F}\lf\{\dint_{\Gamma(x)}\Phi(y,\,t)\,dy\,dt\r\}\,dx,
\end{eqnarray*}
where $\mathcal{R}(F^{*}_{\gz}):= \cup_{x\in F_\gz^*}\Gamma(x)$
and $\Gamma(x)$ for $x\in\rn$ is as in \eqref{2.x1}.
\end{lemma}

\begin{proof}[Proof of Theorem \ref{t2.6}]
We show this theorem in the order of (ii), (i) and (iii).

To show (ii), let $F\in WT^p(\rr^{n+1}_+)$. For all $i\in\zz$, let
$O_i:=\{x\in\rn:\ \mathcal{A}(F)(x)>2^i\}$. It is easy to see that
$O_{i+1}\subset O_{i}$. Moreover, since $F\in
WT^p(\mathbb{R}^{n+1}_+)$, we readily see that $|O_i|<\fz$. For
fixed $\gz\in(0,\,1)$ satisfying the same restriction as in Lemma
\ref{l2.8}, let
$$(O_i)_\gz^*:=\{x\in\rn:\ \mathcal{M}(\chi_{O_i})(x)>1-\gz\}.$$
By abuse of notation, we \emph{simply
write}  $O_i^*$ instead of $(O_i)_\gz^*$. Since $O_i$ is open, we easily see that
$O_i\subset O_i^*$ and, by the weak $(1,1)$ boundedness of $\cm$,
we further know that there exists a positive constant $C(\gz)$, depending on $\gz$,  such
that $|O^*_i|\le C(\gz) |O_i|$. For each $O^*_i$, using Lemma
\ref{l2.x7}, we obtain a Whitney decomposition
$\{Q_{i,j}\}_{j\in\zz_+}$ of $O^*_i$.
Let $B_{i,\,j}$ be the ball having the same center as $Q_{i,\,j}$ with the radius
$5\sqrt n\, l_{Q_{i,\,j}}$, where $l_{Q_{i,\,j}}$ denotes the length of $Q_{i,\,j}$.
By Lemma \ref{l2.x7}(ii), we immediately see that
$B_{i,\,j}\cap (O^*_i)^\complement\ne \emptyset$.

Now, for all $i\in\zz$ and $j\in\zz_+$, let $\bdz_{i,j}:=\widehat{B}_{i,\,j}\cap (Q_{i,j}\times
(0,\,\fz))\cap (\widehat O_i^*\setminus \widehat O_{i+1}^*)$,
$\lz_{i,j}:=\wz C2^i |B_{i,j}|^{\frac{1}{p}}$ and
$A_{i,j}:=F\chi_{\bdz_{i,j}}/\lz_{i,j}$, where
$\widehat{B}_{i,\,j}$ and $\widehat{O}_i^*$ denote, respectively, the tents over $B_{i,\,j}$
and $O_i^*$, and $\wz C$ is a positive constant independent of $F$, which
will be determined later. From the fact that
$\supp F\subset \cup_{i\in\zz}\cup_{j\in\zz_+} \bdz_{i,j}$, it
follows that
\begin{equation}\label{2.2}
F(x,t)=\sum_{i\in\zz,j\in\zz_+}F(x,t)\chi_{\bdz_{i,j}}(x,t)
=\sum_{i\in\zz,j\in\zz_+}\lz_{i,j}A_{i,j}(x,\,t)
\end{equation}
pointwisely almost everywhere in $(x,\,t)\in\rr^{n+1}_+$.

Moreover, for all $i\in\zz$, by the definition of $\lz_{i,j}$, Lemma \ref{l2.x7}(i),
the fact that $|O^*_i|\le C(\gz)|O_i|$ and the
definition of $O_i$, we conclude that
\begin{eqnarray}\label{2.3}
\dsum_{j\in\zz_+}|\lz_{i,j}|^p&&\sim 2^{ip}\dsum_{j\in\zz_+}|B_{i,j}|\sim 2^{ip}\dsum_{j\in\zz_+}|Q_{i,j}|\ls
2^{ip}|O_i|\sim 2^{ip}|\{x\in\rn:\ \mathcal{A}(F)(x)>2^i\}|\\
\nonumber &&\ls\|\mathcal{A}(F)\|_{WL^p(\rn)}^p
\sim\|F\|_{WT^p(\mathbb{R}^{n+1}_+)}^p,
\end{eqnarray}
which immediately implies (ii).

On the other hand, for any closed set $F$, let
$\mathcal{R}(F):=\cup_{x\in F}\Gamma(x)$. For all $(y,\,t)\in
\bdz_{i,j}$, it follows, from the fact that
$\widehat{O}_i^*=\{(y,\,t):\ \dist(y,\,\rn\setminus O_i^*)\ge t\}$
and Lemma \ref{l2.x7}(ii), that $t\le\dist(y,\, \rn\setminus
O_i^*)<r_{B_{i,j}}$. Thus, for all $H\in
T^2(\mathbb{R}^{n+1}_+)$ satisfying
$\|H\|_{T^2(\mathbb{R}^{n+1}_+)}\le 1$, by the fact that
$\rr^{n+1}_+\setminus \widehat O_{i+1}^*=\mathcal{R}(\rn\setminus
O^{*}_{i+1})$, the support condition of $A_{i,j}$ ($\supp A_{i,\,j}\subset \widehat{B}_{i,\,j}$),
Lemma \ref{l2.8},  H\"older's inequality and the definitions of
$A_{i,j}$ and $O_{i+1}$, we see that
\begin{eqnarray}\label{2.4}
&&\lf|\lf\laz A_{i,j},\,H\r\raz_{T^2(\mathbb{R}^{n+1}_+)}\r| \\
&&\hs\le\dint_{\mathcal{R}(\rn\setminus
O^{*}_{i+1})}\lf|A_{i,j}(y,\,t)H(y,\,t)\r|\chi_{\bdz_{i,j}}(y,\,t)\,
\frac{dy\,dt}{t}\nonumber \\
\nonumber&&\hs\ls \dint_{\rn\setminus
O_{i+1}}\lf\{\dint_0^{r_{B_{i,j}}}\dint_{|y-x|<t}\lf|A_{i,j}(y,t)
H(y,t)\r|\chi_{\bdz_{i,j}}(y,\,t)\frac{dy\,dt}{t^{n+1}}\r\}\,dx\\
\nonumber &&\hs\sim\dint_{2B_{i,j}\setminus
O_{i+1}}\lf\{\dint_0^{r_{B_{i,j}}}\dint_{|y-x|<t}\lf|A_{i,j}(y,t)
H(y,t)\r|\chi_{\bdz_{i,j}}(y,\,t)\frac{dy\,dt}{t^{n+1}}\r\}\,dx\\
\nonumber &&\hs\ls\lf\{\dint_{2B_{i,j}\setminus
O_{i+1}}\lf|\mathcal{A}(A_{i,j})(x)\r|^2\,dx\r\}^{\frac{1}{2}}\|
\mathcal{A}(H)\|_{L^2(\rn)}\\ \nonumber
&&\hs\ls\frac 1{\wz C}2^{-i}|B_{i,j}|^{-\frac{1}{p}}\lf\{\dint_{2B_{i,j}\setminus
O_{i+1}}\lf|\mathcal{A}(F)(x)\r|^2\,dx\r\}^{\frac{1}{2}}\\
&&\hs\ls\frac 1{\wz C}2^{-i}|B_{i,j}|^{-\frac{1}{p}}\lf[2^{2i}
|B_{i,j}|\r]^{\frac{1}{2}}\sim\frac 1{\wz C}
|B_{i,j}|^{\frac{1}{2}-\frac{1}{p}},\nonumber
\end{eqnarray}
which, together with $(T^2(\rr^{n+1}_+))^*=T^2(\rr^{n+1}_+)$, further implies that
$\|A_{i,j}\|_{T^2(\rr^{n+1}_+)}\ls \frac 1{\wz C}|B_{i,j}|^{\frac{1}{2}-\frac{1}{p}}$.
Moreover, from the definitions of $\bdz_{i,j}$ and
$\widehat{O}_i^*$, and Lemma \ref{l2.x7}(ii), we deduce that
$\bdz_{i,j}\subset \widehat{B_{i,j}}$. Thus, by choosing $\wz C$ large enough such that
$\|A_{i,j}\|_{T^2(\rr^{n+1}_+)}\le |B_{i,j}|^{\frac{1}{2}-\frac{1}{p}}$,
then $A_{i,j}$ is a $T^p(\mathbb{R}^{n+1}_+)$-atom associated to the ball $B_{i,j}$,
which, combined with \eqref{2.2}, implies (i).

Observe that (iii) follows readily from Lemma
\ref{l2.x7}(vi), which completes the proof of Theorem \ref{t2.6}.
\end{proof}

If a function $F\in WT^p(\mathbb{R}^{n+1}_+)$ also
belongs to $T^2(\mathbb{R}^{n+1}_+)$, then the weak atomic
decomposition obtained in Theorem \ref{t2.6} also converges in $T^2
(\mathbb{R}^{n+1}_+)$, which is the conclusion of the following corollary.

\begin{corollary}\label{c2.9}
Let $p\in(0,\,1]$. For all $F\in WT^p(\mathbb{R}^{n+1}_+)\cap
T^2(\mathbb{R}^{n+1}_+)$, the weak atomic decomposition
$F=\sum_{i\in\zz,j\in\zz_+}\lz_{i,j}A_{i,j}$ obtained in Theorem
\ref{t2.6} also holds true in $T^2(\mathbb{R}^{n+1}_+)$.
\end{corollary}

\begin{proof}
Let $F\in WT^p(\mathbb{R}^{n+1}_+)\cap T^2(\mathbb{R}^{n+1}_+)$.
We use the same notation as in Theorem \ref{t2.6} and its proof.
By Theorem \ref{t2.6}, we see that
the weak atomic decomposition
$F=\sum_{i\in\zz,j\in\zz_+}\lz_{i,j}A_{i,j}$ holds true pointwisely
almost everywhere in $\rr^{n+1}_+$. Thus, for all $N_1,\ N_2\in\nn$, from Fubini's theorem, the definitions of
$A_{i,j}$ and $\lz_{i,j}$ and the bound overlap property of
$\{\bdz_{i,j}\}_{i\in\zz,j\in\zz_+}$, we deduce that
\begin{eqnarray*}
\lf\|\dsum_{|i|\ge N_1\, \mathrm{or}\, j\ge
N_2}\lz_{i,j}A_{i,j}\r\|_{T^2(\mathbb{R}^{n+1}_+)}^2
&&\sim\dint_0^\fz\dint_{\rn} \lf|\dsum_{|i|\ge N_1\, \mathrm{or}\,
j\ge
N_2}\lz_{i,j}A_{i,j}(y,\,t)\r|^2\,\frac{dy\,dt}{t}\\
&&\ls\dint_0^\fz\dint_{\rn}\lf|F(y,\,t)\chi_{\cup_{|i|\ge
N_1\,\mathrm{or}\,j\ge
N_2}\bdz_{i,j}}(y,\,t)\r|^2\,\frac{dy\,dt}{t}.
\end{eqnarray*}
By letting $N_1$, $N_2\to \infty$ and using the condition that $F\in
T^2(\mathbb{R}^{n+1}_+)$, we know that
$F=\sum_{i\in\zz,j\in\zz_+}\lz_{i,j}A_{i,j}$ holds true in
$T^2(\mathbb{R}^{n+1}_+)$, which completes the proof of Corollary
\ref{c2.9}.
\end{proof}

\begin{remark}\label{r2.10}
Let $k\in\zz$, $F\in WT^p(\rr^{n+1}_+)$, $F\neq 0$ and
$$O_{k}:=\{x\in\rn:\
\mathcal{A}(F)(x)>2^{k}\}.$$
Since $F\in WT^p(\rr^{n+1}_+)$, it is easy to see that,
for all $k\in\zz$, $|O_k|<\fz$ and $|O_{k+1}|\le |O_k|$.
Observe that, for some $k\in\zz$, $|O_{k+1}|$ may equal to $|O_{k}|$.

We now construct two index sets $\mathcal{I}\subset \zz$ and
$\wz{\mathcal{I}}\subset \zz$, which both are needed
in establishing the dual theory of the weak Hardy space $WH_L^p(\rn)$ in Section \ref{s3} below
and have the following properties:

\textbf{Case 1)} If, for all $k\in\zz$, $|O_k|\in(0,\,\fz)$, in this case, we fix any $i_0\in\zz$ and
$\wz i_0=0$ and then choose $\{i_j\}_{j\in\zz\setminus\{0\}}\subset \zz$ and
$\{\wz i_j\}_{j\in\zz\setminus\{0\}}\subset \zz$ such that
\begin{itemize}
\item[(i)] $\mathcal{I}:=\{i_j\}_{j\in\zz}\subset \zz$ and $\wz{\mathcal{I}}:=\{\wz i_j\}_{j\in\zz}
\subset \zz$ are strictly increase in $j$;

\item[(ii)] for all $j\in\zz$ and $i_j\in\mathcal{I}$,
$|O_{i_{j+1}}|<|O_{i_{j}}|;$

\item[(iii)] for all $j\in\nn$, $i_j\in \mathcal{I}$ and $\wz{i}_j\in \wz{\mathcal{I}}$,
$$|O_{i_j}|\in [2^{-\wz
i_j-1}|O_{i_0}|,\,2^{-\wz i_j}|O_{i_0}|)$$
and
$$|O_{i_{-j}}|\in (2^{-\wz
i_{-j}-1}|O_{i_0}|,\,2^{-\wz i_{-j}}|O_{i_0}|].$$
\end{itemize}

\textbf{Case 2)} If there exists $k\in\zz$ such that $|O_k|=0$, in this case, let
$$ i_0:= \min\{k-1\in\zz:\ |O_k|=0\}.$$
Then choose $\{i_j\}_{j\in\zz\setminus\{0\}}\subset \zz$ and
$\{\wz i_j\}_{j\in\zz\setminus\{0\}}\subset \zz$ such that
\begin{itemize}
\item[(i)] for all $j\in\zz_+$, $i_j\in \mathcal{I}$ and
$\wz{i}_j\in \wz{\mathcal{I}}$, let $i_j=i_0$ and $\wz{i}_j=0$;

\item[(ii)] for all $j\in\zz\setminus\zz_+$, choose $i_j\in \mathcal{I}$ and
$\wz{i}_j\in \wz{\mathcal{I}}$ satisfy that
$|O_{i_{j+1}}|<|O_{i_{j}}|$
and
$$|O_{i_{j}}|\in (2^{-\wz
i_{j}-1}|O_{i_0}|,\,2^{-\wz i_{j}}|O_{i_0}|].$$
\end{itemize}

Indeed, to prove the above claim in Case 1),  for
any fixed $i_0\in\zz$ as in Case 1), we let
$$i_1:= \min\{i\in\zz:\
|O_i|<|O_{i_0}|\}$$
and
$$i_{-1}:= \max\{i\in\zz:\ |O_i|>|O_{i_0}|\}.$$
From the facts $\lim_{i\to \fz} |O_i|=0$ and  $\lim_{i\to -\fz} |O_i|=\fz$,
we deduce that such $i_1$ and $i_{-1}$ do exist.

Then choose $\wz{i}_1,\, \wz{i}_{-1}\in\zz$ satisfying $$|O_{i_1}|\in [2^{-\wz
i_1-1}|O_{i_0}|,\,2^{-\wz i_1}|O_{i_0}|)$$
and
$$|O_{i_{-1}}|\in (2^{-\wz
i_{-1}-1}|O_{i_0}|,\,2^{-\wz i_{-1}}|O_{i_0}|].$$
By a simple calculation, it is easy to see that $i_{-1}<i_0<i_1$,
$\wz i_{-1}< 0\le\wz{i}_{1}$ and
$|O_{i_1}|<|O_{i_0}|<|O_{i_{-1}}|$.

Now, let  $$i_2:= \min\{i\in\zz:\
|O_i|<2^{-\wz i_{1}-1}|O_{i_0}|\},$$
$$i_{-2}:= \max\{i\in\zz:\
|O_i|>2^{-\wz i_{-1}}|O_{i_0}|\}$$
and choose
$\wz{i}_2,\, \wz{i}_{-2}\in\zz$ satisfying $$|O_{i_2}|\in [2^{-\wz
i_2-1}|O_{i_0}|,\,2^{-\wz i_2}|O_{i_0}|)$$
and
$$|O_{i_{-2}}|\in (2^{-\wz
i_{-2}-1}|O_{i_0}|,\,2^{-\wz i_{-2}}|O_{i_0}|].$$
It is easy to see that $ i_{-2}< i_{-1}<i_0< {i}_{1}<{i}_{2}$,
$\wz i_{-2}<\wz i_{-1}< 0\le \wz{i}_{1}<\wz{i}_{2}$
and $$|O_{i_2}|<|O_{i_1}|<|O_{i_0}|<|O_{i_{-1}}|<|O_{i_{-2}}|.$$

Continuing this
process, we obtain a sequence $\{O_{i_j}\}_{j\in\zz}$ of
strictly decreasing open sets, and sequences $\{i_j\}_{j\in\zz}$,
$\{\wz{i}_j\}_{j\in\zz}$ of increasing numbers. Denote the
index sets $\{i_j\}_{j\in\zz}$ and $\{\wz{i}_j\}_{j\in\nn}$,
respectively, by $\mathcal{I}$ and $\wz{\mathcal{I}}$, we conclude that
$\mathcal{I}$ and $\wz{\mathcal{I}}$ have the desired properties in Case 1).

We now turn to the Case 2). In this case,
we define the index sets $\mathcal{I}:=\{i_j\}_{j\in\zz}$ and
$\wz{\mathcal{I}}:=\{\wz{i}_j\}_{j\in\zz}$ as follows.
For all $j\in\zz\setminus\zz_+,$ since $|O_{i_0}|>0$, we
choose the indices $i_j$ and
$\wz{i}_j$ as in Case 1).
For all $j\in\zz_+$, let $i_j:=i_0$ and $\wz i_j:=0$.
Observe that, for all $j\in\nn$, $|O_{i_j}|=0$.
By some calculations similar to those used in Case 1),
we know that $\mathcal{I}$ and $\wz{\mathcal{I}}$ also
have the desired properties. Thus, both claims in Case 1)
and Case 2) hold true.

Finally, we point out that, by following
the same line of the proof of Theorem \ref{t2.6}, but replacing
$\{O_i\}_{i\in\zz}$ by $\{O_i\}_{i\in\mathcal{I}}$ defined here, we also obtain a weak
atomic decomposition of $T^p(\rr^{n+1}_+)$ with the same properties.
In this case, the achieved atomic decomposition is of the following form
\begin{eqnarray*}
F=\dsum_{i\in\mathcal{I}}\dsum_{j\in\zz_+}\lz_{i,j}A_{i,j},
\end{eqnarray*}
here and hereafter, for notional simplicity, we denote $i_j\in \mathcal{I}$
simply by $i\in\mathcal{I}$.
\end{remark}

Now, we establish another weak atomic decomposition of
$WT^p(\rr^{n+1}_+)$ which plays an important role in
establishing the dual theory of $WH_L^p(\rn)$ in Section \ref{s3}.

\begin{theorem}\label{t2.12}
Let $p\in(0,\,1]$. Then, for all
$F\in WT^p(\mathbb{R}^{n+1}_+)$, there
exist an index set $\mathcal{I}\subset \zz$
and $\{{A}_{i,\,j}\}_{i\in\mathcal{I}, j\in \Lambda_i}$ of
$T^p(\mathbb{R}^{n+1}_+)$-atoms associated to balls $\{B_{i,\,j}\}_{i\in
\mathcal{I}, j\in \Lambda_i}$ such that

{\rm(i)} $$F=\dsum_{i\in\mathcal{I}}
\dsum_{j\in \Lambda_i}
{\lz}_{i,\,j} {A}_{i,\,j}$$
pointwisely almost everywhere in $\rr^{n+1}_+$, where, for all $i\in\mathcal{I}$,
$\Lambda_i\subset \zz_+$ is an index set depending on $i$
and,  for all $j\in{\Lambda_i}$,  ${\lz}_{i,\,j}:=\wz C2^i
|B_{i,\,j}|^{\frac{1}{p}}$ and $\wz C$ is a positive constant
independent of $F$;

{\rm(ii)} for all $i\in\mathcal{I}$ and $j\in\Lambda_i$,
let $r_{i}:=\inf_{j\in\Lambda_i}\{r_{B_{i,\,j}}\}$ and $\wz B_{i,\,j}:=\frac{1}{10\sqrt{n}}B_{i,\,j}$.
Then, for all $i\in\mathcal{I}$, $r_i>0$ and $\{\wz B_{i,j}\}_{j\in\Lambda_i}$ are mutually disjoint;

{\rm(iii)} there exists a positive
constant $C$, depending only on $n$ and $p$,
such that, for all $i\in\mathcal{I}$,
$$\lf\{\sum_{j\in{\Lambda_i}}
\lf|\lz_{i,\,j}\r|^p\r\}^{\frac{1}{p}}\le C \|F\|_{WT^p(\mathbb{R}^{n+1}_+)}.$$
\end{theorem}

\begin{proof}
We first prove  (i) and (ii) of Theorem \ref{t2.12}.
Let $\mathcal{I}$ be as in Remark \ref{r2.10},
$i\in\mathcal{I}$ and $O_i^*$ be as in the proof of Theorem
\ref{t2.6}. Without loss of generality, we may assume that $|O_i^*|>0$;
otherwise, we neglect the set $O_i^*$,
since we only need the atomic decomposition to hold true
almost everywhere in $\rr^{n+1}_+$.

For all $i\in \mathcal{I}$, let $\ez\in(0,\,\fz)$ such
that the \emph{open set}
\begin{eqnarray}\label{2.5}
O_{i,\ez}^*:=O_{i}^*\cup\lf\{x\in (O_{i}^*)^{\complement}:\ \dist(
x,\,\pat{O}_{i}^*)<\ez\r\}
\end{eqnarray}
satisfies $|O_{i,\ez}^*|< 2|O_{i}^*|$.
Let $\{Q_{i,j}\}_{j\in\zz_+}$ be a Whitney decomposition of
$O_{i,\ez}^*$ as in Lemma \ref{l2.x7}. Assume that $\{Q_{i,j}\}_{j\in{\Lambda_i}}$ is the
\emph{maximal subsequence} of $\{Q_{i,j}\}_{j\in\zz_+}$
such that, for all $j\in{\Lambda_i}$, $Q_{i,j} \cap O_{i}^*\ne \emptyset$, where $\Lambda_i\subset\zz_+$.
Let $l_i:=\inf\{l_{Q_{i,\,j}},\,j\in{\Lambda_i}\}$, we now claim that, for all $i\in\mathcal{I}$,
$l_i>0$.

Indeed, if $Q_{i,\,j}\cap \pat O_i^*\ne \emptyset$, then we see that, for all $y\in \pat O_i^*$,
\begin{eqnarray}\label{2.x6}
\dist((O_{i,\,\epsilon}^*)^\complement,\,y)\ge \epsilon.
\end{eqnarray}
Otherwise, if $\dist((O_{i,\,\epsilon}^*)^\complement,\,y)< \epsilon$, then there exists
$x\in (O_{i,\,\epsilon}^*)^\complement$ such that $d(x,\,y)<\epsilon$.
However, from the definition of $(O_{i,\,\epsilon}^*)^\complement$, we deduce that
$d(x,\,y)\ge \dist(x,\,\pat O_i^*)\ge \epsilon$. This derives a contradiction. Thus, \eqref{2.x6}
holds true, which, together with  Lemma \ref{l2.x7}(ii), implies that
\begin{eqnarray}\label{2.x7}
l_{Q_{i,\,j}}\ge \frac{1}{4\sqrt{n}}
\dist(Q_{i,\,j},\,(O_{i,\,\epsilon}^*)^\complement)\ge \frac{\epsilon}{4\sqrt{n}}.
\end{eqnarray}

If $Q_{i,\,j}\cap \pat O_i^*= \emptyset$ and $Q_{i,\,j}\cap O_i^*\ne \emptyset$,  then
$Q_{i,\,j}\subset (O_i^*)^\circ$, where $(O_i^*)^\circ$ denotes the \emph{interior} of $O_i^*$.
Moreover, for all $y\in Q_{i,\,j}$, take $z\in (O_{i,\,\epsilon}^*)^\complement$ such that
\begin{eqnarray*}
\dist(y,\,(O_{i,\,\epsilon}^*)^\complement)\ge \frac{1}{2} d(y,\,z),
\end{eqnarray*}
which, together with Lemma \ref{l2.x7}(ii),
implies that
\begin{eqnarray}\label{2.x8}
l_{Q_{i,\,j}}\ge \frac{1}{4\sqrt n}\dist(y,\,(O_{i,\,\epsilon}^*)^\complement)\ge \frac{1}{8\sqrt n}d(y,\,z)
\ge \frac{1}{8\sqrt n}\dist(O_{i}^*,\,z)\ge \frac{1}{8\sqrt n}\epsilon.
\end{eqnarray}
Thus, combined \eqref{2.x7} and \eqref{2.x8}, we see that, for all $j\in\Lambda_i$,
\begin{eqnarray*}
l_{Q_{i,\,j}}\ge \frac{1}{8\sqrt n}\epsilon \ \ \text{and hence, for $i\in\mathcal{I}$,}\
 l_i\ge \frac{1}{8\sqrt n}\epsilon,
\end{eqnarray*}
which shows that the above claim is true.

Now, for all $i\in\mathcal{I}$ and $j\in\Lambda_i$, let $B_{i,\,j}$ be the ball having the same
center as $Q_{i,\,j}$ with the radius $5\sqrt n\, l_{Q_{i,\,j}}$.
With the help of the above claim and Lemma \ref{l2.x7}, we conclude
that, for all $i\in \mathcal{I}$,  the sequence $\{B_{i,j}\}_{j\in{\Lambda_i}}$ of balls has the following
properties:

\smallskip

(i) $O_{i}^*\subset \cup_{j\in{\Lambda_i}}B_{i,\,j}$;

(ii) $r_i:=\inf\{r_{B_{i,\,j}},\,j\in{\Lambda_i}\}>0$;

(iii) let $\wz B_{i,\,j}:=\frac{1}{10\sqrt n}B_{i,\,j}$. Then
$\{\wz B_{i,\,j}\}_{j\in{\Lambda_i}}$ are mutually disjoint;

(iv) there exists $M\in(0,\,\fz)$, independent of $i\in \mathcal{I}$,
such that $\sum_{j\in{\Lambda_i}} |B_{i,\,j}|< M|O_i^*|$.

\smallskip

Now, for all $i\in\mathcal{I}$ and $j\in{\Lambda_i}$, let
$$\bdz_{i,\,j}:=\widehat{B}_{i,\,j}\bigcap \lf(Q_{i,\,j}\times (0,\,\fz)\r)\bigcap
\lf(\widehat{O}^*_{i}\setminus \widehat{O}^*_{i+1}\r).$$ For all
$F\in WT^p(\rr^{n+1}_+)$, recall that
\begin{eqnarray*}
\supp F\subset
\bigcup_{i\in\mathcal{I}}\bigcup_{j\in{\Lambda_i}}\lf[\widehat{B}_{i,\,j}\bigcap (Q_{i,\,j}\times (0,\,\fz))\cap
(\widehat{O}^*_{i}\setminus \widehat{O}^*_{i+1})\r]
=:\bigcup_{i\in\mathcal{I}}\bigcup_{j\in{\Lambda_i}}\Delta_{i,\,j}.
\end{eqnarray*}
Moreover, let $\lz_{i,j}:=\wz C2^i |B_{i,j}|^{\frac{1}{p}}$ and
$A_{i,j}:={F\chi_{\bdz_{i,j}}}/{\lz_{i,\,j}}$, where
$\widehat{O}_i^*$ denotes the tent over $O_i^*$ and
$\wz C$ denotes a positive constant independent of $F$, which
will be determined later.
By following the same line of the proof of Theorem \ref{t2.6},
if we choose $\wz C$ large enough and independent of $F$,
we then conclude that $A_{i,\,j}$ is a $T^p(\rr^{n+1}_+)$-atom and, for almost
every $(x,\,t)\in\rr^{n+1}_+$,
\begin{eqnarray*}
F(x,\,t)=\dsum_{i\in\mathcal{I}}\dsum_{j\in{\Lambda_i}}\lz_{i,\,j}A_{i,\,j}(x,\,t)
\end{eqnarray*}
pointwisely and, moreover,
$$\dsup_{i\in\zz}\lf(\sum_{j\in\Lambda_i}|\lz_{i,j}|^p\r)^{\frac{1}{p}}\le C
\|F\|_{WT^p(\mathbb{R}^{n+1}_+)},$$
which, together with the properties of $\{B_{i,\,j}\}_{j\in{\Lambda_i}}$,
completes the proof of Theorem \ref{t2.12}.
\end{proof}

\subsection{The weak Hardy spaces $WH_L^p(\rn)$}\label{s2.3}

\hskip\parindent In this subsection, we study the weak Hardy space
$WH_L^p(\rn)$. First, we recall the definition of the classical
weak Hardy space from  \cite{fso86,li91,lu95,lyj14}. Let $p\in(0,\,1]$ and
$\fai\in\mathcal{S}(\rn)$ support in the unit ball $B(0,\,1)$. The
{\it weak Hardy space} $WH^p(\rn)$ is defined to be the space
$$\lf\{f\in\mathcal{S}'(\rn):\ \|f\|_{WH^p(\rn)}:=
\sup_{\az>0}\lf(\az^p\lf|\lf\{x\in\rn:\ \sup_{t>0}\lf|\fai_t\ast
f(x)\r|>\az\r\}\r|\r)^{1/p}<\fz\r\}.$$

Now, let $L$ satisfy Assumptions $(\mathcal{L})_1$,
$(\mathcal{L})_2$ and $(\mathcal{L})_3$. For all $f\in L^2(\rn)$ and
$x\in\rn$, the \emph{$L$-adapted non-tangential square function
$S_Lf$} is defined by
\begin{eqnarray}\label{2.x10}
S_Lf(x):=\lf\{\iint_{\Gamma(x)}\lf|t^{2k}Le^{-t^{2k}L}f(y)\r|^2
\frac{dy\,dt}{t^{n+1}}\r\}^{1/2},
\end{eqnarray}
where $\Gamma(x)$ is as in \eqref{2.x1}.

Let $p\in(0,\,1]$. A function $f\in L^2(\rn)$ is said to be in
$\mathbb{H}_L^p(\rn)$ if $S_Lf\in L^p(\rn)$; moreover, define
$\|f\|_{H_L^p(\rn)}:=\|S_Lf\|_{L^p(\rn)}.$ The {\it Hardy space}
$H_L^p(\rn)$ associated to $L$ is then defined to be the completion
of $\mathbb{H}_L^p(\rn)$ with respect to the \emph{quasi-norm}
$\|\cdot\|_{H_L^p(\rn)}$ (see \cite{cy}).

Now, we introduce the notion of the weak Hardy space $WH_L^p(\rn)$.

\begin{definition}\label{d2.13}
Let $p\in(0,\,1]$ and $L$ satisfy Assumptions $(\mathcal{L})_1$,
$(\mathcal{L})_2$ and $(\mathcal{L})_3$. A function $f\in L^2(\rn)$
is said to be in $\mathbb{WH}_L^p(\rn)$, if $S_Lf$ belongs to the
weak Lebesgue space $WL^p(\rn)$; moreover, define
$\|f\|_{WH_L^p(\rn)}:=\|S_Lf\|_{WL^p(\rn)}.$ The {\it weak Hardy
space} $WH_L^p(\rn)$ \emph{associated to} $L$ is then defined to be
the completion of $\mathbb{WH}_L^p(\rn)$ with respect to the
\emph{quasi-norm} $\|\cdot\|_{WH_L^p(\rn)}$.
\end{definition}

\begin{remark}\label{r2.14}
We point out that, unlike the Hardy space $H^p(\rn)$, with $p\in(0,\,1]$,
in which the Lebesgue space $L^2(\rn)$ is dense (see, for
example, \cite[Proposition 3.2]{lu95}), the space $L^2(\rn)$ is not dense
in the weak Hardy space $WH^p(\rn)$ in the sense of Fefferman and
Soria \cite{fso86} (see also a very recent
work of He \cite{he13}).  Thus, when $L=-\bdz$, the  weak Hardy space
$WH_{-\bdz}^p(\rn)$ defined as in Definition \ref{d2.13}
coincides with the space
$$\overline{WH^p(\rn)\cap L^2(\rn)}^{\|\cdot\|_{WH^p(\rn)}},$$
namely, the closure of $WH^p(\rn)\cap L^2(\rn)$ on the quasi-norm
$\|\cdot\|_{WH^p(\rn)}$, which is a proper subspace of $WH^p(\rn)$.
\end{remark}

Now, let $T$ be a nonnegative self-adjoint operator in $L^2(\rn)$
satisfying the Davies-Gaffney estimates. It is known that $T$ is a
special case of operators $L$ satisfying Assumptions
$(\mathcal{L})_1$, $(\mathcal{L})_2$ and $(\mathcal{L})_3$. We first
establish the weak atomic decomposition of the weak Hardy space
$WH^p_T(\rn)$.

\begin{definition}[\cite{hlmmy,jy11}]\label{d2.15}
Let $p\in(0,\,1]$, $M\in\nn$ and $B:=B(x_B,\,r_B)$ be a ball with
$x_B\in \rn$ and $r_B\in(0,\,\fz)$. A
function $a\in L^2(\rn)$ is called a $(p,\,2,\,M)_T$-\emph{atom
associated to $B$}, if the following conditions  are satisfied:

\medskip

(i) there exists a function $b$ belonging to the domain of $T^M$,
$D(T^M)$, such that $a=T^Mb$;

(ii) for all $\ell\in\{0,\,\ldots,\,M\}$, $\supp (T^\ell b)\subset
B$;

(iii) for all $\ell\in\{0,\,\ldots,\,M\}$, $\|(r_B^2T)^\ell
b\|_{L^2(\rn)}\le r_B^{2M}|B|^{\frac{1}{2}-\frac{1}{p}}$.
\end{definition}

For all $p\in(0,\,1]$ and $M\in\nn$, let $f\in L^2(\rn)$ and
$\{a_{i,j}\}_{i\in\zz,j\in{\zz_+}}$ be a sequence of
$(p,\,2,\,M)_T$-atoms associated to balls
$\{B_{i,j}\}_{i\in\zz,j\in{\zz_+}}$. The equality
$f=\sum_{i\in\zz,j\in{\zz_+}}\lz_{i,j}a_{i,j}$ holding true in $L^2(\rn)$ is called a
\emph{weak atomic} $(p,\,2,\, M)_T$-\emph{representation of} $f$, if

(i) $\lz_{i,j}:=\wz C2^i|B_{i,j}|^{\frac{1}{p}}$, where $\wz C$ is a positive constant
independent of $f$;

(ii) there exists a positive constant $C_2$, depending only on $f,\,n,\,p$,
$M$ and $\wz C$, such that
\begin{eqnarray*}
\dsup_{i\in\zz}\lf(\dsum_{j\in{\zz_+}}|\lz_{i,j}|^p\r)^{\frac{1}{p}}\le
C_2.
\end{eqnarray*}

The \emph{weak atomic Hardy space} $WH_{T,\mathrm{at},M}^p(\rn)$ is
defined to be the completion of the \emph{space}
$$\mathbb{WH}_{T,\mathrm{at},M}^p(\rn):=
\{f\in L^2(\rn): \ f\ \text{has a  weak atomic $(p,\,2,\,
M)_T$-representation}\}$$ with respect to the \emph{quasi-norm}
$$\|f\|_{WH_{T,\mathrm{at},M}^p(\rn)}:=\inf\lf\{\sup_{i\in\zz}
\lf(\sum_{j\in{\zz_+}}|\lz_{i,j}|^{p}\r)^{\frac{1}{p}}\r\},$$ where
the infimum is taken over all the weak atomic $(p,\,2,\,
M)_T$-representations of $f$ as above.

We have the following weak atomic characterization of $WH^p_T(\rn)$.

\begin{theorem}\label{t2.16}
Let $p\in(0,\,1]$ and $T$ be a nonnegative self-adjoint operator on
$L^2(\rn)$ satisfying the Davies-Gaffney estimates. Assume that
$M\in\nn$ satisfies $M>\frac{n}{2}(\frac{1}{p}-\frac{1}{2})$. Then
$WH_T^p(\rn)=WH_{T,\mathrm{at},M}^p(\rn)$ with equivalent
quasi-norms.
\end{theorem}

To prove this theorem, we need to recall some notions and known
results from \cite{hlmmy}.

Let $C_3\in[1,\,\fz)$. Assume that $\fai\in C_\mathrm{c}^\infty(\rr)$ is
even, $\supp \fai \subset (-C_3^{-1},\,C_3)$, $\fai\ge0$ and there
exists a positive constant $C_4$ such that, for all $t\in
(-\frac{1}{2C_3},\, \frac{1}{2C_3})$, $\fai(t)\ge C_4$.  Let
$M\in\nn$, $\Phi$
be the \emph{Fourier transform} of $\fai$ and
$\Psi(t):=t^{2(M+1)}\Phi(t)$ for all $t\in[0,\,\fz)$.
For $T$ as in Theorem \ref{t2.16},  all $F\in T^2(\rr^{n+1}_+)$ and
$x\in\rn$, define the \emph{operator} $\Pi_{\Psi,T}(F)(x)$ by setting
\begin{eqnarray}\label{2.6}
\Pi_{\Psi,T}(F)(x):=\int_0^\infty
\Psi(t\sqrt{T})(F(\cdot,t))(x)\,\frac{dt}{t}.
\end{eqnarray}

From Fubini's theorem and the quadratic estimates, it follows that
$\Pi_{\Psi,T}$ is bounded from $T^2(\rr^{n+1}_+)$ to $L^2(\rn)$. By
using the finite speed of the propagation of the wave equation and
the Paley-Wiener theorem, Hofmann et al. \cite{hlmmy} proved the following conclusion.

\begin{lemma}[\cite{hlmmy}]\label{l2.17}
Let $p\in(0,\,1]$, $M\in\nn$ and $T$ be a nonnegative self-adjoint
operator on $L^2(\rn)$ satisfying the Davies-Gaffney estimates.
Assume that $A$ is a $T^p(\rr^{n+1}_+)$-atom associated to the ball
$B$ and $\Pi_{\Psi,T}$ is as in \eqref{2.6}. Then there exists a
positive constant $C(M)$, independent of $A$, such that
$[C(M)]^{-1} \Pi_{\Psi,T}(A)$ is a
$(p,2,M)_T$-atom associated to the ball $2B$.
\end{lemma}

We also need the following \emph{superposition principle} on the weak type
estimate.
\begin{lemma}[\cite{stw81}]\label{l2.18}
Let $p\in(0,\,1)$ and $\{f_j\}_{j\in\zz_+}$ be a sequence of
measurable functions. If $\sum_{j\in\zz_+}\lf|\lz_j\r|^p<\fz$ and
there exists a positive constant $C$ such that, for all $j\in\zz_+$ and $\az\in
(0,\fz)$, $|\{x\in\rn:\ |f_j(x)|>\az\}|\le
C\az^{-p}$, then there exists a positive constant $\wz C$, independent of
$\{\lz_j\}_{j\in\zz_+}$ and $\{f_j\}_{j\in\zz_+}$, such that,
for all $\az\in(0,\fz)$,
\begin{eqnarray*}
\lf|\lf\{x\in\rn:\ \lf|\dsum_{j\in\zz_+}\lz_jf_j(x)\r|>\az\r\}\r|\le
\wz C\frac{2-p}{1-p}\az^{-p}\sum_{j\in\zz_+}|\lz_j|^p.
\end{eqnarray*}
\end{lemma}

With these preparations, we now prove Theorem \ref{t2.16}.

\begin{proof}[Proof of Theorem \ref{t2.16}]
In order to prove Theorem \ref{t2.16}, it suffices to show that
$$(WH_T^p(\rn)\cap L^2(\rn))=\mathbb{WH}_{T,\mathrm{at},M}^p(\rn)$$
with
equivalent quasi-norms. We first prove the inclusion that
$$(WH_T^p(\rn)\cap L^2(\rn))\subset
\mathbb{WH}_{T,\mathrm{at},M}^p(\rn).$$
Let $f\in WH^p_T(\rn)\cap
L^2(\rn)$. From its definition and the quadratic estimate, it
follows that $t^2Te^{-t^2T}f\in WT^p(\mathbb{R}^{n+1}_+)\cap
T^2(\mathbb{R}^{n+1}_+)$. By Theorem \ref{t2.6}, there exist
sequences $\{\lz_{i,j}\}_{i\in\zz,j\in\zz_+}\subset\cc$ and
$\{A_{i,j}\}_{i\in\zz,j\in\zz_+}$ of $T^p(\mathbb{R}^{n+1}_+)$-atoms
associated to the balls $\{B_{i,j}\}_{i\in\zz,j\in\zz_+}$ such that
$$t^2Te^{-t^2T}(f)=\sum_{i\in\zz,j\in\zz_+}\lz_{i,j}A_{i,j}$$
pointwisely almost
everywhere in $\rr^{n+1}_+$, $\lz_{i,j}=\wz C2^i|B_{i,j}|^{1/p}$ and
\begin{eqnarray}\label{2.7}
\dsup_{i\in\zz}\lf(\dsum_{j\in{\zz_+}}|\lz_{i,j}|^p\r)^{\frac{1}{p}}
\ls \lf\|t^2Te^{-t^2T}(f)\r\|_{WT^p(\rr^{n+1}_+)}\sim\|f\|_{WH^p(\rr^{n+1}_+)},
\end{eqnarray}
where $\wz C$ is a positive constant independent of $f$.
Moreover, by the bounded $H_\infty$
functional calculus in $L^2(\rn)$, Corollary \ref{c2.9} and the fact
that $\Pi_{\Psi,T}$ is bounded from $T^2(\rr^{n+1}_+)$ to
$L^2(\rn)$, we conclude that there exists a constant $C(\Psi)$,
depending on $\Psi$, such that
\begin{eqnarray}\label{2.x15}
f&&=C(\Psi)\dint_0^\fz\Psi(t\sqrt{T})(t^2Te^{-t^2T})f\,\frac{dt}{t}=
C(\Psi)\dint_0^\fz\Psi(t\sqrt{T})\lf(\sum_{i\in\zz,j\in\zz_+}
\lz_{i,j}A_{i,j}\r)\,\frac{dt}{t}\\
\nonumber &&=C(\Psi)\sum_{i\in\zz,j\in\zz_+}\lz_{i,j}
\dint_0^\fz\Psi(t\sqrt{T})(A_{i,j})\,\frac{dt}{t},
\end{eqnarray}
where the above equalities hold true in $L^2(\rn)$. For all $i\in\zz$ and
$j\in\zz_+$, let
\begin{eqnarray}\label{2.x16}
a_{i,j}:=\int_0^\fz\Psi(t\sqrt{T})(A_{i,j})\,\frac{dt}{t}.
\end{eqnarray}
By Lemma \ref{l2.17}, we see that $a_{i,j}$ is a $(p,2,M)_{T}$-atom
associated to $\{2B_{i,j}\}_{i\in\zz,j\in\zz_+}$ up
to a harmless positive constant. Thus, we conclude that $f$ has a weak atomic
$(p,\,2,\, M)_T$-representation $\sum_{i\in\zz,j\in\zz_+}
\lz_{i,j}a_{i,j}$  and $f\in \mathbb{WH}_{T,\mathrm{at},M}^p(\rn)$.
Moreover, from \eqref{2.7},  we deduce that
\begin{eqnarray*}
\|f\|_{WH_{T,\mathrm{at},M}^p(\rn)}\ls
\lf(\dsup_{i\in\zz}\dsum_{j\in\zz_+}|\lz_{i,j}|^p\r)^{\frac{1}{p}}\ls
\|f\|_{WH_{T}^p(\rn)},
\end{eqnarray*}
which immediately implies that $(WH_{T}^p(\rn)\cap L^2(\rn))\subset
\mathbb{WH}_{T,\mathrm{at},M}^p(\rn)$.

Now we prove the converse, namely,
$$\mathbb{WH}_{T,\mathrm{at},M}^p(\rn) \subset
(WH_{T}^p(\rn)\cap L^2(\rn)).$$
Let $f\in\mathbb{WH}_{T,\mathrm{at},M}^p(\rn)$. From its definition, it
follows that there exists a sequence
$\{a_{i,j}\}_{i\in\zz,j\in\zz_+}$ of $(p,\,2,\,M)_T$-atoms
associated to the balls $\{B_{i,j}\}_{i\in\zz,j\in\zz_+}$ such that
$f=\sum_{i\in\zz,j\in\zz_+}\lz_{i,j}a_{i,j}$ in $L^2(\rn)$ and, for all $i\in\zz$,
\begin{eqnarray}\label{2.8}
\sum_{j\in\zz_+}|\lz_{i,j}|^p\ls \|f\|^p_{WH_{T,\rm{at},M}^p(\rn)},
\end{eqnarray}
where $\{\lz_{i,j}\}_{i\in\zz,j\in\zz_+}:=\{\wz C2^i|B_{i,j}|^{1/p}\}_{i\in\zz,j\in\zz_+}$ and $\wz C$ is a positive constant
independent of $f$.

Given $\az\in(0,\,\fz)$, let $i_0\in\zz$ satisfy that
$2^{i_0}\le \az<2^{i_0+1}$. We then see that
\begin{eqnarray*}
f=\dsum_{i=-\fz}^{i_0}\dsum_{j\in\zz_+}\lz_{i,j}a_{i,j}+
\dsum_{i=i_0+1}^{\fz}\dsum_{j\in\zz_+}\cdots=:f_1+f_2
\end{eqnarray*}
holds true in $L^2(\rn)$.

We first estimate $f_2$.
Let $\wz{B}_{i_0}:=\cup_{i=i_0+1}^\fz
\cup_{j\in\zz_+} 8B_{i,j}$. From the definition of $\lz_{i,j}$ and \eqref{2.8}, we
deduce that
\begin{eqnarray}\label{2.9}
\quad\quad\lf|\wz B_{i_0}\r|&&\ls \dsum_{i=i_0+1}^\fz \dsum_{j\in\zz_+}
\lf|B_{i,j}\r|\ls \dsum_{i=i_0+1}^\fz 2^{-ip} \lf(\dsum_{j\in\zz_+}|\lz_{i,j}|^p\r)
\ls\dsum_{i=i_0+1}^\fz 2^{-ip}\|f\|_{WH^p_{T,\rm{at},M}(\rn)}^p\\
&&\nonumber \ls \frac{1}{\az^p}\|f\|_{WH^p_{T,\rm{at},M}(\rn)}^p.
\end{eqnarray}

Now, for $q\in(0,\, p)$, we write
\begin{eqnarray}\label{2.10}
\hs\hs\hs f_2=
\dsum_{i=i_0+1}^\fz\dsum_{j\in\zz_+}\lf(\wz C2^i\lf|B_{i,j}\r|^{\frac{1}{q}}\r)
\lf(\frac{1}{\wz C}|B_{i,j}|^{\frac{1}{p}-\frac{1}{q}}a_{i,j}\r)=:
\dsum_{i=i_0+1}^\fz\dsum_{j\in\zz_+}\wz{\lz}_{i,j}\wz{a}_{i,j}.
\end{eqnarray}
By \eqref{2.8} and the fact that $q\in(0,\,p)$, we know that
\begin{eqnarray}\label{2.11}
\dsum_{i=i_0+1}^\fz\dsum_{j\in\zz_+}\lf|\wz\lz_{i,j}\r|^q&&\sim
\dsum_{i=i_0+1}^\fz2^{iq}\dsum_{j\in\zz_+}\lf|B_{i,j}\r|\sim
\dsum_{i=i_0+1}^\fz 2^{i(q-p)} \lf(\dsum_{j\in\zz_+}|\lz_{i,j}|^p\r)\\
&& \nonumber \ls \|f\|_{WH^p_{T,\rm{at},M}(\rn)}^p
\dsum_{i=i_0+1}^\fz 2^{i(q-p)} \ls2^{i_0(q-p)} \|f\|_{WH^p_{T,\rm{at},M}(\rn)}^p,
\end{eqnarray}
which, combined with \eqref{2.9}, \eqref{2.10} and Lemma \ref{l2.18},
implies that, to show that $S_T(f_2)\in WL^p(\rn)$,
it suffices to prove that, for all $\az\in(0,\,\fz)$,
$i\in\zz\cap [i_0+1,\,\fz)$ and $j\in\zz_+$,
\begin{eqnarray}\label{2.12}
\lf|\lf\{x\in \lf(8 B_{i,j}\r)^\complement:\ S_T(\wz
a_{i,j})(x)>\az\r\}\r|\ls\frac{1}{\az^q}.
\end{eqnarray}
Indeed, if \eqref{2.12} is true, then, for $N_1\in \zz \cap[i_0+1,\,\fz)$ and
$N_2\in\zz_+$, by Tchebychev's inequality, the sub-linearity
of $S_T$,  Lemma \ref{l2.18} and the $L^2(\rn)$
boundedness of $S_T$, we conclude that
\begin{eqnarray*}
\lf|\lf\{x\in\rn:\
S_T\lf(f_2\r)(x)>\az\r\}\r|&&\le\lf|\lf\{x\in\rn:\
\dsum_{i=i_0+1}^{N_1}\dsum_{j=0}^{N_2}\wz\lz_{i,j}S_T\lf(\wz{a}_{i,j}
\r)(x)>\frac{\az}{2}\r\}\r|\\
&&\hs+\lf|\lf\{x\in\rn:\ S_T\lf(\dsum_{\gfz{N_1+1\le i <\fz}
{\mathrm{or}\,j\ge N_2+1}} \lz_{i,j} {a}_{i,j}
\r)(x)>\frac{\az}{2}\r\}\r|\\
&&\ls \az^{-q}
\dsum_{i=i_0+1}^{N_1}\dsum_{j=0}^{N_2}\lf|\wz\lz_{i,j}\r|^q+
\az^{-2}\lf\|\dsum_{\gfz{N_1+1\le i <\fz}{\mathrm{or}\,j\ge
N_2+1}}\lz_{i,j} a_{i,j}\r\|^2_{L^2(\rn)},
\end{eqnarray*}
where
$$\dsum_{\gfz{N_1+1\le i <\fz}{\mathrm{or}\,j\ge
N_2+1}}:=\dsum_{i=N_1+1}^\fz \dsum_{j\in\zz_+}+
\dsum_{i=i_0+1}^\fz \dsum_{j=N_2+1}^\fz+
\dsum_{i=N_1+1}^\fz \dsum_{j=N_2+1}^\fz.$$
By letting $N_1$ and $N_2\to \fz$, the fact that
$f_2=\sum_{i=i_0+1}^\fz\sum_{j\in\zz_+} \lz_{i,j} a_{i,j}$
holds true in $L^2(\rn)$ and \eqref{2.11}, we see that
\begin{eqnarray*}
\lf|\lf\{x\in\rn:\
S_T\lf(f_2\r)(x)>\az\r\}\r|\ls\az^{-q}
\dsum_{i=i_0+1}^{\fz}\dsum_{j\in\zz_+}\lf|\wz\lz_{i,j}\r|^q\ls
\frac 1{\az^p}\|f\|_{WH_{T,\rm{at},M}^p(\rn)}^p,
\end{eqnarray*}
which is desired.

To prove \eqref{2.12}, from Chebyshev's inequality and H\"older's
inequality, we deduce that
\begin{eqnarray}\label{2.13}
&&\lf|\lf\{x\in \lf(8 B_{i,j}\r)^\complement:\ S_T(\wz
a_{i,j})(x)>\az\r\}\r|\\
&&\nonumber\hs\ls\az^{-q}\dint_{(8B_{i,j})^\complement}
\lf[S_T(\wz a_{i,j})(x)\r]^q\,dx\\ \nonumber
&&\hs\ls\az^{-q}\dsum_{l=4}^{\fz}
\lf\{\dint_{S_l(B_{i,j})}\lf|S_T(\wz
a_{i,j})(x)\r|^2\,dx\r\}^{\frac{q}{2}}\lf|S_l(B_{i,j})\r|^{1-\frac{q}
{2}},
\end{eqnarray}
where $S_l(B_{i,j}):=2^l B_{i,j}\setminus (2^{l-1} B_{i,j})$ for all
$l\in\nn$. For $l\ge 4$, let
$$\mathrm{I}_{l,i,j}:=\lf\{\int_{S_l(B_{i,j})}\lf|S_T(\wz
a_{i,j})(x)\r|^2\,dx\r\}^{\frac{q}{2}},$$
$b_{i,j}:=T^{-M}a_{i,j}$ and
$\wz b_{i,j}:=\lf|B_{i,j}\r|^{\frac{1}{p}-\frac{1}{q}}b_{i,j}$.

By Minkowski's inequality, we write $\mathrm{I}_{l,i,j}$ into
\begin{eqnarray*}
\mathrm{I}_{l,i,j}&&\ls\lf\{\dint_{S_l(B_{i,j})}\dint_0^{r_{B_{i,j}}}
\dint_{|y-x|<t} \lf|t^{2}Te^{-t^2T}\wz a_{i,j}(y,t)\r|^2\,
\frac{dy\,dt}{t^{n+1}}\,dx\r\}^{\frac{q}{2}}\\&&\hs+
\lf\{\dint_{S_l(B_{i,j})}\dint_{r_{B_{i,j}}}^{\frac{\dist(x,\,B_{i,j})}
{4}} \dint_{|y-x|<t} \cdots \,\frac{dy\,dt}{t^{n+1}}\,dx\r\}^{\frac{q}{2}}\\
&&\hs+ \lf\{\dint_{S_l(B_{i,j})}\dint_{\frac{\dist(x,\,B_{i,j})}
{4}}^{\fz} \dint_{|y-x|<t} \cdots\,
\frac{dy\,dt}{t^{n+1}}\,dx\r\}^{\frac{q}{2}}
=:\mathrm{J}_{l,i,j}+\mathrm{K}_{l,i,j}+\mathrm{Q}_{l,i,j}.
\end{eqnarray*}

To estimate $\mathrm{J}_{l,i,j}$, notice that
$\dist(S_l(B_{i,j}),\,B_{i,j})>2^{l-2}r_{B_{i,j}}$, when $l\ge 4$.
Let $$\mathrm{E}_{l,i,j}:=\lf\{x\in\rn: \
\dist(x,\,S_l(B_{i,j}))<r_{B_{i,j}}\r\}.$$ We easily see that
$\dist(\mathrm{E}_{l,i,j},\,B_{i,j})>2^{l-3}r_{B_{i,j}}$, which, together
with Fubini's theorem, Lemma \ref{l2.3} and Definition \ref{d2.15},
implies that there exists a positive constant
$\az_0>\frac{n}{q}(1-\frac{q}{2})$ such that
\begin{eqnarray}\label{2.14}
\mathrm{J}_{l,i,j}&&\ls
\lf\{\dint_0^{r_{B_{i,j}}}\dint_{\mathrm{E}_{l,i,j}}
\lf|t^{2}Te^{-t^2T}\wz
a_{i,j}(y,t)\r|^2\,\frac{dy\,dt}{t}\r\}^{\frac{q}{2}}\\
&&\nonumber\ls
\lf\{\dint_0^{r_{B_{i,j}}}\exp\lf\{-C_1\frac{\lf[\dist(\mathrm{E}_{l,i,j},\,
B_{i,j})\r]^2}{t^2}\r\} \,\frac{dt}{t}\r\}^{\frac{q}{2}}\lf\|\wz
a_{i,j}\r\|_{L^2(B_{i,j})}^q\\
&&\ls2^{-lq\az_0}\lf|B_{i,j}\r|^{\frac{q}{2}-1}
\sim2^{-l[q\az_0-n(1-\frac{q}{2})]}\lf|S_l(B_{i,j})\r|^{\frac{q}{2}-1}.\nonumber
\end{eqnarray}

To estimate $\mathrm{K}_{l,i,j}$, let $\mathrm{F}_{l,i,j}:=
\{x\in\rn:\ \dist(x,\,S_l(B_{i,j}))<\frac{\dist(x,\,B_{i,j})}{4}\}$. It is easy to
see that $\dist(\mathrm{F}_{l,i,j},\,B_{i,j})>2^{l-3}r_{B_{i,j}}$.
Moreover, by Fubini's theorem, Lemma \ref{l2.3} and Definition
\ref{d2.15}, we know that there exists a positive constant
$\az_1\in(\frac{n}{q}(1-\frac{q}{2}),\,2M)$ such that
\begin{eqnarray}\label{2.15}
\hs\hs\hs\quad\mathrm{K}_{l,i,j}&&\ls\lf\{\dint_{r_{B_{i,j}}}^{\fz}
\dint_{\mathrm{F}_{l,i,j}} \lf|t^{2(M+1)}T^{M+1}e^{-t^2T}\wz
b_{i,j}(y)\r|^2\,\frac{dy\,dt}{t^{4M+1}}\r\}^{\frac{q}{2}}\\
\nonumber
&&\ls\lf\{\dint_{r_{B_{i,j}}}^{\fz}\exp\lf\{-C_1\frac{\lf[\dist
\lf(\mathrm{F}_{l,i,j},\,B_{i,j}\r)\r]}{t^2}\r\}\,\frac{dy\,dt}{t^{4M+1}}
\r\}^{\frac{q}{2}}\lf\|\wz
b_{i,j}\r\|_{L^2(B_{i,j})}^q\\
\nonumber
&&\ls\lf\{\dint_{r_{B_{i,j}}}^{\fz}\lf[\frac{t^2}{2^{2l}r_{B_{i,j}}}
\r]^{\az_1}\,\frac{dt}{t^{4M+1}}
\r\}^{\frac{q}{2}}r_{B_{i,j}}^{2Mq}\lf|B_{i,j}\r|^{\frac{q}{2}-1}
\ls 2^{-l[q\az_1-n(1-\frac{q}{2})]}\lf|S_l(B_{i,j})\r|^{\frac{q}{2}-1}.
\end{eqnarray}
Similar to the estimates of \eqref{2.14} and \eqref{2.15}, we obtain
\begin{eqnarray*}
\mathrm{Q}_{l,i,j}&&\ls\lf\{\dint_{2^{l-2}r_{B_{i,j}}}^{\fz} \dint_{\rn}
\lf|(t^2T)^{M+1}e^{-t^2T}\wz
b_{i,j}(y)\r|^2\,\frac{dy\,dt}{t^{4M+1}}\r\}^{\frac{q}{2}}\\
&&\ls\lf\{\dint_{2^{l-2}r_{B_{i,j}}}^{\fz}\frac{dt}{t^{4M+1}}
\r\}^{\frac{q}{2}}\lf\|\wz b_{i,j}\r\|^2_{L^2(B_{i,j})}\ls
2^{-l[2qM-n(1-\frac{q}{2})]}\lf|S_l(B_{i,j})\r|^{\frac{q}{2}-1},
\end{eqnarray*}
which, together with \eqref{2.13}, \eqref{2.14} and \eqref{2.15},
implies that, for all $\az\in(0,\,\fz)$, $i\in\zz\cap [i_0+1,\,\fz)$
and $j\in\zz_+$,
\begin{eqnarray*}
\lf|\lf\{x\in\lf(C_0 B_{i,j}\r)^\complement:\ S_T(\wz
a_{i,j})(x)>\az\r\}\r|\ls \az^{-q}\dsum_{l=4}^\fz\lf(\mathrm{J}_{l,i,j}
+\mathrm{K}_{l,i,j}+
\mathrm{Q}_{l,i,j}\r)\lf|S_l(B_{i,j})\r|^{1-\frac{q}{2}}
\ls\frac{1}{\az^q}.
\end{eqnarray*}
Thus, \eqref{2.12} is true.

We now turn to the estimates  of $f_1$.
For any $r\in (1,\,2)$, let $b\in(0,\,1)$ such that $b<1-\frac{p}{r}$. By H\"older's inequality, we know that
\begin{eqnarray*}
S_T(f_1)\ls \lf\{\dsum_{i=-\fz}^{i_0}2^{ib r'}\r\}^{\frac{1}{r'}}
\lf\{\dsum_{i=-\fz}^{i_0} \lf[S_T\lf(2^{-ib}\dsum_{j\in\zz_+} \lz_{i,\,j}a_{i,\,j}\r)
\r]^r\r\}^{\frac{1}{r}},
\end{eqnarray*}
which, together with $\az\sim 2^{i_0}$, Chebyshev's inequality and the $L^r(\rn)$ boundedness of $S_T$
(which can be deduced from the interpolation of $H_T^p(\rn)$; see \cite[Proposition 9.5]{hlmmy}),
implies that there exists a positive constant $C$, independent of $\az$, $f$
and $x$, such that
\begin{eqnarray}\label{2.xx26}
&&\lf|\lf\{x\in\rn:\ S_T(f_1)(x)>\az\r\}\r|\\
&&\nonumber\hs\le \lf|\lf\{x\in\rn:\
\dsum_{i=-\fz}^{i_0} \lf[S_T\lf(2^{-ib}\dsum_{j\in\zz_+} \lz_{i,\,j}a_{i,\,j}\r)
(x)\r]^r>C2^{i_0(1-b)r}\r\}\r|
\\&&\nonumber\hs\ls \frac{1}{2^{i_0(1-b)r}}
\dsum_{i=-\fz}^{i_0} 2^{-ibr}\dint_\rn \lf|\dsum_{j\in\zz_+} \lz_{i,\,j}a_{i,\,j}(x)\r|^r\,dx
=: \frac{1}{2^{i_0(1-b)r}}
\dsum_{i=-\fz}^{i_0} 2^{-ibr} \lf[\mathrm{I}_i\r]^r.
\end{eqnarray}

Now, let $g\in L^{r'}(\rn)$ satisfying $\|g\|_{L^{r'}(\rn)}\le 1$ such that
\begin{eqnarray*}
\dint_\rn \lf|\dsum_{j\in\zz_+} \lz_{i,\,j}a_{i,\,j}(x)\r|^r\,dx\sim
\lf|\dint_\rn \lf[\dsum_{j\in\zz_+} \lz_{i,\,j}a_{i,\,j}(x)\r]\overline{g(x)}\,dx \r|^r.
\end{eqnarray*}
For any $k\in\nn$, let $S_k(B_{i,\,j}):=2^k B_{i,\,j}\setminus2^{k-1}B_{i,\,j}$ and
$S_0(B_{i,\,j}):=B_{i,\,j}$.
Let $\wz{B}_{i,\,j}:=\frac{1}{10\sqrt{n}}B_{i,\,j}$.
By H\"older's inequality, Definition \ref{d2.19} and the definition of the Hardy-Littlewood maximal operator
$\mathcal{M}$ as in \eqref{2.x2}, we know that
\begin{eqnarray*}
\mathrm{I}_i&&\ls \dsum_{j\in\zz_+} \dint_\rn \lf|\lz_{i,\,j}a_{i,\,j}(x)\overline{ g(x)}\r|\,dx\\
&&\ls \dsum_{j\in\zz_+} \dsum_{k\in\zz_+}
\dint_{S_k(B_{i,\,j})} \lf|\lz_{i,\,j}a_{i,\,j}(x)\overline{ g(x)}\r|\,dx\\
&&\ls \dsum_{j\in\zz_+} \dsum_{k\in\zz_+} 2^{i}|B_{i,\,j}|^{\frac{1}{p}}
\lf\{\dint_{S_k(B_{i,\,j})} \lf|a_{i,\,j}(x)\r|^2\,dx\r\}^{\frac{1}{2}}
\lf\{\dint_{S_k(B_{i,\,j})} \lf|g(x)\r|^2\,dx\r\}^{\frac{1}{2}}\\
&&\ls \dsum_{j\in\zz_+} \dsum_{k\in\zz_+} 2^{i}|B_{i,\,j}|2^{-k\epsilon}
\lf\{\frac{1}{|2^kB_{i,\,j}|}\dint_{2^kB_{i,\,j}} \lf|g(x)\r|^2\,dx\r\}^{\frac{1}{2}}
\\&&\ls \dsum_{j\in\zz_+}  2^{i}|B_{i,\,j}|
\dinf_{x\in \wz B_{i,\,j}}\lf\{\mathcal{M} \lf(\lf|g\r|^2\r)(x)\r\}^{\frac{1}{2}}
\\&&\ls \dsum_{j\in\zz_+}  2^{i}| B_{i,\,j}|
\lf\{\frac{1}{|\wz B_{i,\,j}|} \dint_{\wz B_{i,\,j}}
\lf[\mathcal{M} \lf(\lf|g\r|^2\r)(x)\r]^{\frac{r'}{2}} \,dx\r\}^{\frac{1}{r'}}
\\&&\ls \dsum_{j\in\zz_+}  2^{i}|B_{i,\,j}|^{\frac{1}{r}}
\lf\{\dint_{\wz B_{i,\,j}}\lf[\mathcal{M} \lf(\lf|g\r|^2\r)(x)\r]^{\frac{r'}{2}} \,dx\r\}^{\frac{1}{r'}},
\end{eqnarray*}
which, together with H\"older's inequality, the uniformly bounded overlap of $\{\wz B_{i,\,j}\}_{j\in\zz_+}$
on $j$, the $L^{r'/2}(\rn)$ boundedness of $\mathcal{M}$,
$r<2$ and $\|g\|_{L^{r'}(\rn)}\le 1$, implies that
\begin{equation}\label{2.xx27}
\mathrm{I}_i\ls \lf\{\dsum_{j\in\zz_+} 2^{ir}|B_{i,\,j}|\r\}^{\frac{1}{r}}
\lf\{\dsum_{j\in\zz_+}  \dint_{\wz B_{i,\,j}} \lf[\mathcal{M}\lf(|g|^2\r)\r]^{\frac{r'}{2}}(x)\r\}^{\frac{1}{r'}}
\ls \lf\{\dsum_{j\in\zz_+} 2^{i(r-p)}|\lz_{i,\,j}|^p\r\}^{\frac{1}{r}}.
\end{equation}
Thus, by \eqref{2.xx26}, \eqref{2.xx27}, $b\in(0,\,\frac{r-p}{r})$ and $2^{i_0}\sim \az$, we conclude that
\begin{eqnarray*}
\lf|\lf\{x\in\rn:\ S_T(f_1)(x)>\az\r\}\r|&&\ls \frac{1}{2^{i_0(1-b)r}}
\dsum_{i=-\fz}^{i_0} 2^{-ibr} \lf[\mathrm{I}_i\r]^r\\
&&\ls \frac{1}{2^{i_0(1-b)r}}
\dsum_{i=-\fz}^{i_0} 2^{-ibr} \lf\{\dsum_{j\in\zz_+} 2^{i(r-p)}|\lz_{i,\,j}|^p\r\}\\
&&\ls \frac{1}{2^{i_0(1-b)r}}
\dsum_{i=-\fz}^{i_0} 2^{i(r-p-br)} \|f\|^p_{WH^p_{T,\rm{mol},\ez,M}(\rn)}\\
&&\ls \frac{1}{\az^p} \|f\|^p_{WH^p_{T,\rm{mol},\ez,M}(\rn)},
\end{eqnarray*}
which shows that $f_1\in WH_T^p(\rn)$ and
$$\|f_1\|_{ WH_T^p(\rn)}\ls  \|f\|^p_{WH^p_{T,\rm{mol},\ez,M}(\rn)}.$$

Combining the
estimates for $f_1$ and $f_2$, we then complete the proof of Theorem
\ref{t2.16}.
\end{proof}

\begin{remark}\label{r2.18}
Observe that, in the proof of Theorem \ref{t2.16},
if we use Theorem \ref{t2.12}
to replace Theorem \ref{t2.6} in the argument above \eqref{2.7}, then, for all
$f\in WH_T^p(\rn)\cap L^2(\rn)$, we obtain a weak atomic decomposition of $f$ of the
form
$f=\sum_{i\in\mathcal{I},j\in{\Lambda_i}}\lz_{i,j}a_{i,j}$ in $L^2(\rn)$,
where the index sets $\mathcal{I}$ and $\Lambda_i$ are as in Theorem \ref{t2.12},
$\{a_{i,j}\}_{i\in\mathcal{I},j\in{\Lambda_i}}$ is a sequence of $(p,\,2,\,M)_T$-atoms associated to balls
$\{B_{i,j}\}_{i\in\mathcal{I},j\in{\Lambda_i}}$ and
$\lz_{i,j}:=\wz C2^i|B_{i,j}|^{\frac{1}{p}}$, with $\wz C$ being a positive constant
independent of $f$, satisfies
\begin{eqnarray*}
\dsup_{i\in\mathcal{I}}
\lf(\dsum_{j\in{\Lambda_i}}|\lz_{i,j}|^p\r)^{\frac{1}{p}}\le C
\|f\|_{WH_T^p(\rn)},
\end{eqnarray*}
where $C$ is a positive constant independent of $f$.
\end{remark}

Now, we try to establish the molecular characterization of $WH^p_L(\rn)$.
We first recall the notion of $(p,\,\ez,\,M)_L$-molecules.

\begin{definition}[\cite{cy}]\label{d2.19}
Let $k\in\nn$, $p\in(0,\,1]$, $\ez\in(0,\,\fz)$, $M\in\nn$ and $L$
satisfy Assumptions $(\mathcal{L})_1$, $(\mathcal{L})_2$ and
$(\mathcal{L})_3$. A function $m\in L^2(\rn)$ is called a
$(p,\,\ez,\,M)_L$-{\it molecule}, if there exists a ball
$B:=B(x_B,\,r_B)$, with $x_B\in\rn$ and $r_B\in(0,\,\fz)$, such that

(i) for each $\ell\in\{1,\,\ldots,\,M\}$, $m$ belongs to the range
of $L^\ell$ in $L^2(\rn)$;

(ii) for all $i\in\zz_+$ and $\ell\in\{0,\,\ldots,\,M\}$,
\begin{eqnarray*}
\lf\|\lf(r_B^{-2k}L^{-1}\r)^\ell m\r\|_{L^2(S_i(B))}
\le(2^ir_B)^{n(\frac{1}{2}-\frac{1}{p})}2^{-i\ez}.
\end{eqnarray*}
\end{definition}

\begin{definition}\label{d2.20}
Let $f\in L^2(\rn)$, $\ez\in(0,\,\fz)$, $M\in\zz_+$, $p\in(0,\,1]$ and $L$ satisfy
Assumptions $(\mathcal{L})_1$, $(\mathcal{L})_2$ and
$(\mathcal{L})_3$. Assume that $\{m_{i,j}\}_{i\in\zz,j\in\zz_+}$ is
a sequence of $(p,\,\ez,\,M)_{L}$-molecules associated to balls
$\{B_{i,j}\}_{i\in\zz,j\in\zz_+}$ and
$\{\lz_{i,j}\}_{i\in\zz,j\in\zz_+}\subset\cc$ satisfying
the conditions that

(i) for all $i\in\zz$ and $j\in\zz_+$,
$\lz_{i,j}:=\wz C2^i|B_{i,j}|^{\frac{1}{p}}$, $\wz C$ is a positive constant
independent of $f$;

(ii) there exists a positive constant $C_5$, depending only on $f$,
$n,\,p,\,\ez$ and $M$, such that
\begin{eqnarray*}
\dsup_{i\in\zz}\lf(\dsum_{j\in\zz_+}\lf|\lz_{i,j}
\r|^p\r)^{\frac{1}{p}}\le C_5.
\end{eqnarray*}
Then
$$f=\sum_{i\in\zz,j\in\zz_+}\lz_{i,j}m_{i,j}$$
is called a \emph{weak molecular} $(p,\,\ez,\,M)_L$-{\it representation} of $f$, if
$f=\sum_{i\in\zz,j\in\zz_+} \lz_{i,j}m_{i,j}$ holds true in $L^2(\rn)$.
The {\it weak molecular Hardy space
$WH_{L,\,\mol,\,\ez,\,M}^p(\rn)$} is then defined to be the
completion of the \emph{space}
$$\mathbb{WH}_{L,\,\mol,\,\ez,\,M}^p(\rn)
:=\{f:\ f\ \text{ has a weak molecular}\
(p,\,\ez,\,M)_L\text{-representation}\}$$  with respect to the
\emph{quasi-norm}
\begin{eqnarray*}
\|f\|_{WH_{L,\mol,\ez,M}^p(\rn)}
:=&&\inf\lf\{\dsup_{i\in\zz}\lf(\dsum_{j\in\zz_+}
|\lz_{i,j}|^p\r)^{1/p}:\ \ \ f=\dsum_{i\in\zz,j\in\zz_+}\lz_{i,j}m_{i,j} \
\text{is a weak}\r.\\
&&\hspace{2.8cm}\text{molecular}\
(p,\,\ez,\,M)_L\text{-representation}\Bigg\},
\end{eqnarray*}
where the infimum is taken over all the weak molecular $(p,\,
\ez,\,M)_L$-representations of $f$ as above.
\end{definition}

We also have the following weak molecular characterization of
$WH_L^p(\rn)$.

\begin{theorem}\label{t2.21}
Let $p\in(0,\,1]$, $k\in\nn$, $\ez\in(0,\,\fz)$, $M\in\zz_+$
satisfy $M>\frac{n}{2k}(\frac{1}{p}-\frac{1}{2})$ and $L$
Assumptions $(\mathcal{L})_1$, $(\mathcal{L})_2$ and
$(\mathcal{L})_4$. Then
$WH_L^p(\rn)=WH_{L,\,\mol,\,\ez,\,M}^p(\rn)$ with equivalent quasi-norms.
\end{theorem}

To prove Theorem \ref{t2.21}, we need the following lemma from
\cite[Propostion 2.13]{ccyy}.
\begin{lemma}[\cite{ccyy}]\label{l2.x21}
Let $L$ satisfy Assumptions $(\mathcal{L})_1$, $(\mathcal{L})_2$ and
$(\mathcal{L})_4$, and  $(p_-(L),\,p_+(L))$ be the {range of
exponents $p\in[1,\,\fz]$ for which the holomorphic semigroup
$\{e^{-tL}\}_{t>0}$ is bounded on $L^p(\rn)$}. Then, for all
$p_-(L)<p\le q<p_+(L)$, $S_L$ is bounded on $L^p(\rn)$.
\end{lemma}

We now prove Theorem \ref{t2.21}.
\begin{proof}[Proof of Theorem \ref{t2.21}]
To prove this theorem, it suffices to prove
$$(WH_L^p(\rn)\cap
L^2(\rn))=\mathbb{WH}_{L,\,\mol,\,\ez,\,M}^p(\rn)$$
with equivalent
quasi-norms. The inclusion that $(WH_L^p(\rn)\cap L^2(\rn))\subset
\mathbb{WH}_{L,\,\mol,\,\ez,\,M}^p(\rn)$ follows from a similar
argument to the corresponding part of the proof of Theorem
\ref{t2.16}. We only remark that, in this case, the operator
$\Pi_{\Psi, T}$ defined in \eqref{2.6} is replaced by a new
\emph{operator} $\Pi_{L,M}$ defined by setting, for all $F\in
T^2(\rr^{n+1}_+)$ and $x\in\rn$,
\begin{eqnarray}\label{2.x22}
\Pi_{L,M}(F)(x):=\int_0^\infty
(t^{2k}L)^Me^{-t^{2k}L}(F(\cdot,t))(x)\,\frac{dt}{t},
\end{eqnarray}
where $k$ is as in Assumption $(\mathcal{L})_3$.
It is known, from \cite[Lemma 4.2(ii)]{cy}, that $\Pi_{L,M}$ maps each
$T^p(\rr^{n+1}_+)$-atom into a $(p,\,\ez,\,M)_L$-molecule up to a
harmless positive constant.

Although the proof of the inclusion
$\mathbb{WH}_{L,\,\mol,\,\ez,\,M}^p(\rn)\subset (WH_L^p(\rn)\cap
L^2(\rn))$ is also similar to the corresponding part of Theorem
\ref{t2.16}, in this case, we need more careful
calculations since the lack of the support condition for the molecules.
Let $f\in \mathbb{WH}_{L,\,\mol,\,\ez,\,M}^p(\rn)$. From Definition
\ref{d2.20}, it follows that $f$ has a weak molecular $(p,\,
\ez,\,M)_L$-representation
$f=\sum_{i\in\zz,j\in\zz_+}\lz_{i,j}m_{i,j}$, where
$\{m_{i,j}\}_{i\in\zz,j\in\zz_+}$ is a sequence of
$(p,\,\ez,\,M)_L$-molecules associated to the balls
$\{B_{i,j}\}_{i\in\zz,j\in\zz_+}:=\{B(x_{B_{i,j}},\,r_{B_{i,j}})\}_{i\in\zz,j\in\zz_+}$,
$\lz_{i,j}:=\wz C2^i|B_{i,j}|^{1/p}$, with $\wz C$ being a positive constant
independent of $f$, and
$\sup_{i\in\zz}(\sum_{j\in\zz_+}|\lz_{i,j}|^p)^{1/p}\ls
\|f\|_{WH^p_{L,\rm{mol},\ez,M}}(\rn)$.

For all $\az\in(0,\,\fz)$, let $i_0\in\zz$ satisfy
$2^{i_0}\le\az<2^{i_0+1}$. We write $f$ into
\begin{eqnarray*}
f=\dsum_{i=-\fz}^{i_0}\dsum_{j\in\zz_+}\lz_{i,j}m_{i,j}+
\dsum_{i=i_0+1}^{\fz}\dsum_{j\in\zz_+}\cdots=:f_1+f_2.
\end{eqnarray*}

As in the proof of Theorem \ref{t2.16}, we first estimate $f_2$.
For fixed $q\in(0,\,p)$, by Definition \ref{d2.20}, we
write
\begin{eqnarray}\label{2.16}
f_1=\dsum_{i=-\fz}^{i_0}\dsum_{j\in\zz_+}
\lf(\wz C2^i\lf|B_{i,j}\r|^{\frac{1}{q}}\r)
\lf(\frac{1}{\wz C}\lf|B_{i,j}\r|^{\frac{1}{p}-\frac{1}{q}}m_{i,j}\r)=:
\dsum_{i=-\fz}^{i_0}\dsum_{j\in\zz_+}\wz\lz_{i,j}\wz m_{i,j}.
\end{eqnarray}
Moreover, from Definition \ref{d2.20} and the fact $q<p$, we
deduce that
\begin{eqnarray*}
\dsum_{i=i_0+1}^\fz\dsum_{j\in\zz_+}|\wz{\lz}_{i,j}|^q &&\ls
\dsum_{i=i_0+1}^\fz2^{iq}\dsum_{j\in\zz_+}|B_{i,j}|\ls
\dsum_{i=i_0+1}^\fz2^{i(q-p)}\dsum_{j\in\zz_+}|\lz_{i,j}|^p\\
&&\ls \|f\|_{WH_{L,\,\mol,\,\ez,\,M}^p(\rn)}^p
\dsum_{i=i_0+1}^\fz2^{i(q-p)}
\ls 2^{i_0(q-p)}\|f\|_{WH_{L,\,\mol,\,\ez,\,M}^p(\rn)}^p,
\end{eqnarray*}
which, together with Lemma \ref{l2.18}, implies that,
to show that $S_L(f_2)\in WL^p(\rn)$, it suffices to
show that, for all $\az\in(0,\,\fz)$, $i\in\zz\cap[i_0+1,\,\fz)$ and $j\in\zz_+$,
\begin{eqnarray}\label{2.17}
\lf|\lf\{x\in\rn:\ S_L(\wz
m_{i,j})(x)>\az\r\}\r|\ls\frac{1}{\az^{q}}.
\end{eqnarray}

To prove \eqref{2.17}, by Chebyshev's inequality and H\"older's
inequality, we write
\begin{eqnarray}\label{2.18}
\lf|\lf\{x\in\rn:\ S_L(\wz
m_{i,j})(x)>\az\r\}\r|\ls\dsum_{l=0}^{\fz}2^{-i_0q}\lf\|S_L(\wz
m_{i,j})\r\|_{L^2(S_l(B_{i,j}))}^{q}\lf|S_l(B_{i,j})\r|^{1-\frac{q}{2}}
\end{eqnarray}

For $l\in\{0,\ldots,4\}$, by Fubini's theorem, the $L^2(\rn)$ boundedness
of $S_L$, Definition \ref{d2.19} and
\eqref{2.16}, we conclude that
\begin{eqnarray}\label{2.19}
\lf\|S_L(\wz m_{i,j})\r\|_{L^2(S_l(B_{i,j}))}\ls \lf\|S_L(\wz
m_{i,j})\r\|_{L^2(\rn)}\ls\lf\|\wz
m_{i,j}\r\|_{L^2(\rn)}\ls|B_{i,j}|^{\frac{1}{2}-\frac1q}.
\end{eqnarray}

For $l\ge5$, let
\begin{eqnarray*}
\wz{\mathrm{J}}_{l,i,j}:=\lf\{\dint_{S_l(B_{i,j})}\lf[\dint_0^{r_{B_{i,j}}}
\dint_{|y-x|<t}\lf|t^{2k}Le^{-t^{2k}L}\wz
m_{i,j}(y)\r|^2\,\frac{dy\,dt}{t^{n+1}}\r]\,dx\r\}^{\frac{q}{2}},
\end{eqnarray*}
\begin{eqnarray*}
&&\wz{\mathrm{K}}_{l,i,j}:=\lf\{\dint_{S_l(B_{i,j})}
\lf[\dint_{r_{B_{i,j}}} ^{\dist(x,B_{i,j})/4}
\dint_{|y-x|<t}\lf|(t^{2k}L)^{M+1}e^{-t^{2k}L}\r.\r.\r. \\
&&\lf.\lf.\hspace{6cm}\circ(L^{-M}\wz
m_{i,j})(y)\r|^2\,\frac{dy\,dt}{t^{4kM+n+1}}\r]\,dx\Bigg\}^{\frac{q}{2}}
\end{eqnarray*}
and
\begin{eqnarray*}
\wz{\mathrm{Q}}_{l,i,j}&&:=\lf\{\dint_{S_l(B_{i,j})}
\lf[\dint_{\dist(x,B_{i,j})/4}^{\fz}
\dint_{|y-x|<t}\lf|(t^{2k}L)^{M+1}e^{-t^{2k}L}\r.\r.\r.\\
&&\lf.\lf.\hspace{5cm}\circ L^{-M}\wz
m_{i,j}(y)\r|^2\,\frac{dy\,dt}{t^{4kM+n+1}}\r]\,dx\Bigg\}^{\frac{q}{2}}.
\end{eqnarray*}

To estimate $\wz{\mathrm{J}}_{l,i,j}$, let $\wz{\mathrm{E}}_{l,i,j}:=
\{x\in\rn:\ \dist(x,\,S_l(B_{i,j}))<r_{B_{i,j}}\}$ and
$$\wz{\mathrm{G}}_{l,i,j}:=\{x\in\rn:\
\dist(x,\,\wz{\mathrm{E}}_{l,i,j})<2^{l-3}r_{B_{i,j}}\}.$$
It is easy to see that $\dist\,(\rn\setminus\wz{\mathrm{G}}_{l,i,j},
\wz{\mathrm{E}}_{l,i,j})>2^{l-4}r_{B_{i,j}}$. Moreover, by Fubini's
theorem, the $L^2(\rn)$ boundedness of $S_L$, Assumption $(\mathcal{L})_3$,
Definition \ref{d2.19} and \eqref{2.16}, we see that there exists a
positive constant $\az_2\in(n(\frac{1}{q}-\frac{1}{2}),\,\fz)$ such
that
\begin{eqnarray}\label{2.20}
\hs\hs\quad\wz{\mathrm{J}}_{l,i,j}&&\ls \lf\{\dint_0^{r_{B_{i,j}}}
\dint_{\wz{\mathrm{E}}_{l,i,j}}\lf|t^{2k}Le^{-t^{2k}L}
\lf[\chi_{\wz{\mathrm{G}}_{l,i,j}}+\chi_{\rn\setminus
\wz{\mathrm{G}}_{l,i,j}}\r]\wz
m_{i,j}(y)\r|^2\,\frac{dy\,dt}{t}\r\}^{\frac{q}{2}}\\
\nonumber &&\ls\lf|B_{i,j}\r|^{\frac{q}{p}-1}\lf\|m_{i,j}
\r\|_{L^2(\wz{\mathrm{G}}_{l,i,j})}^q+\lf\|m_{i,j}\r\|_{L^2(\rn)}^q
\lf|B_{i,j}\r|^{\frac{q}{p}-1}\\
\nonumber&&\hs\times\lf\{\dint_0^{r_{B_{i,j}}}\exp
\lf\{-C_1\frac{[\dist(\rn\setminus
\wz{\mathrm{G}}_{l,i,j},\,\wz{\mathrm{E}}_{l,i,j})]^{2k/(2k-1)}}
{t^{2k/(2k-1)}}\r\}\,\frac{dt}{t}\r\}^{\frac{q}{2}}\\
&&\ls\lf\{2^{-l[\ez q-n(\frac{q}{p}-1)]}+2^{-l[q\az_2-n(1-\frac{q}{2})]}\r\}\lf|2^l
B_{i,j}\r|^{\frac{q}{2}-1}\ls2^{-l\ez_0}\lf|2^l
B_{i,j}\r|^{\frac{q}{2}-1},\nonumber
\end{eqnarray}
where $\ez_0:=\max\{\ez
q-n(\frac{q}{p}-1),\,q\az_2-n(1-\frac{q}{2})\}$.

The estimates for $\wz{\mathrm{K}}_{l,i,j}$ and $\wz{\mathrm{Q}}_{l,i,j}$
are deduced from a way similar to that of $\wz{\mathrm{J}}_{l,i,j}$. We omit the details
and only point out that, to obtain the needed convergence, we need
$M>\frac{n}{4k}(\frac{1}{q}-\frac{1}{2})$, which, together with
\eqref{2.18}, \eqref{2.19} and \eqref{2.20}, implies that
\eqref{2.17}. Thus,
\begin{eqnarray*}
\lf|\lf\{x\in\rn:\ S_L(f_2)(x)>\az\r\}\r|\ls\frac{1}{\az^{p}}\|f\|^p_{WH^p_{L,\rm{mol},
\ez,M}(\rn)}.
\end{eqnarray*}

We now turn to the estimates  of $f_1$.
Let $(p_-(L),\,p_+(L))$ be the interior of the maximal interval of the exponents $p$ such that $\{e^{-tL}\}_{t>0}$
is $L^p(\rn)$ bounded. For any $r\in (p_-(L),\,2)$, let
$a\in(0,\,1)$ such that $a<1-\frac{p}{r}$. By H\"older's inequality, we know that
\begin{eqnarray*}
S_L(f_1)\ls \lf\{\dsum_{i=-\fz}^{i_0}2^{ia r'}\r\}^{\frac{1}{r'}}
\lf\{\dsum_{i=-\fz}^{i_0} \lf[S_L\lf(2^{-ia}\dsum_{j\in\zz_+} \lz_{i,\,j}m_{i,\,j}\r)
\r]^r\r\}^{\frac{1}{r}},
\end{eqnarray*}
which, together with $\az\sim 2^{i_0}$, Chebyshev's inequality and the $L^r(\rn)$ boundedness of $S_L$
(see Lemma \ref{l2.x21}),
implies that there exists a positive constant $C$ such that
\begin{eqnarray}\label{2.x26}
&&\lf|\lf\{x\in\rn:\ S_L(f_1)(x)>\az\r\}\r|\\
&&\nonumber\hs\le \lf|\lf\{x\in\rn:\
\dsum_{i=-\fz}^{i_0} \lf[S_L\lf(2^{-ia}\dsum_{j\in\zz_+} \lz_{i,\,j}m_{i,\,j}\r)
\r]^r>C2^{i_0(1-a)r}\r\}\r|
\\&&\nonumber\hs\ls \frac{1}{2^{i_0(1-a)r}}
\dsum_{i=-\fz}^{i_0} 2^{-iar}\dint_\rn \lf(\dsum_{j\in\zz_+} \lz_{i,\,j}m_{i,\,j}\r)^r\,dx
=: \frac{1}{2^{i_0(1-a)r}}
\dsum_{i=-\fz}^{i_0} 2^{-iar} \lf[\mathrm{I}_i\r]^r.
\end{eqnarray}

The estimate for $\mathrm{I}_i$ is similar to that of $\mathrm{I}_i$ in the proof of Theorem
\ref{t2.16}. We also obtain
\begin{eqnarray}\label{2.x27}
\mathrm{I}_i\ls \lf\{\dsum_{j\in\zz_+} 2^{i(r-p)}|\lz_{i,\,j}|^p\r\}^{\frac{1}{r}}.
\end{eqnarray}
Thus, by \eqref{2.x26}, \eqref{2.x27}, $a\in(0,\,\frac{r-p}{r})$ and $2^{i_0}\sim \az$, we conclude that
\begin{eqnarray*}
\lf|\lf\{x\in\rn:\ S_L(f_1)(x)>\az\r\}\r|&&\ls \frac{1}{2^{i_0(1-a)r}}
\dsum_{i=-\fz}^{i_0} 2^{-iar} \lf[\mathrm{I}_i\r]^r\\
&&\ls \frac{1}{2^{i_0(1-a)r}}
\dsum_{i=-\fz}^{i_0} 2^{-iar} \lf\{\dsum_{j\in\zz_+} 2^{i(r-p)}|\lz_{i,\,j}|^p\r\}\\
&&\ls \frac{1}{2^{i_0(1-a)r}}
\dsum_{i=-\fz}^{i_0} 2^{i(r-p-ar)} \|f\|^p_{WH^p_{L,\rm{mol},\ez,M}(\rn)}\\
&&\ls \frac{1}{\az^p} \|f\|^p_{WH^p_{L,\rm{mol},\ez,M}(\rn)},
\end{eqnarray*}
which shows that $f_1\in WH_L^p(\rn)$ and
$$\|f_1\|_{ WH_L^p(\rn)}\ls  \|f\|^p_{WH^p_{L,\rm{mol},\ez,M}(\rn)}.$$

Combining the estimates of $f_1$ and $f_2$, we
then complete the proof of Theorem \ref{t2.21}.
\end{proof}

\begin{remark}\label{r2.22}
(i) As was observed in the proof of Theorem \ref{t2.12}, we know that
Assumption {\bf $(\mathcal{L})_4$} is needed only when proving
$WH_{L,\,\mol,\,\ez,\,M}^p(\rn)\subset WH_L^p(\rn)$. Thus, if
$L$ only satisfies Assumptions $(\mathcal{L})_1$, $(\mathcal{L})_2$ and
$(\mathcal{L})_3$, then $WH_L^p(\rn)\subset WH_{L,\,\mol,\,\ez,\,M}^p(\rn)$.

(ii) Let $(p_-(L),\,p_+(L))$ be the \emph{interval of the exponents $p$}
for which the semigroup $\{e^{-tL}\}_{t>0}$ is  bounded on
$L^p(\rn)$. Assume that $q\in(p_-(L),\,p_+(L))$. Similar to the
notion of $(p,\,\ez,\,M)_L$-molecules as in Definition \ref{d2.19}, we also
define the $(p,\,q,\,\ez,\,M)_L$-\emph{molecule} as $m\in L^q(\rn)$
belonging to the range of $L^\ell$ for all
$\ell\in\{0,\,\ldots,\,M\}$ and satisfying that there exist a ball
$B:=(x_B,\,r_B)$, with $x_B\in\rn$ and $r_B\in(0,\,\fz)$,
and a positive constant $C$ such that, for all
$i\in\zz_+$,
$$\lf\|\lf(r_B^{2k}L\r)^{-\ell}m\r\|_{L^q(S_i(B))}\le C 2^{-i\ez}
\lf|S_i(B)\r|^{n(\frac{1}{q}-\frac{1}{p})}.$$ Moreover, the
corresponding \emph{weak molecular Hardy space}
$WH_{L,\mathrm{mol},q,\ez,M}^p(\rn)$ can be defined analogously to
Definition \ref{d2.20}.

Assume  further that $L$ satisfies Assumption {\bf $(\mathcal{L})_4$}.
By using the method similar to that used in the proofs of
\cite[Proposition 4.2]{jy10} and \cite[Theorem 2.23]{ccyy},
we can also prove the equivalence between
$WH_L^p(\rn)$ and the molecular weak Hardy space
$WH_{L,\mathrm{mol},q,\ez,M}^p(\rn)$. Recall that, in \cite[Proposition 2.10]{ccyy} (see
also Proposition \ref{p2.2}), it was proved that, if $L$ is the
$2k$-order divergence form homogeneous elliptic operator as in \eqref{1.1},
then $L$ satisfies Assumption {\bf $(\mathcal{L})_4$}.
\end{remark}

Let $L$ be as in \eqref{1.1}. Applying the weak molecular characterization, we now study the boundedness of the
associated Riesz transform $\nabla^kL^{-1/2}$ and the fractional
power $L^{-\az/(2k)}$ as follows.

\begin{theorem}\label{t2.23}
Let $k\in\nn$ and $L$ be as in \eqref{1.1}. Then, for all
$p\in(\frac{n}{n+k},\,1]$, the Riesz transform $\nabla^{k}L^{-1/2}$
is bounded from $WH_{L}^p(\rn)$ to $WH^p(\rn)$.
\end{theorem}

\begin{proof}
Let $f\in WH_{{L}}^p(\rn)\cap L^2(\rn)$. From Theorem \ref{t2.21}, we
deduce that there exist sequences $\{\lz_{i,j}\}_{i\in\zz,j\in\zz_+}\subset\cc$
and $\{m_{i,j}\}_{i\in\zz,j\in\zz_+}$ of
$(p,\,\ez,\,M)_L$-molecules such that
$$f=\dsum_{i\in\zz}\dsum_{j\in\zz_+}\lz_{i,j}m_{i,j}$$ in
$L^2(\rn)$ and
$$\dsup_{i\in\zz}\lf(\dsum_{j\in\zz_+}\lf|\lz_{i,j}\r|^p\r)^{1/p}
\sim\|f\|_{WH^p_{L}(\rn)}.$$
By the proof of \cite[Theorem 6.2]{cy}, we know that,
for all $i\in\zz$ and $j\in\zz_+$, $\nabla^{k}{L}^{-1/2}(m_{i,j})$ is
a classical $H^p(\rn)$-molecule up to a harmless positive constant. From this and Remark \ref{r2.14},
together with Theorem \ref{t2.21} in the case $L=-\bdz$, it follows
that $\nabla^{k}{L}^{-1/2}f\in WH^p(\rn)$ and
\begin{eqnarray*}
\lf\|\nabla^{k}{L}^{-1/2}(f)\r\|_{WH^p(\rn)}\ls
\dsup_{i\in\zz}\lf(\dsum_{j\in\zz_+} \lf|\lz_{i,j}\r|^p\r)^{1/p}
\sim\|f\|_{WH_{{L}}^p(\rn)},
\end{eqnarray*}
which, together with a density argument, then completes the proof of
Theorem \ref{t2.23}.
\end{proof}

\begin{theorem}\label{t2.24}
Let $k\in\nn$ and ${L}$ be as in \eqref{1.1}. Then, for all
$0<p<r\le1$ and $\az=n(\frac{1}{p}-\frac{1}{r})$, the fractional
power ${L}^{-\az/(2k)}$ is bounded from $WH_L^p(\rn)$ to
$WH_L^r(\rn)$.
\end{theorem}
\begin{proof}
Similar to the proof of \cite[Theorems 7.2 and 7.3]{jy10}, we
know that ${L}^{-\az/(2k)}$ maps each
$(p,\,\ez,\,M)_L$-\emph{molecule} to a
$(r,\,q,\,\ez,\,M)_L$-\emph{molecule} with
$\az=n(\frac{1}{2}-\frac{1}{q})$, up to a harmless positive constant. This,
together with the fact that $L^{-\az/(2k)}$ is bounded from
$L^2(\rn)$ to $L^q(\rn)$ (see \cite[Lemma 3.10]{ccyy}) and Remark
\ref{r2.22}, then finishes the proof of Theorem \ref{t2.24}.
\end{proof}

\begin{remark}\label{r2.25}
(i) The boundedness of the Riesz transform $\nabla^k L^{-1/2}$,
for $k\in\nn$ and $L$ as in \eqref{1.1}, on the Hardy space
$H_L^p(\rn)$ associated to $L$ is known. Indeed,
it was proved in \cite{cy,cyy} that, for all
$p\in(\frac{n}{n+k},\,1]$, $\nabla^{k}L^{-1/2}$
is bounded from  $H_{L}^p(\rn)$ to the classical Hardy space $H^p(\rn)$
and, at the endpoint case $p=\frac{n}{n+k}$,
$\nabla^{k}L^{-1/2}$ is bounded from
 $H_{L}^{n/(n+k)}(\rn)$ to the classical weak Hardy space $WH^{n/(n+k)}(\rn)$.

 (ii) The boundedness of the fractional power $L^{-\az/(2k)}$,
 for $k\in\nn$ and $L$ as in \eqref{1.1}, on the Hardy space
$H_L^p(\rn)$ associated to $L$ is also well known.
Indeed, it was proved that, for all
$0<p<r\le1$ and $\az=n(\frac{1}{p}-\frac{1}{r})$,  ${L}^{-\az/(2k)}$
is bounded from $H_L^p(\rn)$ to $H_L^r(\rn)$ (see \cite{ccyy,hmm,jy10}).
\end{remark}

\section{The real interpolation of intersections}\label{s4}

\hskip\parindent In this section, we establish a real interpolation theorem on the weak Hardy spaces
$WH_L^p(\rn)$ by showing that
$L^2(\rn)\cap WH_L^p(\rn)$ is an intermediate space between the Hardy spaces $L^2(\rn)\cap H_L^p(\rn)$ for different $p\in(0,\,1]$.
To this end, we first recall some basic results on the real interpolation (see \cite{bl76,tr78} for more details).

Let $(X_0,\,X_1)$ be a \emph{quasi-normed couple}, namely, $X_0$ and $X_1$ are two quasi-normed spaces
which are linearly and continuously imbedded in some Hausdorff topological vector space X. Recall that, for any $f\in X_0+X_1$ and $t\in(0,\,\fz)$,  \emph{Peetre's K-functional} $K(t,\,f;\,X_0,\,X_1)$ is defined by setting,
\begin{eqnarray*}
K(t,\,f;\,X_0,\,X_1):=\dinf\lf\{\|f_0\|_{X_0}+t\|f_1\|_{X_1}:\ f=f_0+f_1,\,f_0\in X_0,\, \ f_1\in X_1\r\}.
\end{eqnarray*}
Then, for all $\tz\in(0,\,1)$ and $q\in[1,\,\fz]$, the \emph{real interpolation space} $(X_0,\,X_1)_{\tz,\,q}$
is defined to be all $f\in X_0+X_1$ such that, for $q\in[1,\,\fz)$,
\begin{eqnarray}\label{4.1}
\|f\|_{\tz,\,q}:=\lf\{\dint_0^\fz \lf[t^{-\tz}K(t,\,f;\,X_0,\,X_1)\r]^q\,\frac{dt}{t}\r\}^{1/q}<\fz
\end{eqnarray}
and
\begin{eqnarray}\label{4.x1}
\|f\|_{\tz,\,\fz}
:=\dsup_{(x,\,t)\in\rr^{n+1}_+} \lf[t^{-\tz}K(t,\,f;\,X_0,\,X_1)\r]<\fz.
\end{eqnarray}

\begin{definition}\label{d4.1}
Let $X$ be a quasi-Banach space whose elements are measurable functions. The space $X$ is said to have the \emph{lattice property} if,
for any $g\in X$ and any measurable function $f$ satisfying $|f|\le |g|$, then $f\in X$.
\end{definition}

Krugljak et al. \cite{kmp99} proved the following interesting result
on the problem of  interpolation of intersections
in the case of Banach spaces.

\begin{proposition}[\cite{kmp99}]\label{p4.1}
Let $X_0,\,X_1$ and $X$ be quasi-Banach spaces whose elements are measurable functions
and have the lattice property. Then, for all $\tz\in(0,\,1)$ and $q\in[1,\,\fz]$,
\begin{eqnarray*}
X\cap (X_0,\,X_1)_{\tz,\,q}=(X\cap X_0,\,X\cap X_1)_{\tz,\,q},
\end{eqnarray*}
where, for any two quasi-Banach spaces $Y$ and $Y_1$, the quasi-norm
in $Y\cap Y_1$ is just the restriction of the quasi-norm from $Y_1$.
\end{proposition}

\begin{proof}
Recall that Proposition \ref{p4.1} in the case of Banach spaces was proved in \cite{kmp99}.
To make it still be valid in the case
of quasi-Banach spaces, we also give a proof based on some ideas from \cite{kmp99}
with some details.

The inclusion that
\begin{eqnarray*}
(X\cap X_0,\,X\cap X_1)_{\tz,\,q}\subset X\cap (X_0,\,X_1)_{\tz,\,q}
\end{eqnarray*}
follows immediately from the fact that, for all $f\in (X\cap X_0,\,X\cap X_1)_{\tz,\,q}$,
\begin{eqnarray*}
 K(t,\,f;\ X_0,\,X_1)\le K(t,\,f;\ X\cap X_0,\,X\cap X_1)
\end{eqnarray*}
and \eqref{4.1}.

We now turn to the proof of the converse inclusion. Let $f\in X\cap (X_0,\,X_1)_{\tz,\,q}$.
By \eqref{4.1}, it suffices to show that, for all $t\in(0,\,\fz)$,
\begin{eqnarray}\label{4.2}
K(t,\,f;\ X\cap X_0,\,X\cap X_1)\ls K(t,\,f;\ X_0,\,X_1).
\end{eqnarray}

To prove \eqref{4.2}, we make the claim: For any $g\in
X\cap (X_0,\,X_1)_{\tz,\,q}$ and $t\in(0,\,\fz)$,
\begin{eqnarray}\label{4.3}
\hs\hs\quad&&K(t,\,g;\ X_0,\,X_1)\\
&&\nonumber\hs=\dinf \lf\{\|g_0\|_{X_0}+t\|g_1\|_{X_1}:\
g=g_0+g_1, \ g_0\in X_0,\, |g_0|\le |g|,\, g_1\in X_1,\, |g_1|\le |g|\r\}.
\end{eqnarray}
Indeed, let
$g=\wz g_0+\wz g_1$ be any decomposition of $g$ satisfying $\wz g_0\in X_0$ and $\wz g_1\in X_1$.
If $|\wz g_0|>|g|$, then using the lattice property of $X_0$, we know $g\in X_0$. Thus,
$g=g+0$ is also a decomposition of $g$ with $g\in X_0$ and $0\in X_1$. Moreover, for all $t\in(0,\,\fz)$,
\begin{eqnarray*}
\|g\|_{X_0}+t\|0\|_{X_1}<\|\wz g_0\|_{X_0}+t\|\wz g_1\|_{X_1},
\end{eqnarray*}
which, together with the definition of Peetre's K-functional, shows \eqref{4.3} is true.
This immediately proves the above claim.

We now continue the proof of \eqref{4.2} by estimating $K(t,\,f;\ X_0,\,X_1)$.
Let $f=f_0+f_1$ be any decomposition of $f$ satisfying $f_0\in X_0$ and $f_1\in X_1$.
By the above claim, we may assume that $|f_0|\le |f|$ and $|f_1|\le |f|$, which, together
with the lattice property of $X$, implies that $f_0\in X_0\cap X$ and $f_1\in X_1\cap X$.
This, together with the definition of Peetre's K-functional, shows \eqref{4.2} holds. Thus,
$(X\cap X_0,\,X\cap X_1)_{\tz,\,q}\subset X\cap (X_0,\,X_1)_{\tz,\,q}$, which completes the proof of
Proposition \ref{p4.1}.
\end{proof}

\begin{remark}\label{r4.1}
Let $X_0,\,X_1$ and $X$ be some quasi-Banach spaces.
The \emph{problem of  interpolation of intersections} asks the question
that, under which conditions, does we have the equality
\begin{eqnarray*}
X\cap (X_0,\,X_1)_{\tz,\,q}=(X\cap X_0,\,X\cap X_1)_{\tz,\,q}.
\end{eqnarray*}
The answer for the above problem is still unknown in the general case
(see \cite{kmp99} and the references cited therein
for more details).  Proposition \ref{p4.1} shows that, if  $X_0,\,X_1$ and $X$
consist of measurable functions and have the lattice property, then
the problem of  interpolation of intersections for $X_1$, $X_2$ and $X$ has a positive
answer.
\end{remark}

\begin{remark}\label{r4.4}
For all $p\in(0,\,\fz)$, let $T^p(\rr^{n+1}_+)$ be the tent space as in \eqref{2.xx2}.
Observe that, for all $0<p_0<p_1<\fz$, $T^{p_0}(\rr^{n+1}_+)$, $T^{p_1}(\rr^{n+1}_+)$
and $T^{2}(\rr^{n+1}_+)$ have the lattice properties. Thus, by the real interpolation of $T^p(\rr^{n+1}_+)$
(see \cite{cms85} in the case $p\in[1,\,\fz)$
and \cite{ccy14} in the case $p\in(0,\,1)$) and Proposition \ref{p4.1},
we conclude that, for all $0<p_0<p_1<\fz$, $\tz\in(0,\,1)$ and $q\in[1,\,\fz]$,
\begin{eqnarray*}
&&\lf(T^{2}(\rr^{n+1}_+)\cap T^{p_0}(\rr^{n+1}_+),\,T^{2}(\rr^{n+1}_+)\cap T^{p_1}(\rr^{n+1}_+)\r)_{\tz,\,q}\\
&&\hs=
T^{2}(\rr^{n+1}_+)\cap (T^{p_0}(\rr^{n+1}_+),\,T^{p_1}(\rr^{n+1}_+))_{\tz,\,q},
\end{eqnarray*}
where $p\in(p_0,\,p_1)$ satisfies $\frac{1}{p}=\frac{1-\tz}{p_0}+\frac{\tz}{p_1}$.
In particular, if $q=\fz$, then $$(T^{p_0}(\rr^{n+1}_+),\,T^{p_1}(\rr^{n+1}_+))_{\tz,\,\fz}=
WT^p(\rr^{n+1}_+).$$

\end{remark}

Now, let $L$ satisfy Assumptions $(\mathcal{L})_1$,
$(\mathcal{L})_2$ and $(\mathcal{L})_4$, and $H_L^p(\rn)$ be the Hardy space
associated to $L$ defined as in Section \ref{s2.3}.
Our main result of this section is as follows.

\begin{theorem}\label{t4.5}
Let $0<p_0<p_1\le 1$, $\tz\in(0,\,1)$ and $L$ satisfy Assumptions $(\mathcal{L})_1$,
$(\mathcal{L})_2$ and $(\mathcal{L})_4$. Then
\begin{eqnarray*}
\lf(L^{2}(\rn)\cap H_L^{p_0}(\rn),\,L^{2}(\rn)\cap H_L^{p_1}(\rn)\r)_{\tz,\,\fz}=
L^{2}(\rn)\cap WH_L^p(\rn),
\end{eqnarray*}
where $p\in(0,\,1]$ satisfies $\frac{1}{p}=\frac{1-\tz}{p_0}+\frac{\tz}{p_1}$.
 \end{theorem}

\begin{remark}\label{r4.6}
Recall that, in \cite{frs74}, Fefferman et al. showed that, for all
$0<p_0<p_1\le 1$, $\tz\in(0,\,1)$ and $p\in(0,\,1]$ satisfies $\frac{1}{p}=\frac{1-\tz}{p_0}+\frac{\tz}{p_1}$,
\begin{eqnarray}\label{4.4}
\lf(H^{p_0}(\rn),\,H^{p_1}(\rn)\r)_{\tz,\,\fz}=WH^p(\rn).
\end{eqnarray}
They proved \eqref{4.4} by only considering the Schwartz functions in $S(\rn)$.
However, as was pointed out in \cite{fso86,he13}, $S(\rn)$ may not dense in $WH^p(\rn)$.
Thus, there is a gap in their proof of \eqref{4.4}. Indeed, Fefferman et al. \cite{frs74}
proved
\begin{eqnarray*}
\lf(S(\rn)\cap H^{p_0}(\rn),\,S(\rn)\cap H^{p_1}(\rn)\r)_{\tz,\,\fz}=S(\rn)\cap WH^p(\rn).
\end{eqnarray*}
Thus, Theorem \ref{t4.5} is a generalization of this result.
\end{remark}

To prove Theorem \ref{t4.5}, we need the following lemma.

\begin{lemma}\label{l4.7}
Let $p\in(0,\,1]$ and $L$ satisfy Assumptions $(\mathcal{L})_1$,
$(\mathcal{L})_2$ and $(\mathcal{L})_4$. Then, for all $M\in\nn$ satisfying
$M>\frac{n}{2k}(\frac{1}{p}-\frac{1}{2})$,
\begin{eqnarray*}
\Pi_{L,\,M}\lf(T^2(\rr^{n+1}_+)\cap WT^p(\rr^{n+1}_+)\r)=L^2(\rn)\cap WH_L^p(\rn),
\end{eqnarray*}
where $\Pi_{L,\,M}$ is the operator defined as in \eqref{2.x22}.
\end{lemma}

\begin{proof}
We first prove
\begin{eqnarray}\label{4.x5}
L^2(\rn)\cap WH_L^p(\rn)\subset \Pi_{L,\,M}\lf(T^2(\rr^{n+1}_+)\cap WT^p(\rr^{n+1}_+)\r).
\end{eqnarray}
Indeed, for any $f\in L^2(\rn)\cap WH_L^p(\rn)$, by the bounded $H_\fz$-functional calculus,
we know that
\begin{eqnarray}\label{3.x5}
f=\Pi_{L,\,M}\circ Q_{t,\,L}(f),
\end{eqnarray}
where $Q_{t,\,L}:=t^{2k}Le^{-t^{2k}L}$.
Moreover, by the definition of $WH_L^p(\rn)$ and $k$-Davies-Gaffney estimates, we know that
$Q_{t,\,L}(f)\in T^2(\rr^{n+1}_+) \cap WT^p(\rr^{n+1}_+)$. This, together with \eqref{3.x5},
implies that $f\in \Pi_{L,\,M}(T^2(\rr^{n+1}_+)\cap WT^p(\rr^{n+1}_+))$. Thus, \eqref{4.x5} is true.

We now prove
\begin{eqnarray}\label{4.7}
 \Pi_{L,\,M}\lf(T^2(\rr^{n+1}_+)\cap WT^p(\rr^{n+1}_+)\r)\subset L^2(\rn)\cap WH_L^p(\rn).
\end{eqnarray}
Indeed, for any $g\in \Pi_{L,\,M}\lf(T^2(\rr^{n+1}_+)\cap WT^p(\rr^{n+1}_+)\r)$, we know that there exists
$$G\in T^2(\rr^{n+1}_+)\cap WT^p(\rr^{n+1}_+)$$ such that $g=\Pi_{L,\,M}(G)$. Using the weak atomic
decomposition of $WT^p(\rr^{n+1}_+)$ (see Theorem \ref{t2.6}),
we know that there exist $\{\lz_{i,\,j}\}_{i\in\zz,\,j\in\zz_+}\subset \cc$ and $\{A_{i,\,j}\}_{i\in\zz,\,j\in\zz_+}$
of $T^p(\rr^{n+1}_+)$-atoms such that
\begin{eqnarray*}
G=\dsum_{i\in\zz,\,j\in\zz_+}\lz_{i,\,j}A_{i,\,j}
\end{eqnarray*}
holds in $T^2(\rr^{n+1}_+)$ and almost everywhere in $\rr^{n+1}_+$. Moreover,
\begin{eqnarray}\label{4.8}
\dsup_{i\in\zz}\lf\{\dsum_{j\in\zz_+}|\lz_{i,\,j}|^p\r\}^{\frac{1}{p}}\sim \|G\|_{WT^p(\rr^{n+1}_+)}.
\end{eqnarray}

Using the boundedness of $\Pi_{L,\,M}$ from $T^2(\rr^{n+1}_+)$ to $L^2(\rn)$ (which can be deduced from the quadratic estimates, since $L$ has a bounded $H_\fz$ functional calculus),
we know that
\begin{eqnarray*}
g=\dsum_{i\in\zz,\,j\in\zz_+} \lz_{i,\,j}  \Pi_{L,\,M}(A_{i,\,j})
\end{eqnarray*}
in $L^2(\rn)$, where, by the argument below \eqref{2.x22}, we know that
$\Pi_{L,\,M}(A_{i,\,j})$ is a $(p,\,\ez,\,M)_L$-molecule up to a
harmless positive constant.  This, together with Theorem \ref{t2.21}, implies that
$g\in L^2(\rn)\cap WH_L^p(\rn)$ and hence \eqref{4.7}
is true, which completes the proof of Lemma \ref{l4.7}.
\end{proof}

We now turn to the proof of Theorem \ref{t4.5}.

\begin{proof}[Proof of Theorem \ref{t4.5}]
Let $\Pi_{L,\,M}$ be defined as in \eqref{2.x22} and $Q_{t,\,L}:=t^{2k}Le^{-t^{2k}L}$.
By the bounded $H_\fz$-functional calculus, we know that
\begin{eqnarray}\label{4.5}
f=\Pi_{L,\,M}\circ Q_{t,\,L}(f),
\end{eqnarray}
which implies that $(L^2(\rn)\cap H_L^{p_0}(\rn),\,L^2(\rn)\cap H_L^{p_1}(\rn))$ is a retract of
$$(T^2(\rr^{n+1}_+)\cap T^{p_0}(\rr^{n+1}_+),\,T^2(\rr^{n+1}_+)\cap T^{p_1}(\rr^{n+1}_+)),$$
namely, there exist two linear bounded operators  $$Q_{t,\,L}: \ L^2(\rn)\cap H_L^{p_i}(\rn)\rightarrow T^2(\rr^{n+1}_+)
\cap T^{p_i}(\rr^{n+1}_+)$$ and  $$\Pi_{L,\,M} : \ T^2(\rr^{n+1}_+)\cap T^{p_i}(\rr^{n+1}_+)\rightarrow
L^2(\rn)\cap H_L^{p_i}(\rn)$$
such that $\Pi_{L,\,M} \circ Q_{t,\,L} = I$ on each $ L^2(\rn)\cap H_L^{p_i}(\rn)$,
where $i\in\{0,\, 1\}$ . Thus, by \cite[Lemma 7.11]{kmm},
we see that
\begin{eqnarray*}
&&(L^2(\rn)\cap H_L^{p_0}(\rn),\,L^2(\rn)\cap H_L^{p_1}(\rn))_{\tz,\,\fz}\\
&&\hs=\Pi_{L,\,M}\lf[(T^2(\rr^{n+1}_+)\cap T^{p_0}(\rr^{n+1}_+),\,T^2(\rr^{n+1}_+)\cap T^{p_0}(\rr^{n+1}_+))_{\tz,\,\fz}\r],
\end{eqnarray*}
which, together with Lemma \ref{l4.7} and Remark \ref{r4.4},
implies that
\begin{eqnarray*}
&&(L^2(\rn)\cap H_L^{p_0}(\rn),\,L^2(\rn)\cap H_L^{p_1}(\rn))_{\tz,\,\fz}\\
&&\hs=\Pi_{L,\,M}\lf[(T^2(\rr^{n+1}_+)\cap T^{p_0}(\rr^{n+1}_+),\,T^2(\rr^{n+1}_+)\cap T^{p_0}(\rr^{n+1}_+))_{\tz,\,\fz}\r]\\
&&\hs=\Pi_{L,\,M}\lf[T^2(\rr^{n+1}_+)\cap \lf(T^{p_0}(\rr^{n+1}_+),\,T^{p_1}(\rr^{n+1}_+)\r)_{\tz,\,\fz}\r]\\
&&\hs=\Pi_{L,\,M}\lf[T^2(\rr^{n+1}_+)\cap WT^p(\rr^{n+1}_+)\r]
=L^2(\rn)\cap WH_L^p(\rn).
\end{eqnarray*}
This finishes the proof of Theorem \ref{t4.5}.
\end{proof}

\section{The dual space of $WH_L^p(\rn)$}\label{s3}

\hskip\parindent In this section, letting $L$ be nonnegative self-adjoint and satisfy the Davies-Gaffney
estimates, we study the dual space of
$WH_L^p(\rn)$.
It turns out that the dual of $WH_L^p(\rn)$ is some
weak Lipschitz space, which can be defined via the
mean oscillation over some bounded open sets.

Before giving the definition of weak Lipschitz spaces, we first
introduce a class of coverings of all bounded open sets, which
is motivated by the subtle covering, appearing in the proof of
Theorem \ref{t2.12}, of the level sets of $\mathcal{A}$-functionals,
obtained via the Whitney decomposition lemma.

\begin{definition}\label{d3.1}
Let $\boz$ be a bounded open set in $\rn$ and ${\Lambda}$ an index set.
A family $\vec{B}:=\{B_j\}_{j\in{\Lambda}}$ of open balls is said to be in the
\emph{class} $\mathcal{W}_{\boz}$, if

(i) $\boz\subset \cup_{j\in{\Lambda}} B_j$;

(ii) $r:=\inf_{j\in\Lambda} \{r_{B_j}\}>0$;

(iii) letting $\wz B_{j}:=\frac{1}{20\sqrt n}B_{j}$, then
$\{\wz B_{j}\}_{j\in\Lambda}$ are mutually disjoint;

(iv) there exists a positive constant $M_0$ such that $\sum_{j\in\Lambda}|B_j|< 2 M_0|\boz|$.
\end{definition}

From the argument below \eqref{2.5}, it follows that, for any bounded
open set $\boz$ with $|\boz|\in(0,\,\fz)$, $\mathcal{W}_{\boz}\ne\emptyset$.

Now, let $\az\in[0,\,\fz)$, $\ez\in(0,\fz)$, $M\in\zz_+$
satisfy $M>\frac{1}{2}(\az+\frac{n}{2})$ and $L$ be nonnegative self-adjoint in $L^2(\rn)$
and satisfy the Davies-Gaffney estimates. The \emph{space} $\mathcal{M}_{\az,L}^{\ez,M}(\rn)$ is
defined to be the space of all functions $u$ in $L^2(\rn)$
satisfying the following two conditions:

(i) for all $\ell\in\{0,\ldots,\,M\}$, $L^{-\ell}u\in L^2(\rn)$;

(ii) letting $Q_0$ be the \emph{unit cube} with
its center at the origin, then
\begin{eqnarray}\label{4.xx1}
\|u\|_{\mathcal{M}_{\az,L}^{\ez,M}(\rn)}:=\sup_{j\in\zz_+}
2^{-j(\frac{n}{2}+\ez+\az)}\sum_{\ell=0}^M
\lf\|L^{-\ell}u\r\|_{L^2(S_j(Q_0))}<\fz.
\end{eqnarray}

Let $\mathcal{M}_{\az,L}^{M,*}(\rn):=\cap_{\ez\in(0,\,\fz)}
(\mathcal{M}_{\az,L}^{\ez,M}(\rn))^*$. For all $r\in(0,\,\fz)$, let
\begin{eqnarray}\label{3.4}
\mathcal{A}_{r}:=\lf(I-e^{-r^{2}L}\r)^M.
\end{eqnarray}
For any $f\in\mathcal{M}_{\az,L}^{M,*}(\rn)$, bounded open set $\boz$
and $\mathcal{N}\in(n(\frac{1}{p}-\frac{1}{2}),\,\fz)$, let
$\mathcal{O}_{\mathcal{N}}(f,\,\boz)$ be the \emph{mean
oscillation of} $f$ \emph{over} $\boz$ defined by
\begin{eqnarray}\label{3.5}
\quad\quad
\mathcal{O}_{{\mathcal{N}}}
(f,\,\boz) &&:=\dsup_{i\in\zz_+}
\dsup_{\vec{B}\in\mathcal{W}_{\boz}}2^{-i\mathcal{N}}
\lf\{\frac{1}{|\boz|^{1+\frac{2\az}{n}}}
\dint_{S_i(\vec{B})}\lf|\mathcal{A}_{r}f(x)\r|^2\,dx
\r\}^{\frac{1}{2}},
\end{eqnarray}
where $r:=\inf_{j\in\Lambda}\{r_{B_j}\}$ and, for $i\in\nn$,
\begin{eqnarray}\label{4.xx}
S_i(\vec{B}):=\lf(\bigcup_{j\in\Lambda}2^iB_j\r)\setminus \lf(\bigcup_{j\in\Lambda}2^{i-1}B_j\r)
\end{eqnarray}
and
$S_0(\vec{B}):=\cup_{j\in\Lambda}B_j$.
By the Davies-Gaffney estimates, we know that the above integral is
well defined.

For all $\dz\in(0,\,\fz)$, define
\begin{equation}\label{3.6}
\omega_{{\mathcal{N}}}(\dz)
:=\sup_{|\boz|=\dz}\mathcal{O}_{{\mathcal{N}}}(f,\,\boz).
\end{equation}
From its definition, it follows that $\omega$ is a decreasing function on
$(0,\,\fz)$. Indeed, assume that $f\in\mathcal{M}_{\az,L}^{M,*}(\rn)$,
$\dz\in(0,\,\fz)$ and $\boz$ is an open set satisfying $|\boz|=\dz$. For
all $\vec{B}\in \mathcal{W}_{\boz}$, we
know, from Definition \ref{d3.1}, that there exists a positive constant
$\wz C\in(1,\,\fz)$ such that $\sum_{j\in\Lambda}|B_j|<
\frac{2M_0}{\wz C}|\boz|$. This implies that, for all open sets
$\wz\boz\subset \boz$ satisfying that
$|\wz\boz|=:\wz\dz\in[\frac{\dz}{\wz C},\,\dz)$,
$$\dsum_{j\in\Lambda}|B_j|< 2M_0|\wz\boz|,$$
which immediately shows that $\vec{B}\in
\mathcal{W}_{\wz\boz}$. Thus, from its definition,  it follows that
$$\mathcal{O}_{\mathcal{N}}(f,\,\boz) \le
\mathcal{O}_{\mathcal{N}}(f,\,\wz\boz)$$ and hence
$\omega_{\mathcal{N}}(\dz)\le
\omega_{\mathcal{N}}(\wz\dz)$ for all $\wz\dz\in
[\frac{\dz}{\wz C},\,\dz)$. This
implies that $\omega$ is decreasing.

Now, we introduce the notion of the weak Lipschitz space associated to $L$.

\begin{definition}\label{d3.2}
Let $\az\in[0,\,\fz)$, $\ez\in(0,\fz)$, $M\in\zz_+$
satisfy $M>\frac{1}{2}(\az+\frac{n}{2})$, and $L$ be
nonnegative self-adjoint and satisfy the Davies-Gaffney estimates.
The \emph{weak Lipschitz space}
$W\Lambda_{L,\mathcal{N}}^\az(\rn)$
is defined to be the space of all
functions $f\in \mathcal{M}_{\az,L}^{M,*}(\rn)$ such that
\begin{eqnarray*}
\|f\|_{W\Lambda_{L,\mathcal{N}}^\az(\rn)}
:=\dint_0^\fz \frac{\omega_{\mathcal{N}}
(\dz)}{\dz}\,d\dz<\fz,
\end{eqnarray*}
where $\omega_{\mathcal{N}}(\dz)$
for $\dz\in(0,\fz)$ is as in \eqref{3.6} and $\mathcal{N}\in(n(\frac{1}{p}-\frac{1}{2}),\,\fz)$.
\end{definition}

We also introduce the notion of the resolvent weak Lipschitz space. To
this end, we need another class of open sets as follows, which is a slight
variant of the class $\mathcal{W}_{\boz}$.

\begin{definition}\label{d3.3}
Let $\boz$ be a bounded open set in $\rn$ and $\Lambda$
an index set. A family
$\vec{B}:=\{B_j\}_{j\in\Lambda}$ of open sets is said to be in the
\emph{class} $\wz{\mathcal{W}}_{\boz}$, if

(i) $\boz\subset \cup_{j\in\Lambda} B_j$;

(ii) $r:=\inf_{j\in\Lambda} \{r_{B_j}\}>0$;

(iii) letting $\wz B_{j}:=\frac{1}{10\sqrt n}B_{j}$, then
$\{\wz B_{j}\}_{j\in\Lambda}$ are mutually disjoint;

(iv) there exists a positive constant $M_0$ such that
$\sum_{j\in\Lambda}|B_j|< {2M_0}|\boz|$.

\end{definition}

It is easy to see that, for any bounded open set $\boz$,
$\wz{\mathcal{W}}_{\boz}\subset {\mathcal{W}}_{\boz}$.
For all $\az\in[0,\,\fz)$, and $M\in\nn$ satisfying $M>\frac{1}{2}(\az+\frac{n}{2})$
and $r\in(0,\,\fz)$, let
\begin{eqnarray}\label{3.7}
\mathcal{B}_r:=\lf[I-\lf(I+r^{2}L\r)^{-1}\r]^{M}.
\end{eqnarray}
Assume that $f\in \mathcal{M}_{\az,L}^{M,*}(\rn)$ and $\dz\in(0,\,\fz)$.
Let
\begin{eqnarray*}
\mathcal{O}_{\mathrm{res}, \wz{\mathcal{N}}_{\boz}}
(f,\,\boz)&&:=\dsup_{i\in\zz_+}
\dsup_{\vec{B}\in\wz{\mathcal{W}}_{\boz}} 2^{-i\mathcal{N}}
\lf\{\frac{1}{|\boz|^{1+\frac{2\az}{n}}} \dint_{S_i(\vec{B})}
\lf|\mathcal{B}_{r}f(x)\r|^2 \,dx\r\}^{\frac{1}{2}},
\end{eqnarray*}
where $\mathcal{N}\in(n(\frac{1}{p}-\frac{1}{2}),\,\fz)$, $r:=\inf_{j\in\Lambda}\{r_{B_{j}}\}$
and $S_i(\vec{B})$ is as in \eqref{4.xx}.
Let also, for any $\dz\in(0,\,\fz)$,
\begin{equation}\label{3.8}
\omega_{\mathrm{res},\mathcal{N}}
(\dz):=\sup_{|\boz|=\dz}
\mathcal{O}_{\mathrm{res},\mathcal{N}}(f,\,\boz).
\end{equation}

\begin{definition}\label{d3.4}
Let $\az\in[0,\,\fz)$, $\ez\in(0,\fz)$, $M\in\zz_+$
satisfy $M>\frac{1}{2}(\az+\frac{n}{2})$, $\mathcal{N}\in(n(\frac{1}{p}-\frac{1}{2}),\,\fz)$,
and $L$ be nonnegative self-adjoint and satisfy the Davies-Gaffney estimates.
The \emph{resolvent weak Lipschitz space}
$W\Lambda_{L,\mathrm{res},{\mathcal{N}}}^\az(\rn)$
is then defined to be the
space of all functions $f\in \mathcal{M}_{\az,L}^{M,*}(\rn)$ satisfying
\begin{eqnarray*}
\|f\|_{W\Lambda_{L,\mathrm{res},{\mathcal{N}}}^\az(\rn)}
:=\dint_0^\fz\frac{\omega_{\mathrm{res},{\mathcal{N}}}
(\dz)}{\dz}\,d\dz<\fz.
\end{eqnarray*}
\end{definition}

We have the following relationship between the weak
Lipschitz space and the resolvent weak Lipschitz space.

\begin{proposition}\label{p3.5}
Let $\az\in[0,\,\fz)$ and $L$ be nonnegative self-adjoint
and satisfy the Davies-Gaffney estimates. Let ${\mathcal{N}}\in(n(\frac{1}{p}-\frac{1}{2}),\,\fz)$.
Then $W\Lambda_{L,{\mathcal{N}}}^{\az}(\rn)
\subset W\Lambda_{L,\mathrm{res},{\mathcal{N}}}^{\az}(\rn)$.
\end{proposition}

\begin{proof}
We prove this proposition by showing that, for all $f\in
W\Lambda_{L,{\mathcal{N}}}^{\az}(\rn)$,
\begin{eqnarray*}
\lf\|f\r\|_{W\Lambda_{L, \mathrm{res}, {\mathcal{N}}}^{\az}(\rn)}
 \ls\lf\|f\r\|_{W\Lambda_{L,{\mathcal{N}}}^{\az}(\rn)}.
\end{eqnarray*}

By an argument similar to that used in the proof of
\cite[(3.42)]{hmm},
we see that
\begin{eqnarray*}
f&&=2^M\lf[r^{-2}\dint_{r}^{2^{1/2}r}
s\lf(I-e^{-s^{2}L}\r)^M\,ds+\dsum_{\ell=1}^M
\binom{M}{\ell}r^{-2}L^{-1}e^{
-\ell r^{2}L}\r.\\
\nonumber &&\lf.\hs\circ \lf(I-e^{-r^{2}L}\r)
\lf(\dsum_{i=0}^{\ell-1}e^{-ir ^{2}L}\r)\r]^Mf\\
\nonumber &&=2^M\lf[r^{-2}\dint_{r}^{2^{1/2}r}s
\lf(I-e^{-s^{2}L}\r)^M\,ds\r]^Mf\\
\nonumber &&\hs+2^M\binom{M}{1}^{\ell_1}
\cdots\binom{M}{M-1}^{\ell_{M-1}} \dsum_{
\gfz{\ell_0+\cdots+\ell_M=M}{\ell_0<M,\ell_M<M}}
\binom{M}{\ell_0,\ldots,\ell_M}\\
&&\hs\times
\lf[r^{-2}\dint_{r}^{ 2^{1/2}r^{2}}s
\lf(I-e^{-s^{2}L}\r)^M\,ds\r]^{\ell_0}\\
\nonumber
&&\hs\times\cdots\times\lf[r^{-2}L^{-1}e^{-\ell
r^{2}L}\lf(I-e^{-r^{2} L}\r)\lf(\dsum_{i=0}^{M-1}
e^{-ir^{2}L}\r)\r]^{\ell_M}f\\
\nonumber&&\hs+2^M\lf[r^{-2}L^{-1}e^{-\ell
r^{2}L}\lf(I-e^{-r^{2}L}\r)\lf(\dsum_{i=0}^{M-1}e^{-ir^{2}L}\r)\r]^{M}f\\
\nonumber&&=:\mathrm{A}_{0}f+\dsum_{\gfz{\ell_0+\cdots+\ell_M=M}
{\ell_0<M,\ell_M<M}}\mathrm{A}_{\ell_0,\ldots,\ell_M}f+
\mathrm{A}_{M}f,
\end{eqnarray*}
where $\binom{M}{\ell}$ denotes the binomial coefficients.

For all $i\in\zz_+$, $\vec{B}\in\wz{\mathcal{W}}_{\boz}$ and
$\mathcal{N}\in(n(\frac{1}{p}-\frac{1}{2}),\,\fz)$, we first estimate
\begin{eqnarray*}
\mathrm{D}:=2^{-i\mathcal{N}}
\lf\{\frac{1}{|\boz|^{1+\frac{2\az}{n}}}
\dint_{S_i(\vec{B})}\lf| \mathcal{B}_{r}
\mathrm{A}_{\ell_0,\ldots,\ell_M}f(x)\r|^2\,dx\r\}^{\frac{1}{2}}.
\end{eqnarray*}
Let $$\wz{A}_{\ell_0,\ldots,\ell_M}:=\lf(r^{2}
L\r)^{\ell_1+\cdots+\ell_M}
\mathrm{A}_{\ell_0,\ldots,\ell_M}\lf(r^{-2}
\int_{r}^{2^{1/2}r }s\mathcal{A}_{s}\,ds\r)^{-1}.$$ From the
functional calculus and the fact that $\{e^{-tL}\}_{t>0}$ satisfies the
Davies-Gaffney estimates, we deduce that $\mathcal{B}_{r}$ satisfies
the Davies-Gaffney estimates with $t\sim r^{2}$, namely,
there exists a positive constant $C_1$ such that, for all
closed sets $E$ and $F$ in $\rn$, $t\in(0,\,\fz)$ and $f\in
L^2(\rn)$ supported in $E$,
\begin{eqnarray*}
\lf\|\mathcal{B}_{r}f\r\|_{L^2(F)}\ls
\exp\lf\{-C_1\frac{\lf[\dist(E,\,F)\r]^{2}}
{r^{2}}\r\}\|f\|_{L^2(E)}.
\end{eqnarray*}
Moreover, from Lemmas \ref{l2.3} and \ref{l2.4},
it follows that $\wz{A}_{\ell_0,\ldots,\ell_M}$ also
satisfies the Davies-Gaffney estimates with $t\sim r^{2}$.
Similarly, $(r^{-2}L^{-1})^{\ell_1+\cdots+\ell_M}\mathcal{B}_{r}$ also satisfies the Davies-Gaffney estimates with $t\sim r^2$.

Now, let
\begin{eqnarray}\label{3.x6}
F_{s}:=\frac{1}{r^{2}}
\dint_{r}^{2^{1/2} r} s
\lf(I-e^{-s^{2}L}\r)^M f\,ds.
\end{eqnarray}
By Minkowski's inequality, the Davies-Gaffney estimates
and H\"older's inequality, we know that
there exists a positive constant $\az_3\in(0,\,\fz)$ such that
\begin{eqnarray*}
\mathrm{D}&&\ls
2^{-i\mathcal{N}}
\lf\{\frac{1}{|\boz|^{1+\frac{2\az}{n}}}
\dint_{S_i(\vec{B})} \lf|\lf[
\lf(r^{-2}L^{-1}\r)^{\ell_1+\cdots+\ell_M} \mathcal{B}_{r}
\r] \wz{\mathrm{A}}_{\ell_0,\ldots,\ell_{M}}
F_{s}(x)\r|^2\,dx\r\}^{\frac{1}{2}}\\ \nonumber&&\ls
2^{-i\mathcal{N}}
\lf\{ \frac{1}{|\boz|^{1+\frac{2\az}{n}}}
\lf[\dsum_{\ell\in\zz_+}\lf\{\dint_{S_i(\vec{B})}\lf| \lf[
\lf(r^{-2}L^{-1}\r)^{\ell_1+\cdots+\ell_M} \mathcal{B}_{r}\r]\r.\r.\r.\r.\\
&&\hs\lf.\lf.\lf.\times \wz
{\mathrm{A}}_{\ell_0,\ldots,\ell_{M}}\lf(\chi_{S_\ell(\vec{B})}
F_{s}\r)(x)\r|^2\,dx\Bigg\}^{1/2} \r]^2\r\}^{\frac{1}{2}}\\
\nonumber&&\ls 2^{-i \mathcal{N}}
\lf\{\frac{1}{|\boz|^{1+\frac{2\az}{n}}}
\lf\{\dsum_{\ell=i-1}^{i+1}\lf[\dint_{S_\ell(\vec{B})}\lf|
F_{s}(x)\r|^2\,dx\r]^{1/2}\r.\r.\\
&&\lf.\lf.\hs +\dsum_{\gfz{\ell\in \zz_+}{\ell\ge i+2 \, \text{or}\, \ell\le i-2}}
2^{-\ell(\mathcal{N}+\az_3)}\lf[\dint_{S_\ell(\vec{B})}
\lf|F_{s}(x)\r|^2\,dx\r]^{1/2}\r\}^2\r\}^{\frac{1}{2}}\\ \nonumber&&\ls
\lf\{ \frac{1}{|\boz|^{1+\frac{2\az}{n}}} \lf[\dsum_{\ell=i-1}^{i+1}
2^{-2\ell\mathcal{N}}+\dsum_{\gfz{\ell\in\zz_+}{\ell\ge i+2\, \text{or}\,\ell\le
i-2}}2^{-2\ell(\mathcal{N}+\az_3)}\r]
\dint_{S_\ell(\vec{B})}\lf|F_{s}(x)\r|^2\,dx\r\}^{\frac{1}{2}}\\
\nonumber&&\ls \dsup_{i\in\zz_+} 2^{-i\mathcal{N}}
\lf\{\frac{1}{|\boz|^{1+\frac{2\az}{n}}} \dint_{ S_i(\vec{B})}
\lf|F_{s}(x)\r|^2\,dx\r\}^{\frac{1}{2}}=:\wz{\mathrm{D}}.
\end{eqnarray*}

We now estimate $\wz{\mathrm{D}}$. By \eqref{3.x6},
Minkowski's integral inequality and the mean
value theorem for integrals, we see that
there exists a positive constant
$\wz r\in(r,\,2^{\frac{1}{2}}r)$ such that
\begin{eqnarray}\label{3.x7}
\wz{\mathrm{D}}&&\ls \dsup_{i\in\zz_+}
2^{-i\mathcal{N}} \lf\{\frac{1}{|\boz|^{1+\frac{2\az}{n}}} \dint_{S_{i}(\vec{B})}
\lf| \frac{1}{r^2}\dint_{r}^{2^{1/2} r}s\lf(I-e^{-s^{2}L}\r)^M
f(x)\,ds \r|^2\,dx\r\}^{\frac{1}{2}}\\
&&\nonumber\ls \dsup_{i\in\zz_+}
2^{-i \mathcal{N}} \frac{1}{r^{2}} \dint_{r}^{2^{1/2} r}
s\lf\{\frac{1}{|\boz|^{1+\frac{2\az}{n}}} \dint_{S_i(\vec{B})}\lf|\lf(I-e^{-s^{2}L}\r)^M
 f(x)\r|^2\,dx\r\}^{\frac{1}{2}}\,ds\\
 &&\nonumber\sim \dsup_{i\in\zz_+}
2^{-i\mathcal{N}}\lf\{\frac{1}{|\boz|^{1+\frac{2\az}{n}}}\dint_{S_i(\vec{B})}\lf|
\lf(I-e^{-\wz {r}^{2}L}\r)^Mf(x)\r|^2\,dx\r\}^{\frac{1}{2}}.
\end{eqnarray}

Moreover, since $\vec{B}\in \wz{W}_{\boz}$, if
let $C_0\in(1,\,2^{1/2})$ satisfying $\wz r=C_0r$, then it
is easy to see that $C_0\vec{B}:=\{C_0B_j\}_{j\in\Lambda}\in W_{\boz}$ and
\begin{eqnarray*}
r_{C_0\vec{B}}:=\dinf_{j\in\Lambda}\lf\{r_{C_0B_j}\r\}=C_0r=\wz r,
\end{eqnarray*}
which, together with \eqref{3.x7} and the fact that $S_i(C_0\vec{B})\subset \cup_{j=i-1}^{i+1}
S_j(C_0\vec{B})$, implies that
\begin{eqnarray*}
\wz{\mathrm{D}}&&\ls \dsup_{i\in\zz_+}
2^{-i\mathcal{N}}\lf\{\frac{1}{|\boz|^{1+\frac{2\az}{n}}} \dsum_{j=i-1}^{i+1}
\dint_{S_j(\vec{B})}\lf|\lf(I-e^{-\wz {r}^{2}L}\r)^Mf(x)\r|^2\,dx\r\}^{\frac{1}{2}}\\
&&\ls \dsup_{i\in\zz_+}\dsup_{\vec{B}\in W_\boz}
2^{-i\mathcal{N}}\lf\{\frac{1}{|\boz|^{1+\frac{2\az}{n}}}
\dint_{S_i(\vec{B})}\lf|\lf(I-e^{-r^{2}L}\r)^Mf(x)\r|^2\,dx\r\}^{\frac{1}{2}},
\end{eqnarray*}
where $r:=\dinf_{j\in\Lambda}\{r_{B_j}\}$. This, together with \eqref{3.5}, implies that
\begin{eqnarray*}
\mathrm{D}\ls \mathcal{O}_\mathcal{N}(f,\,\boz).
\end{eqnarray*}
Thus, for all $\ell_0+\cdots+\ell_M=M$
satisfying $\ell_0<M$ and $\ell_M<M$, we have
\begin{eqnarray*}
\mathcal{O}_{\mathrm{res}, {\mathcal{N}}_{\boz}}(\mathrm{A}_{\ell_0,\ldots,\ell_M}f,\,\boz)
=\sup_{i\in\zz}\dsup_{\vec{B}\in \wz{W}_\boz}\mathrm{D}\ls
\mathcal{O}_{\mathcal{N}}(f,\,\boz).
\end{eqnarray*}
This implies that $\|\mathrm{A}_{\ell_0,\ldots,\ell_M}f\|_{W\Lambda_{L,\mathrm{res},
\mathcal{N}}^{\az}(\rn)} \ls
\|f\|_{W\Lambda_{L,\mathcal{N}}^{\az}(\rn)}$.
Similarly, we also have
$$\|\mathrm{A}_{0}f\|_{W\Lambda_{L,\mathrm{res},
\mathcal{N}}^{\az}(\rn)} \ls
\|f\|_{W\Lambda_{L,\mathcal{N}}^{\az}(\rn)}$$
and
$\|\mathrm{A}_{M}f\|_{W\Lambda_{L,\mathrm{res},
\mathcal{N}}^{\az}(\rn)} \ls
\|f\|_{W\Lambda_{L,\mathcal{N}}^{\az}(\rn)}$,
which completes the proof of Proposition \ref{p3.5}.
\end{proof}

Now we state the main result of this section.

\begin{theorem}\label{t3.6}
Let $p\in(0,\,1]$ and $L$ be a nonnegative self-adjoint
operator in $L^2(\rn)$ satisfying the Davies-Gaffney estimates.
Then, for all ${\mathcal{N}}\in(n(\frac{1}{p}-\frac{1}{2}),\,\fz)$,
$(WH_L^p(\rn))^*=W\Lambda_{L,{\mathcal{N}}}^{n(\frac{1}{p}-1)}(\rn)$.
\end{theorem}

To prove Theorem \ref{t3.6}, we first introduce the following notion of weak
Carleson measures.

\begin{definition}\label{d3.x3}
Let $\boz$ be a bounded open set in $\rn$ and $\Lambda$
an index set. A family
$\vec{B}:=\{B_j\}_{j\in\Lambda}$ of open sets is said to be in the
\emph{class} $\wz{\wz{\mathcal{W}}}_{\boz}$, if

(i) $\boz\subset \cup_{j\in\Lambda} B_j$;

(ii) $r:=\inf_{j\in\Lambda} \{r_{B_j}\}>0$;

(iii) letting $\wz B_{j}:=\frac{1}{10\sqrt n}B_{j}$, then
$\{\wz B_{j}\}_{j\in\Lambda}$ are mutually disjoint;

(iv) there exists a positive constant $M_0$ such that $\sum_{j\in\Lambda}|B_j|< {2M_0}|\boz|$.

(v) letting $Q_{j}$ for any $j\in\Lambda$ be the closed cube having the same center as $B_{j}$ with the length
$\frac{r_{B_j}}{5\sqrt n}$, then $\{Q_j\}_{j\in\Lambda}$ are mutually disjoint and
$\{2Q_j\}_{j\in \Lambda}$ have bounded overlap. Moreover,
$\boz\subset\cup_{j\in \Lambda}Q_i$.

\end{definition}

From the argument below \eqref{2.5}, it follows that, for any bounded
open set $\boz$ with $|\boz|\in(0,\,\fz)$, $\wz{\wz{\mathcal{W}}}_\boz\ne\emptyset$.
Moreover, $\wz{\wz{\mathcal{W}}}_{\boz}\subset{\wz{\mathcal{W}}}_{\boz}
\subset \mathcal{W}_{\boz}$.

Now, let $\az\in[0,\,\fz)$ and $\mu$ be a positive
measure on $\rr^{n+1}_+$. For all bounded open sets $\boz$, let
\begin{eqnarray*}
\mathcal{C}_\az(\mu,\,\boz):=\sup_{\vec{B}\in
\wz{\wz{\mathcal{W}}}_{\boz}}\lf\{\frac{1}{|\boz|^{1+\frac{2\az}{n}}}
\mu\lf(\bigcup_{j\in\Lambda}
\widehat{B}_{j}\cap (Q_j\times (0,\,\fz))\r)\r\}^{\frac{1}{2}},
\end{eqnarray*}
where $\wz{\wz{\mathcal{W}}}_{\boz}$ is as in Definition \ref{d3.x3}.

For all $\dz\in(0,\,\fz)$, let
\begin{eqnarray}\label{3.9}
\omega_{\mathcal{C}}(\dz):=\sup_{|\boz|=\dz}
\mathcal{C}_\az(\mu,\,\boz).
\end{eqnarray}
 Then $\mu$ is called a \emph{weak
Carleson measure of order $\az$}, denoted by $\mathcal{C}_\az$, if
\begin{eqnarray*}
\|\mu\|_{\mathcal{C}_\az}:=\dint_0^\fz\frac{\omega_{\mathcal{C}}(\dz)}
{\dz}\,d\dz<\fz.
\end{eqnarray*}

Recall that, if $L$ is nonnegative self-adjoint and satisfies the Davies-Gaffney
estimates, then $L$ satisfies the
\emph{finite speed propagation property for solutions of the corresponding wave equation},
namely, there exists a positive constant $C_0$ such that
\begin{eqnarray}\label{3.x9}
(\cos(t \sqrt L)f_1 , f_2) =0
\end{eqnarray}
for all closed sets $U_1$, $U_2$, $0 <C_0 t < \dist(U_1,\, U_2)$,
 $f_1\in L^2(U_1)$ and
$f_2\in L^2(U_2)$ (see, for example, \cite{hlmmy}).
Moreover, we have the following conclusion, which can be deduced from
\cite[Lemma 3.5]{hlmmy} with a slight modification, the details being omitted.

\begin{lemma}[\cite{hlmmy}]\label{l3.8}
Let $L$ be nonnegative self-adjoint and satisfy the Davies-Gaffney
estimates. Assume that $\fai\in C_0^\fz(\rn)$ is even, $\supp \fai
\subset (-C_1^{-1},\,C_1^{-1})$, where $C_1\in(0,\,1)$ is a positive constant.
Let $\Phi$ be the Fourier transform of $\fai$. Then, for all $k\in\zz_+$
and $t\in(0,\,\fz)$, the kernel $K_{(t^2L)^k\Phi(t\sqrt{L})}(x,\,y)$
of the operator $(t^2L)^k\Phi(t\sqrt{L})$ satisfies
\begin{eqnarray}
\supp K_{(t^2L)^k\Phi(t\sqrt{L})}(x,\,y)\subset
\lf\{(x,\,y)\in \rn\times\rn:\ d(x,\,y)\le \frac{C_0}{C_1}t\r\}.
\end{eqnarray}
\end{lemma}

\begin{proposition}\label{p3.7}
Let  $\az\in[0,\,\fz)$, $M\in\nn$ satisfy that
$M>\frac{n}{2}(\frac{1}{p}-\frac{1}{2})$, ${\mathcal{N}}\in (n(\frac{1}{p}-\frac{1}{2}),\,\fz)$,
and $L$ be a nonnegative self-adjoint operator satisfying the Davies-Gaffney estimates.
Assume that $\fai\in C_0^\fz(\rn)$ is even, $\supp \fai
\subset (-C_1^{-1},\,C_1^{-1})$ satisfies $\frac{C_0}{C_1}\in(0,\,\frac{1}{10 \sqrt n})$,
where $C_0$ is as in \eqref{3.x9}.
Let $\Phi$ be the Fourier transform of $\fai$.
Then, for all $f\in W\Lambda_{L,{\mathcal{N}}}^\az(\rn)$,
\begin{eqnarray*}
\mu_f(x,\,t):=\lf|\lf(t^{2}L\r)^{M}\Phi(t\sqrt{L})f(x)\r|^2\frac{dx\,dt}{t},
\quad \forall \ (x,\,t)\in\rr^{n+1}_+,
\end{eqnarray*}
is a weak Carleson measure of order $\az$. Moreover, there exists
a positive constant $C$ such that, for all $f\in
W\Lambda_{L,{\mathcal{N}}}^\az(\rn)$,
 $$\|\mu_f\|_{\mathcal{C}_\az}
\le C \|f\|_{W\Lambda_{L,{\mathcal{N}}}^\az(\rn)}.$$
\end{proposition}

\begin{proof}
Let $f\in W\Lambda_{L,{\mathcal{N}}}^\az(\rn)$.
For all bounded open sets
$\boz$ and $\vec{B}=\{B_j\}_{j\in \Lambda} \in \wz{\wz{\mathcal{W}}}_{\boz},$
by the definition of weak Carleson measures,  we need
to estimate
\begin{eqnarray*}
\mathrm{I}:=\lf\{\frac{1}{|\boz|^{1+\frac{2\az}{n}}}
\mu_f\lf(\bigcup_{j\in\Lambda}
\widehat{B}_{j}\cap (Q_j\times(0,\,\fz))\r)\r\}^{\frac{1}{2}}.
\end{eqnarray*}
To this end, let $r:=\inf_{j\in\Lambda}\{r_{B_j}\}$ and
$\mathcal{B}_r$ be as in \eqref{3.7}. By Minkowski's
inequality, we write
\begin{eqnarray*}
\mathrm{I}&&\sim \lf\{\frac{1}{|\boz|^{1+\frac{2\az}{n}}}
\dsum_{{j}\in\Lambda}
\iint_{\widehat{B}_{j}\cap (Q_j\times(0,\,\fz))}
\lf|\lf(t^{2}L\r)^{M}\Phi(t\sqrt{L})f(x)\r|^2 \frac{dx\,dt}{t} \r\}^{\frac{1}{2}}\\
&&\ls\lf\{\frac{1}{|\boz|^{1+\frac{2\az}{n}}}
\dsum_{{j}\in\Lambda} \iint_{\widehat{B}_{j}\cap (Q_j\times(0,\,\fz))}
\lf|\lf(t^{2}L\r)^{M} \Phi(t\sqrt{L})\mathcal{B}_{r}
f(x)\r|^2 \frac{dx\,dt}{t} \r\}^{\frac{1}{2}}\\
\nonumber &&\hs\hs+\lf\{\frac{1}{|\boz|^{1+\frac{2\az}{n}}}
\dsum_{{j}\in \Lambda} \iint_{\widehat{B}_{j}\cap (Q_j\times(0,\,\fz))}
\lf|\lf(t^{2}L\r)^{M} \Phi(t\sqrt{L})\lf(I-\mathcal{B}_{r} \r)
f(x)\r|^2\frac{dx\,dt}{t}\r\}^{\frac{1}{2}}=:\mathrm{I}_1+
\mathrm{I}_2.
\end{eqnarray*}

To estimate $\mathrm{I}_1$, by Minkowski's inequality again, we obtain
\begin{eqnarray*}
\mathrm{I}_1&&\ls \lf\{\frac{1}{|\boz|^{1+\frac{2\az}{n}}}
\dsum_{{j}\in\Lambda} \lf[\dsum_{i\in\zz_+}
\lf(\iint_{\widehat{B}_{j}\cap (Q_j\times(0,\,\fz))}\lf|\lf(t^{2}L\r)^{M}
 \Phi(t\sqrt{L})\lf(\chi_{S_i(Q_j)}\mathcal{B}_{r}
f\r)(x)\r|^2\r.\r.\r.\\
&&\lf.\lf.\lf.\hs\times \frac{dx\,dt}{t}\r)^{1/2}\r]^2 \r\}^{\frac{1}{2}}.
\end{eqnarray*}

Now, let $\wz{S_i}(B_j):=2^{i+1}B_j\setminus 2^{i-2}B_j$. By Lemma \ref{l3.8} with
$\frac{C_0}{C_1}\in(0,\,\frac{1}{10 \sqrt n})$, $t\in(0,\,r_{B_{i,\,j}})$ and
$Q_j=\frac{1}{5\sqrt{n}}B_j$, we know that
$\supp\{(t^2L)^M\Phi(t\sqrt{L})(\chi_{S_i(Q_j)}\mathcal{B}_rf)\}\subset
\wz{S_i}(Q_j)$,
which, together with Definitions  \ref{d3.x3}, \ref{d3.3} and the quadratic estimates, implies that
\begin{eqnarray*}
\mathrm{I}_1&&\ls \lf\{\frac{1}{|\boz|^{1+\frac{2\az}{n}}}
\dsum_{j\in\Lambda} \lf[\dsum_{i=0}^1\lf\{\iint_{Q_j\cap (0,\,r_{B_j})} \lf|\lf(t^{2}L\r)^{M}
\Phi(t\sqrt{L})\lf(\chi_{S_i(Q_j)}\mathcal{B}_rf\r)(x)\r|^2\,\frac{dx\,dt}{t}\r\}^{\frac{1}{2}}\r]^2\r\}^{\frac{1}{2}}\\
&&\ls \lf\{\frac{1}{|\boz|^{1+\frac{2\az}{n}}}
\dsum_{j\in\Lambda} \dsum_{i=0}^1\dint_\rn\lf[\lf\{\dint_{0}^\fz \lf|\lf(t^{2}L\r)^{M}
\Phi(t\sqrt{L})\lf(\chi_{S_i(Q_j)}\mathcal{B}_rf\r)(x)\r|^2\,\frac{dt}{t}\r\}^{\frac{1}{2}}\r]^2\,dx\r\}^{\frac{1}{2}}\\
&&\ls \lf\{\frac{1}{|\boz|^{1+\frac{2\az}{n}}}
\dsum_{j\in\Lambda} \dint_{2Q_j}\lf|\mathcal{B}_rf(x)\r|^2\,dx\r\}^{\frac{1}{2}}
\ls \lf\{\frac{1}{|\boz|^{1+\frac{2\az}{n}}} \dint_{\cup_{j\in\Lambda}
2Q_j}\lf|\mathcal{B}_rf(x)\r|^2\,dx\r\}^{\frac{1}{2}}\\
&&\ls \dsup_{i\in\zz}\dsup_{\vec{B}\in \wz{W}_{\boz}}
2^{-i\mathcal{N}}\lf\{\frac{1}{|\boz|^{1+\frac{2\az}{n}}} \dint_{S_i(\vec{B})}
\lf|\mathcal{B}_rf(x)\r|^2\,dx\r\}^{\frac{1}{2}}
\ls \mathcal{O}_{\mathrm{res},{\mathcal{N}}}(f,\,\boz),
\end{eqnarray*}
which is desired.

The estimate of $\mathrm{I}_2$ is similar to that for $\mathrm{I}_1$. In
this case, we need the following operator equality that, for all
$r\in(0,\,\fz)$,
$$\lf(I-\lf[I-\lf(I+r^{2}L\r)^{-1}\r]^M\r)\lf[I-\lf(I+r^{2}L
\r)^{-1}\r]^{-M}=\dsum_{\ell=1}^M \frac{M!}{(M-\ell)!\ell !}
\lf(r^{2}L\r)^{-\ell},$$ the details being omitted.

By combining the estimates for $\mathrm{I}_1$ and $\mathrm{I}_2$, we conclude that
$\mathcal{C}_\az(\mu_f,\,\boz)\ls\mathcal{O}_{\mathrm{res},
{\mathcal{N}}}(f,\,\boz)$,
which, together with Proposition \ref{p3.5}, shows that $\mu_f$ is a
weak Carleson measure of order $\az$. This finishes the proof of
Proposition \ref{p3.7}.
\end{proof}

We now turn to the proof of Theorem \ref{t3.6}.

\begin{proof}[Proof of Theorem \ref{t3.6}]
We first prove that
$$W\Lambda_{L,{\mathcal{N}}}^{n
(\frac{1}{p}-1)}(\rn)\subset
(WH_L^p(\rn))^*.$$

For all $\epsilon\in(0,\,\fz)$ and
$M>\frac{n}{2}(\frac{1}{p}-\frac{1}{2})$,
let $\mathcal{M}_{n(\frac{1}{p}-1),L}^{\ez,M}(\rn)$
be the space defined as in \eqref{4.xx1}.  For all $g\in W\Lambda_{L,{\mathcal{N}}}^{n
(\frac{1}{p}-1)}(\rn)$, since $$g\in
\bigcap_{\ez\in(0,\,\fz)}(\mathcal{M}_{n(\frac{1}{p}-1),L}^{\ez,M}(\rn))^*,$$
from the fact that all $(p,\,2,\,M)_L$-atoms $a$ belong to
$\cup_{\ez\in(0,\,\fz)}\mathcal{M}_{n(\frac{1}{p}-1),L}^{\ez,M}(\rn)$,
it follows that, for all $(p,\, 2,\,M)_L$-atom $a$, $\langle
g,\,a\rangle$ is well defined.
Moreover, for all $f\in
WH_L^p(\rn)\cap L^2(\rn)$, by Remark \ref{r2.18},
we see that $f$ has a weak atomic
$(p,\,2,\,M)_L$-representation
$\sum_{i\in\mathcal{I},j\in\Lambda_i}\wz \lz_{i,j}a_{i,j}$
such that $\{a_{i,j}\}_{i\in\mathcal{I},j\in{\Lambda_i}}$ is a sequence of
$(p,\,2,\,M)_T$-atoms associated to the balls $\{{B}_{i,j}\}_{i\in\mathcal{I},j\in{\Lambda_i}}$
and $\{\wz\lz_{i,j}\}_{i\in\mathcal{I},j\in{\Lambda_i}}
=\{\wz C2^i|{B}_{i,j}|^{\frac{1}{p}}\}_{i\in\mathcal{I},j\in{\Lambda_i}}$,
with $\wz C$ being a positive constant independent of $f$, satisfies
\begin{eqnarray}\label{4.xx12}
\sup_{i\in\mathcal{I}}
\lf(\sum_{j\in{\Lambda_i}}|\wz\lz_{i,j}|^{p}\r)^{{1}/{p}}\ls
\|f\|_{WH_L^p(\rn)}.
\end{eqnarray}
By the definition of $WH_L^p(\rn)$,
we further know that $t^{2}Le^{-t^{2}L}f\in WT^p(\rr^{n+1}_+)\cap T^2(\rr^{n+1}_+)$.
This, together with Theorem \ref{t2.12} and the proof of
Theorem \ref{t2.16} (see \eqref{2.x15} and \eqref{2.x16}), implies
that there exist sequences
$\{{A}_{i,j}\}_{i\in\mathcal{I}, j\in \Lambda_i}$ of
$T^p(\rr^{n+1}_+)$-atoms, associated with balls
$\{B_{i,j}\}_{i\in\mathcal{I}, j\in \Lambda_i}$,
and $\{{\lz}_{i,j}\}_{i\in\mathcal{I}, j\in\Lambda_i}\subset\cc$ satisfying
$\lz_{i,j}=\wz {\wz C}2^i|{B}_{i,j}|^{\frac{1}{p}}$, with $\wz {\wz C}$ being a positive constant independent of $f$,
such that
$$t^{2}Le^{-t^{2}L}f=\dsum_{i\in\mathcal{I}}
\dsum_{j\in \Lambda_i}{\lz}_{i,j} {A}_{i,j}$$
holds true pointwisely almost everywhere in $\rr^{n+1}_+$
and in $T^2(\rr^{n+1}_+)$, where $\mathcal{I}$ and $\Lambda_i$
are as in Theorem \ref{t2.12}. By \eqref{4.xx12} and the definitions
of $\lz_{i,\,j}$ and $\wz\lz_{i,\,j}$, we further obtain
\begin{eqnarray}\label{4.x14}
\sup_{i\in\mathcal{I}}
\lf(\sum_{j\in{\zz_+}}|\lz_{i,j}|^{p}\r)^{\frac{1}{p}}\ls
\|f\|_{WH_L^p(\rn)}.
\end{eqnarray}

Now, for all $i\in\mathcal{I}$, let
$\boz_i:= \cup_{j\in\Lambda_i} B_{i,j}$.
From the proof of Theorem \ref{t2.12} (see the argument below \eqref{2.x8}),
it follows that $\{B_{i,\,j}\}_{j\in\Lambda_i}
\in \wz{\mathcal {W}}_{\boz_i}$. Moreover, by comparing the
quasi-norms between $W\Lambda_{L, {\mathcal{N}}}^{n(\frac{1}{p}-1)}(\rn)$ and
$\Lambda_{L}^{n(\frac{1}{p}-1)}(\rn)$, we see that
$W\Lambda_{L, {\mathcal{N}}}^{n(\frac{1}{p}-1)}(\rn)\subset
\Lambda_{L}^{n(\frac{1}{p}-1)}(\rn)$, where
$\Lambda_{L}^{n(\frac{1}{p}-1)}(\rn)$ denotes the \emph{``strong"}
\emph{Lipschitz space associated to} $L$ defined as in
\cite[(1.26)]{hmm}. This, together with the Calder\'on reproducing
formula \cite[Lemma 8.4]{hm09}, Theorem \ref{t2.6}, Remark
\ref{r2.10}, Theorem \ref{t2.12} and H\"older's
inequality, implies that
\begin{eqnarray*}
&&\lf|\langle g,\,f\rangle\r|\\
&&\hs\sim\lf|\iint_{\rr^{n+1}_+}\lf( t^{2}L\r)^M
\Phi(t\sqrt{L})g(x)\overline{t^{2}Le^{-t^{2}L}f(x)}
\,\frac{dx\,dt}{t}\r|\\
&&\hs\ls\iint_{\rr^{n+1}_+}\lf|\lf( t^{2}L\r)^M
\Phi(t\sqrt{L})g(x)  \lf[\dsum_{i\in\mathcal{I}}\dsum_{j\in \Lambda_i} 2^i
|B_{i,j}|^{\frac{1}{p}}\overline
{A_{i,j}(x,\,t)}\r]\r| \,\frac{dx\,dt}{t}\\
&&\hs\ls \dsum_{i\in\mathcal{I}}2^i
\dsum_{j\in\Lambda_i} \lf|B_{i,j}\r|^{\frac{1}{p}}
\lf[\iint_{\widehat{B}_{i,j}\cap (\supp{A_{i,\,j})}}
\lf|G_{L}(t,x)\r|^2\,\frac{dx\,dt}{t}\r]^{\frac{1}{2}}
\lf[\iint_{ \widehat{B}_{i,j}}\lf|{A}_{i,{j}}(x,t)\r|^2\,
\frac{dx\,dt}{t}\r]^{\frac{1}{2}}\\&&\hs=:\rm{I},
\end{eqnarray*}
where,  for all $t\in(0,\,\fz)$ and $x\in\rn$,
we let
$$G_{L}(t,x):=(t^{2}L)^M\Phi(t\sqrt{L})g(x).$$

To estimate $\mathrm{I}$,
by Proposition \ref{p3.7}, H\"older's inequality,
the definition of $\lz_{i,\,j}$ and \eqref{4.x14},  we conclude that
\begin{eqnarray*}
\rm{I}&&\ls \dsum_{i\in\mathcal{I}}2^{i}\lf|\boz_{i}\r|^{\frac{1}{p}-\frac{1}{2}}
\lf\{\dsum_{j\in\Lambda_i} \lf|B_{i,j}\r|^{\frac{1}{2}}
\lf[\frac{1}{|\boz_i|^{\frac{2}{p}-1}} \iint_{\widehat{B}_{i,j}\cap (\supp A_{i,\,j})}\lf|G_{L}(t,x)\r|^2\,
\frac{dx\,dt}{t}\r]^{\frac{1}{2}}\r\}\\
&&\ls\dsum_{i\in\mathcal{I}}2^{i}\lf|\boz_{i}\r|^{\frac{1}{p}-\frac{1}{2}}
\lf[\dsum_{j\in\Lambda_i} \lf|B_{i,j}\r|\r]^{\frac{1}{2}}
\lf[\dsum_{j\in\Lambda_i}\frac{1}{|\boz_i|^{\frac{2}{p}-1}} \iint_{\widehat{B}_{i,j}\cap
(\supp A_{i,\,j})}\lf|G_{L}(t,x)\r|^2\,
\frac{dx\,dt}{t} \r]^{\frac{1}{2}}\\
&&\ls\dsum_{i\in\mathcal{I}}2^i|\boz_i|^{\frac{1}{p}}
\lf[\frac{1}{|\boz_i|^{\frac{2}{p}-1}} \iint_{\cup_{j\in\Lambda_i}\widehat{B}_{i,j}\cap
(\supp A_{i,\,j})}\lf|G_{L}(t,x)\r|^2\,
\frac{dx\,dt}{t}\r]^{\frac{1}{2}}\\
&&\ls\|f\|_{WH_L^p(\rn)} \dsum_{i\in
\mathcal{I}}\omega_{\mathcal{C}}(|\boz_i|).
\end{eqnarray*}
From this, together with Remark \ref{r2.10},
the fact that $\omega_{\mathcal{C}}\ls\omega_{\mathcal{N}}$ and the decreasing property of
$\omega_\mathcal{N}$, where $\omega_\mathcal{N}$ and $\omega_{\mathcal{C}}$ are, respectively,
as in \eqref{3.6} and \eqref{3.9}, we know that
\begin{eqnarray*}
 \rm{I}&&\ls\|f\|_{WH_L^p(\rn)}
\dsum_{\wz{i}_j\in \wz{\mathcal{I}}}
\omega_{{\mathcal{N}}}\lf(2^{-\wz{i}_j}|\boz_{{i}_0}|\r)\\
&&\ls\|f\|_{WH_L^p(\rn)} \dsum_{i\in
\zz}\omega_{\mathcal{N}}(2^{-i}|\boz_{{i}_0}|)\ls\|f\|_{WH_L^p(\rn)}
\dint_0^\fz\frac{\omega_{\mathcal{N}}(\dz)}{\dz}\,d\dz\\&&\sim\|f\|_{WH_L^p(\rn)}\|g\|_{W
\Lambda_{L,\mathcal{N}}^{n(\frac{1}{p}-1)}(\rn)},
\end{eqnarray*}
where both $\wz{i}_j$ and $i_0$ are as in Remark \ref{r2.10},
which, combined with a density argument, implies that
$$W\Lambda_{L, \mathcal{N}}^{n(\frac{1}{p}-1)}(\rn)\subset
\lf(WH_L^p(\rn)\r)^*.$$

Now, we prove the inclusion that $(WH_L^p(\rn))^*\subset
W\Lambda_{L, \mathcal{N}}^{n(\frac{1}{p}-1)}(\rn)$.

Let $g\in(WH_L^p(\rn))^*$. For all $(p,\,\ez,\,M)_L$-molecules
$m$, from the fact that
$\|m\|_{WH_L^p(\rn)}\ls1$, it follows that
$|g (m)|\ls 1$. By this and the fact that, for all
$\ez\in(0,\,\fz)$, $M\in\nn$ and
$m_0\in\mathcal{M}_{\az,L}^{\ez,M}(\rn)$, $m_0$ is a
$(p,\,\ez,\,M)$-molecule, we conclude that $g\in
\mathcal{M}_{\az,L}^{M,*}(\rn)$.

Now, we prove that
$\|g\|_{W\Lambda^{n(\frac{1}{p}-1)}_{L,\mathcal{N}}
(\rn)}\ls\|g\|_{(WH_L^p(\rn))^*}$. By Definition \ref{d3.2},
\eqref{3.5} and \eqref{3.6}, we first write
\begin{eqnarray}\label{3.10}
\dint_0^\fz\frac{\omega_{\mathcal{N}}{(\dz)}}{\dz}\,d\dz
&&\ls\dsum_{l\in\zz}\omega_{\mathcal{N}}
(2^{-l})\\ \nonumber&&\ls \dsum_{l\in\zz}
\dsup_{|\boz|=2^{-l}} \dsup_{i\in\zz_+} \dsup_{\vec{B}\in \mathcal{W}_{\boz}}
2^{-i\mathcal{N}}\lf[\frac{1}{|\boz|^{\frac{2}{p}-1}}\dint_{S_i(\vec{B})}
\lf|\mathcal{A}_{r}g(x)\r|^2\,dx\r]^{\frac{1}{2}}\\ \nonumber&&=: \dsum_{l\in\zz}
\dsup_{|\boz|=2^{-l}} \dsup_{i\in\zz_+} \dsup_{\vec{B}\in \mathcal{W}_{\boz}}
\mathrm{A}_l,
\end{eqnarray}
where $\mathcal{A}_{r}$ and $S_i(\vec{B})$ are, respectively, as in \eqref{3.4} and
\eqref{4.xx}.

To estimate $\mathrm{A}_l$, from  the dual norm of $L^2(\rn)$,
we deduce that there exists $\fai_l \in L^2(\rn)$ satisfying $\|\fai_l\|_{L^2(\rn)}\le 1$
such that
\begin{equation}\label{4.x13}
\mathrm{A}_l\sim
\lf|\lf \langle g,\, 2^{-i\mathcal{N}} \frac{1}{|\boz|^{\frac{1}{p}-\frac{1}{2}}}
\mathcal{A}_r\lf(\chi_{S_i(\vec{B})}\fai_l\r) \r\rangle_{L^2(\rn)}\r|
=:\lf|\lf \langle g,\, f_l \r\rangle_{L^2(\rn)}\r|.
\end{equation}

We now estimate the $WH_L^p(\rn)$ quasi-norm of $f_l$. For all $\az\in(0,\,\fz)$, by
Chebyshev's inequality, H\"older's inequality, the definition of $S_L$ and the fact that,
for all $\ell\in\zz_+$, $|S_\ell (\vec{B})|\ls 2^{\ell n}|\boz|\sim 2^{\ell n}2^{-l}$, we obtain
\begin{eqnarray}\label{3.x11}
\hs\quad&&\az^p\lf|\lf\{x\in\rn:\ S_L(f_l)(x)>\az\r\}\r|\\
&&\nonumber\hs\ls \dint_\rn \lf|S_L(f_l)(x)\r|^p\,dx
\ls\dsum_{\ell\in\zz_+}\lf\{\dint_{S_\ell(\vec{B})} \lf|S_L(f_l)(x)\r|^2
\,dx\r\}^{\frac{p}{2}}|S_\ell(\vec{B})|^{1-\frac{p}{2}}\\
&&\nonumber\hs\ls\dsum_{\ell\in \zz_+} 2^{-i\mathcal{N}p} 2^{\ell n(1-\frac{p}{2})}
\lf\{\lf(\dint_{S_{\ell}(\vec{B})}\lf[\dint_0^r+\dint_r^{2^{\ell-3}r}+\dint_{2^{\ell-3}r}^\fz\r]
\dint_{\{y\in\rn:\ |y-x|<t\}}\r.\r.\\
&&\nonumber\lf.\hs\hs\times
\lf|t^2Le^{-t^2L}\mathcal{A}_r\lf(\chi_{S_i (\vec{B})}\fai_l\Bigg)(y)\r|^2\,\frac{dy\,dt\,dx}{t^{n+1}}
\r)^{\frac{1}{2}}\r\}^p=:\mathrm{Q}_1+\mathrm{Q}_2+\mathrm{Q}_3.
\end{eqnarray}

We first estimate $\mathrm{Q}_1$.
For all $\ell\in\zz_+$, let
\begin{eqnarray}\label{3.x13}
\mathrm{E}_{\ell}:=\lf\{x\in\rn:
\dist(S_{\ell} (\vec{B}),\, x)<r\r\}.
\end{eqnarray}
It is easy to see that, for all $\ell>i+1$,
\begin{eqnarray*}
\frac{r}{\dist(E_{\ell},\, S_{i}(\vec{B}))}\ls 2^{-\ell},
\end{eqnarray*}
where $r$ is as in Definition \ref{d3.1}(ii).
Thus, by Fubini's theorem, the quadratic estimates and
the Davies-Gaffney estimates, we know that there exists a positive constant $\az_6$, independent of $i$, $l$ and $\ell$,
such that
\begin{eqnarray*}
\mathrm{Q}_1&&\ls\dsum_{\ell\in\zz_+}
2^{-i\mathcal{N}p}2^{\ell n(1-\frac{p}{2})}
\lf\{\lf[ \dint_0^{r}
\dint_{\mathrm{E}_\ell}\lf|t^{2}Le^{-t^{2}L} \mathcal{A}_{r}\lf(\chi_{S_i(\vec{B})} \fai_{l}\r)(y)\r|^2\,
\frac{dy\,dt}
{t}\r]^{\frac{1}{2}}\r\}^p\\
&&\ls \dsum_{\ell=0}^{i+1} 2^{-i\mathcal{N}p}
2^{\ell n(1-\frac{p}{2})}
\|\fai_l\|^p_{L^2(S_i(\vec{B}))}\\
&&\nonumber\hs+\dsum_{\ell=i+2}^\fz 2^{-i\mathcal{N}p}2^{\ell n(1-\frac{p}{2})}
\lf[\dint_0^{r} \exp\lf\{-\frac{[\dist(E_l,\,S_i(\vec{B}))]^2}{t^2}\r\}\,\frac{dt}
{t} \r]^{\frac{p}{2}}\|\fai_l\|_{L^2(S_i(\vec{B}))}^p\\ \nonumber
&&\ls  \dsum_{\ell=0}^{i+1} 2^{-i\mathcal{N}p}
2^{\ell n(1-\frac{p}{2})}
\|\fai_l\|^p_{L^2(S_i(\vec{B}))}\\
&&\nonumber\hs+\dsum_{\ell=i+2}^\fz 2^{-i\mathcal{N}p}2^{\ell n(1-\frac{p}{2})}
2^{-\ell (\mathcal{N}+\az_6)}\|\fai_l\|_{L^2(S_i(\vec{B}))}^p
\ls1.\nonumber
\end{eqnarray*}

To estimate $\mathrm{Q}_2$, for all $\ell\in\zz_+$, let
\begin{eqnarray*}
F_\ell:=\lf\{x\in\rn:\ \dist(x,\,S_\ell(\vec{B}))<2^{\ell-3}r\r\}.
\end{eqnarray*}
By Fubini's theorem, the Davies-Gaffney estimates and the quadratic estimates, we know that there exists a positive constant $\az_7$, independent of $i$, $l$ and $\ell$, such that
\begin{eqnarray*}
\mathrm{Q}_2&&\ls\dsum_{\ell\in\zz_+}2^{\ell n(1-\frac{p}{2})}
2^{-i\mathcal{N}p}
\lf\{\lf[\dint_r^{2^{\ell-3}r}
\dint_{F_\ell} \lf|t^2Le^{-t^2L}\mathcal{A}_r\lf(\chi_{S_i(\vec{B})}\fai_l\r)(y)\r|^2\,
\frac{dy\,dt}{t}\r]^{\frac{1}{2}}\r\}^p\\
&&\ls\dsum_{\ell\in\zz_+}2^{\ell n(1-\frac{p}{2})}
2^{-i\mathcal{N}p}
\lf\{\lf[\dint_r^{\fz}\dint_{F_\ell} \lf|\lf(t^2L\r)^{M+1}
e^{-t^2L} \lf(r^2L\r)^{-M}\mathcal{A}_r\lf(\chi_{S_i(\vec{B})}\fai_l\r)(y)\r|^2\,\r.\r.\\
&&\lf.\lf.\hs\times\frac{dy\,dt}{t^{4M+1}}\r]^{\frac{1}{2}}r^{2M}\r\}^p\\
&&\ls\dsum_{\ell=0}^{i+1}2^{\ell n(1-\frac{p}{2})}
2^{-i\mathcal{N}p}
\lf\{\lf[\dint_r^{\fz}\frac{dt}{t^{4M+1}}\r]^{\frac{1}{2}}r^{2M}\|\chi_{S_i(\vec{B})}\fai_l\|_{L^2(\rn)}\r\}^p
\\&&\hs+\dsum_{\ell =i+2}^{\fz}2^{\ell n(1-\frac{p}{2})}
2^{-i\mathcal{N}p} \lf\{\lf[\dint_r^{\fz}
\frac{dt}{t^{4M+1}}\r]^{\frac{1}{2}}r^{2M} \exp\lf\{-C\frac{[\dist(F_\ell,\, S_i(\vec{B}))]^2}
{t^2}\r\}\r.\\&&\hs\times
\|\chi_{S_i(\vec{B})}\fai_{l}\|_{L^2(\rn)}\Bigg\}^p\\
&&\ls \dsum_{\ell=0}^{i+1} 2^{-i\mathcal{N}p} 2^{\ell n(1-\frac{p}{2})}\|\chi_{S_i(\vec{B})}\fai_l\|_{L^2(\rn)}^p
+ \dsum_{\ell=0}^{i+1} 2^{-i\mathcal{N}p} 2^{\ell n(1-\frac{p}{2})}
2^{-\ell(\mathcal{N}+\az_7)}\|\chi_{S_i(\vec{B})}\fai_l\|_{L^2(\rn)}^p\\
&&\ls 1.
\end{eqnarray*}

Similar to the estimates of $\mathrm{Q}_2$, we also obtain
$\mathrm{Q}_3\ls 1.$ Thus, by \eqref{3.x11} and the estimates of $\mathrm{Q}_1$, $\mathrm{Q}_2$
and $\mathrm{Q}_3$, we see that $$\az^p|\lf\{x\in\rn:\ S_L(f_l)(x)>\az\r\}|\ls 1,$$
where the implicit positive constant is independent of $l$. Thus, $\{f_l\}_{l\in\zz}$ are uniformly bounded in
$WH_L^p(\rn)$, which implies that, for all $N\in \nn$,
$\sum_{l=-N}^{N}f_l\in WH_L^p(\rn).$
By choosing $\wz \az\in(0,\,\fz)$ satisfying
\begin{eqnarray*}
\lf\|\dsum_{l=-N}^N f_l\r\|^p_{WH_L^p(\rn)}\sim \wz \az^p\lf|\lf\{x\in\rn:\
S_L\lf(\dsum_{l=-N}^{N}f\r)(x)>\wz \az\r\}\r|
\end{eqnarray*}
and $l_0\in\nn$ satisfying $2^{l_0}\le \wz \az^p <2^{l_0+1}$, we write
\begin{eqnarray*}
\dsum_{l=-N}^{N}f_l=\dsum_{l=-N}^{l_0}f_l+\dsum_{l=l_0+1}^{N}\cdots=:f_{1}+f_{2}.
\end{eqnarray*}
Here, without loss of generality, we may assume that $l_0\le N$; otherwise, we only need to
estimate $f_{1}$.

To estimate $f_{1}$, let $q\in(p,\,2)$. By Chebyshev's inequality, Minkowski's inequality, H\"older's inequality and the definition of $f_l$, we know that
\begin{eqnarray*}
&&\lf|\lf\{x\in\rn:\
S_L(f_{1})(x)>{\wz\az}\r\}\r|\\
&&\hs\ls{\wz\az}^{-q} \lf(\lf[\dint_{\rn}\lf|S_L(f_{1})(x)\r|^q\,dx\r]^{\frac{1}{q}}\r)^q
\ls{\wz\az}^{-q} \lf(\dsum_{l=-N}^{l_0}
\lf[\dint_{\rn}\lf|S_L(f_{l})(x)\r|^q\,dx\r]^{\frac{1}{q}}\r)^q\\
&&\hs\ls {\wz\az}^{-q} \lf\{\dsum_{l=-N}^{l_0}
\dsum_{\ell\in\zz_+}\lf[\dint_{S_\ell(\vec{B})} \lf|S_L(f_{l})(x)\r|^2\,dx\r]^{\frac{1}{2}}
\lf|S_{\ell}(\vec{B})\r|^{\frac{1}{q}-\frac{1}{2}}\r\}^q\\
&&\hs\ls {\wz\az}^{-q} \lf[\dsum_{l=-N}^{l_0}
\dsum_{\ell\in\zz_+} 2^{\ell n(\frac{1}{q}-\frac{1}{2})} |\boz|^{(\frac{1}{2}-\frac{1}{p})+
(\frac{1}{q}-\frac{1}{2})} 2^{-i\mathcal{N}}\lf\{\dint_{S_{\ell}(\vec{B})}
\lf[\dint_{0}^r +\dint_{r}^{2^{\ell-3}r}+\dint_{2^{\ell-3}r}^\fz\r]\r.\r.\\
&&\lf.\lf.\hs\hs\times
\dint_{\{y\in\rn:\ |y-x|<t\}}\lf|t^2Le^{-t^2L}\mathcal{A}_r\lf(\chi_{S_i(\vec{B})}\fai_l\r)(y)\r|^{2}
\,\frac{dy\,dt\,dx}{t^{n+1}}\r\}^{\frac{1}{2}}\r]^q
=:\wz{\mathrm{Q}}_1+\wz{\mathrm{Q}}_2+\wz{\mathrm{Q}}_3.
\end{eqnarray*}

We first estimate $\wz{\mathrm{Q}}_1$. For all $\ell\in\zz_+$, let $E_\ell$ be  as in \eqref{3.x13}.
It is easy to see that, for all $i\in\nn$ and $\ell>i+1$,
\begin{eqnarray*}
\frac{r}{\dist(E_\ell,\,S_i(\vec{B}))}\ls 2^{-\ell}.
\end{eqnarray*}
Thus, by the fact $|\boz|\sim 2^{-l}$,
Fubini's theorem, the Davies-Gaffney estimates, the quadratic estimates,
$q\in(p,\,2)$ and the assumption $2^{l_0}\sim \wz \az^p $, we conclude that
 \begin{eqnarray*}
 \wz{\mathrm{Q}}_1&&\ls \wz\az^{-q} \lf[\dsum_{l=-N}^{l_0}\dsum_{\ell=0}^{i+1}2^{\ell n(\frac{1}{q}-
 \frac{1}{2})}2^{l(\frac{1}{p}-\frac{1}{q})}2^{-i\mathcal{N}}\|\fai_l\|_{L^2(S_i(\vec{B}))}\r]^q\\
 &&\hs+\wz\az^{-q} \lf[\dsum_{l=-N}^{l_0}\dsum_{\ell=i+2}^{\fz} 2^{\ell n(\frac{1}{q}-
 \frac{1}{2})}2^{l(\frac{1}{p}-\frac{1}{q})}2^{-i\mathcal{N}}
\dint_0^r\exp\lf\{-C\frac{[\dist(E_\ell,\,S_i(\vec{B}))]^2}
{t^2}\r\} \frac{dt}{t}\r.\\
&&\hs\times\|\fai_l\|_{L^2(S_i(\vec{B}))}\Bigg]^q\\
&&\ls \wz\az^{-q} \lf[\dsum_{l=-\fz}^{l_0}2^{l(\frac{1}{p}-\frac{1}{q})}\r]^q+\wz{\az}^{-q}
2^{l_0}\lf[\dsum_{l=-\fz}^{l_0}2^{l(\frac{1}{p}-\frac{1}{q})}\r]^{q}
\sim \wz\az^{-q} 2^{l_0(\frac{q}{p}-1)}\sim \wz\az^{-p}.
 \end{eqnarray*}
Similar to the estimates of $\mathrm{Q}_1$ and $\mathrm{Q}_2$,
we have $\wz{\mathrm{Q}}_2+ \wz{\mathrm{Q}}_3\ls \wz \az^{-p}.$
Thus,
\begin{eqnarray*}
\wz\az^p\lf|\lf\{x\in\rn:\ S_L(f_{1})(x)>\wz\az\r\}\r|\ls \wz\az^{p-q} 2^{l_0(\frac{1}{p}-\frac{1}{q})q}
\sim1.
\end{eqnarray*}

On the other hand, to estimate $f_2$, let $\wz q\in(0,\,p)$. Then
\begin{eqnarray*}
\mathrm{J}:=&&\lf|\lf\{x\in\rn:\ S_L\lf(\dsum_{l=l_0+1}^N f_l\r)(x)>\wz\az\r\}\r|\\
\ls&& \wz\az^{-q}
\lf[\dsum_{l=-N}^{l_0} \dsum_{\ell\in\zz_+} \lf\{\dint_{S_\ell(\vec{B})}
\lf[S_L(f_l)(x)\r]^2\,dx\r\}^{\frac{1}{2}}\lf|S_\ell(\vec{B})\r|^{\frac{1}{q}-\frac{1}{2}}\r]^{q}.
\end{eqnarray*}
Similar to the estimates of $f_1$, we also obtain
\begin{eqnarray*}
\mathrm{J}\ls \wz\az^{-q}2^{l(\frac{1}{p}-\frac{1}{q})\epsilon},
\end{eqnarray*}
which implies that
$\wz\az^p|\{x\in\rn:\ S_L(f_2)(x)>\wz\az\}|\ls 1$. Combining the estimates of $f_1$ and $f_2$,
we conclude that
\begin{eqnarray*}
\lf\|\dsum_{l=-N}^N f_l\r\|_{WH_L^p(\rn)}\sim \wz\az^p\lf|\lf\{x\in\rn:\ S_L\lf(\dsum_{l=-N}^N
f_l\r)(x)>\wz\az\r\}\r|\ls 1,
\end{eqnarray*}
where the implicit constants are independent of $N$. Thus, by letting $N\to \fz$, we obtain
\begin{eqnarray*}
\lf\|\dsum_{l=-\fz}^{\fz}f_l\r\|_{WH_L^p(\rn)}\ls 1.
\end{eqnarray*}
Thus, from \eqref{3.10}, \eqref{4.x13} and the assumption that $g\in(WH_L^p(\rn))^*$,
we deduce that
\begin{eqnarray*}
\|g\|_{W\Lambda_{L,\,\mathcal{N}}^{n(\frac{1}{p}-1)}}\sim
\dint_0^\fz \frac{w_\mathcal{N}(\dz)}{\dz}\,d\dz \ls
\lim_{N\to \fz}
\lf|\lf\langle g,\,\dsum_{l=-N}^{N}
 f_l \r\rangle_{L^2(\rn)}\r|\ls \|g\|_{(WH_L^p(\rn))^*},
\end{eqnarray*}
which implies that $g\in W\Lambda_{L,\,\mathcal{N}}^{n(\frac{1}{p}-1)}$ and
\begin{eqnarray*}
\|g\|_{W\Lambda_{L,\,\mathcal{N}}^{n(\frac{1}{p}-1)}}\ls \|g\|_{(WH_L^p(\rn))^*}.
\end{eqnarray*}
This shows that $(WH_L^p(\rn))^*\subset W\Lambda_{L,\,\mathcal{N}}^{n(\frac{1}{p}-1)}$
and hence finishes the proof of Theorem \ref{t3.6}.
\end{proof}

\begin{remark}\label{r3.8}
By Theorem \ref{t3.6}, we see that
$W\Lambda_{L, {\mathcal{N}}}^{\az}(\rn)$ for $\az\in[0,\,\fz)$
is independent of the choice of ${\mathcal{N}}\in(n(\frac{1}{p}-\frac{1}{2}),\,\fz)$.
Thus, we can write $W\Lambda_{L, {\mathcal{N}}}^{\az}(\rn)$ simply by
$W\Lambda_{L}^{\az}(\rn)$.
\end{remark}

\smallskip

\thanks{\bf Acknowledgment.}
Part of the results in this paper was announced by the second author during the International Workshop on Nonlinear Variational Analysis 2014 which
was held on August 1st to 3rd at Kaohsiung Medical University. The second author would like to express his profound gratitude
to Professor Jen-Chih Yao for his invitation and for the warm hospitality extended to him during his stay in Taiwan.
The authors would like to thank Dr.
Dongyong Yang for his careful reading and many helpful
discussions on the topic of this article.

\bigskip

\noindent Jun Cao and Dachun Yang (Corresponding author)

\medskip

\noindent School of Mathematical Sciences, Beijing Normal
University, Laboratory of Mathematics and Complex Systems, Ministry
of Education, Beijing 100875, People's Republic of China

\smallskip

\noindent{\it E-mails:} \texttt{caojun1860@mail.bnu.edu.cn} (J. Cao)

\hspace{1.02cm}\texttt{dcyang@bnu.edu.cn} (D. Yang)

\bigskip

\noindent \emph{Current address} of Jun Cao

\medskip

\noindent Department of Applied Mathematics, Zhejiang University of Technology,
Hangzhou 310032, People's Republic of China

\bigskip

\noindent Der-Chen Chang

\medskip

\noindent Department of Mathematics and Department of Computer
Science, Georgetown University, Washington D. C. 20057, U. S. A.

\smallskip

\noindent{\it E-mail:} \texttt{chang@georgetown.edu}

\bigskip

\noindent Huoxiong Wu

\medskip

\noindent School of Mathematical Sciences, Xiamen University, Xiamen
361005, People¡¯s Republic of China

\smallskip

\noindent{\it E-mail:} \texttt{huoxwu@xmu.edu.cn}

\end{document}